\documentclass[a4paper]{amsart}
\usepackage{amscd}
\usepackage{amsmath}
\usepackage{amssymb}
\usepackage{amsthm}
\usepackage{bbm}
\usepackage{stmaryrd}
\usepackage[arrow, matrix, curve]{xy}

\usepackage[T1]{fontenc}

\newif\ifpdf
\ifx\pdfoutput\undefined
   \pdffalse        
\else
   \pdfoutput=1     
   \pdftrue
\fi

\ifpdf
   \usepackage[pdftex]{graphicx}
\else
   \usepackage{graphicx}
\fi

\numberwithin{equation}{section} \swapnumbers

\newtheorem{satz}{Satz}[section]

\newtheorem{theorem}[satz]{Theorem}
\newtheorem{proposition}[satz]{Proposition}

\newtheorem{lemma}[satz]{Lemma}

\newtheorem{definition}[satz]{Definition}

\newtheorem{remark}[satz]{Remark}

\newtheorem{example}[satz]{Example}

\begin{document}

\hyphenation{ma-ni-folds Swit-zer-land}

\title[Invariant manifolds with boundary for jump-diffusions]{Invariant manifolds with boundary for jump-diffusions}
\author{Damir Filipovi\'c \and Stefan Tappe \and Josef Teichmann}
\date{20 June 2014}

\address{EPFL and Swiss Finance Institute, Quartier UNIL-Dorigny, Extranef 218, CH-1015 Lausanne, Switzerland}
\email{damir.filipovic@epfl.ch}
\address{Leibniz Universit\"{a}t Hannover, Institut f\"{u}r Mathematische Stochastik, Welfengarten 1, D-30167 Hannover, Germany}
\email{tappe@stochastik.uni-hannover.de}
\address{ETH Z\"urich, Department of Mathematics, R\"amistrasse 101, CH-8092 Z\"urich, Switzerland}
\email{josef.teichmann@math.ethz.ch}
\begin{abstract}
We provide necessary and sufficient conditions for stochastic invariance of finite dimensional submanifolds with boundary in Hilbert spaces for stochastic partial differential equations driven by Wiener processes and Poisson random measures.
\end{abstract}
\keywords{Stochastic partial differential equation,
submanifold with boundary, stochastic invariance, jump-diffusion}
\subjclass[2010]{60H15, 60G17}
\maketitle

\section{Introduction}

Consider a stochastic partial differential equation (SPDE) of the form
\begin{align}\label{SPDE}
\left\{
\begin{array}{rcl}
dr_t & = & (Ar_t + \alpha(r_t))dt + \sigma(r_t) dW_t
+ \int_E \gamma(r_{t-},x) (\mu(dt,dx) - F(dx) dt)
\medskip
\\ r_0 & = & h_0
\end{array}
\right.
\end{align}
on a separable Hilbert space $H$ driven by some trace class Wiener process $W$ on a separable Hilbert space $\mathbb{H}$ and a compensated
Poisson random measure $\mu$ on some mark space $E$ with $dt \otimes F(dx)$ being its compensator.
Throughout this paper, we assume that $A$ is the generator of a $C_0$-semigroup on $H$ and that the mappings $\alpha$, $\sigma = (\sigma^j)_{j \in \mathbb{N}}$ and $\gamma$ satisfy appropriate regularity conditions.

Given a finite dimensional $C^3$-submanifold $\mathcal{M}$ with
boundary of $H$, we study the
\textit{stochastic viability} and \textit{invariance problem}
related to the SPDE (\ref{SPDE}). In particular, we provide necessary and sufficient conditions such that for each $h_0 \in \mathcal{M}$ there is a (local) mild solution $r$ to (\ref{SPDE}) with $r_0 = h_0$ which stays (locally) on the submanifold~$\mathcal{M}$.

Any finite dimensional invariant submanifold $\mathcal{M}$ for the SPDE \eqref{SPDE} gives rise to a finite dimensional Markovian realization of the respective particular solution processes $r$ with initial values in $\mathcal{M}$, i.e.~a deterministic $C^3$-function $ G $ and a finite dimensional Markov process $ X $ such that $ r_t=G(X_t) $ up to some stopping time. This proves to be useful in applications, since it renders the stochastic evolution model \eqref{SPDE} analytically and numerically tractable for initial values in $\mathcal{M}$. An important example is the so-called Heath-Jarrow-Morton (HJM) SPDE that describes the evolution of the interest rate curve. Stochastic invariance for the HJM SPDE has been discussed in detail in \cite{Bj-Ch,Bj-La,Bj-Sv,Filipovic-note,Filipovic-exponential,Filipovic-Teichmann,
Filipovic-Teichmann-royal, Nakayama} for the diffusion case. The present paper completes the results from \cite{Filipovic-inv, Filipovic-Teichmann, Filipovic-Teichmann-royal} by providing explicit stochastic invariance conditions for the general case of a SPDE with jumps.

Stochastic invariance has been extensively studied also for other sets than manifolds. In finite dimension the general stochastic invariance problem for closed sets has been treated, e.g., in \cite{bqrt:10} in the diffusion case, and in \cite{sim:00} in the case of jump-diffusions. In infinite dimension we mention, e.g., the works of \cite{Nakayama-Support,Nakayama,tz:06}, where stochastic invariance has been established by means of support theorems for diffusion-type SPDEs.

We shall now present and explain the invariance conditions which we derive in this paper. Let us first consider the situation where the jumps in (\ref{SPDE}) are of finite variation. Then the conditions
\begin{align}\label{domain}
&\mathcal{M} \subset \mathcal{D}(A),
\\ \label{sigma-tangent} &\sigma^j(h) \in
\begin{cases}
T_h \mathcal{M}, & h \in \mathcal{M} \setminus \partial \mathcal{M},
\\ T_h \partial \mathcal{M}, & h \in \partial \mathcal{M},
\end{cases}
\quad \text{for all $j \in \mathbb{N}$,}
\\ \label{jumps} & h + \gamma(h,x) \in \overline{\mathcal{M}} \quad \text{for $F$-almost all $x \in E$,} \quad  \text{for all $h \in  \mathcal{M}$,}
\\ \label{intro-drift} &Ah + \alpha(h) - \frac{1}{2} \sum_{j \in \mathbb{N}} D \sigma^j(h)\sigma^j(h)
\\ \notag & \quad - \int_{E} \gamma(h,x) F(dx) \in
\begin{cases}
T_h \mathcal{M}, & h \in \mathcal{M} \setminus \partial \mathcal{M},
\\ (T_h \mathcal{M})_+, & h \in \partial \mathcal{M}
\end{cases}
\end{align}
are necessary and sufficient for stochastic invariance of $\mathcal{M}$ for (\ref{SPDE}).

Condition (\ref{domain}) says that the submanifold $\mathcal{M}$
lies in the domain of the infinitesimal generator $A$. This ensures
that the mapping in (\ref{intro-drift}) is well-defined.
Condition (\ref{sigma-tangent}) means that the
volatilities $h \mapsto \sigma^j(h)$ must be tangential to
$\mathcal{M}$ in its interior and tangential to the
boundary $\partial \mathcal{M}$ at boundary points.
Condition (\ref{jumps}) says that the functions $h \mapsto h + \gamma(h,x)$ map the submanifold $\mathcal{M}$ into its closure $\overline{\mathcal{M}}$.
Condition (\ref{intro-drift}) means that the adjusted drift must be tangential to $\mathcal{M}$ in its interior and additionally inward pointing at boundary points.

In the general situation, where the jumps in (\ref{SPDE}) may be of infinite variation, condition (\ref{intro-drift}) is replaced by the three conditions
\begin{align}\label{manifold-boundary}
&\int_{E} |\langle \eta_h, \gamma(h,x) \rangle| F(dx) < \infty,
\quad h \in \partial \mathcal{M},
\\ \label{alpha-tangent} &Ah + \alpha(h) - \frac{1}{2} \sum_{j \in \mathbb{N}} D \sigma^j(h) \sigma^j(h)
\\ \notag &\quad - \int_{E} \Pi_{(T_h \mathcal{M})^{\perp}} \gamma(h,x) F(dx) \in T_h \mathcal{M}, \quad h \in \mathcal{M},
\\ \label{alpha-tangent-plus} &\langle \eta_h,Ah + \alpha(h) \rangle
- \frac{1}{2} \sum_{j \in \mathbb{N}} \langle \eta_h, D \sigma^j(h) \sigma^j(h) \rangle
\\ \notag &\quad - \int_{E} \langle \eta_h,\gamma(h,x) \rangle F(dx) \geq 0, \quad h \in \partial \mathcal{M},
\end{align}
where $\eta_h$ denotes the inward pointing normal vector to $\partial \mathcal{M}$ at boundary points $h \in \partial \mathcal{M}$.

Condition (\ref{manifold-boundary}) concerns the small jumps of
$r$ at the boundary of the submanifold and means that the discontinuous part of the solution must be of finite variation, unless it is parallel to the boundary $\partial
\mathcal{M}$.
Denoting by $\Pi_K$ the orthogonal projection on a closed subspace $K \subset H$, we decompose
\begin{align*}
\gamma(h,x) &= \Pi_{T_h \mathcal{M}} \gamma(h,x) + \Pi_{(T_h
\mathcal{M})^{\perp}} \gamma(h,x).
\end{align*}
As we will show, condition (\ref{jumps}) implies
\begin{align}\label{jumps-orthogonal}
\int_E \| \Pi_{(T_h \mathcal{M})^{\perp}} \gamma(h,x) \| F(dx) <
\infty, \quad h \in \mathcal{M}.
\end{align}
The essential idea is to perform a second order Taylor expansion for a parametrization around $h$ to obtain
\begin{align*}
\| \Pi_{(T_h \mathcal{M})^{\perp}} \gamma(h,x) \| = \| \gamma(h,x) - \Pi_{T_h \mathcal{M}} \gamma(h,x) \| \leq C \| \gamma(h,x) \|^2
\end{align*}
for some constant $C \geq 0$. By virtue of (\ref{jumps-orthogonal}), the integral in (\ref{alpha-tangent}) exists, and hence, conditions (\ref{alpha-tangent}), (\ref{alpha-tangent-plus}) correspond to (\ref{intro-drift}).

As in previous papers on this subject we are dealing with mild
solutions of SPDEs, i.e.~stochastic processes taking values in a
Hilbert space whose drift characteristic is quite irregular (e.g., not
continuous with respect to the state variables). Therefore, the
arguments to translate stochastic invariance into conditions on the
characteristics are not straightforward. The
arguments to prove our stochastic invariance results can be structured as
follows: First, we show that we can (pre-)localize the problem by
separating big and small jumps. Second, prelocal invariance of
parametrized submanifolds can be pulled back to $ \mathbb{R}^m $ by a
linear projection argument tracing back to
\cite{fillnm}. Both steps require a careful analysis of
jump structures, which leads to the involved invariance conditions.

The remainder of this paper is organized as follows. In Section~\ref{sec-statement-main} we state our main results. In Section~\ref{sec-notation} we provide some notation and auxiliary results about stochastic invariance. In Section~\ref{sec-proof-half-space} we  perform local analysis of the invariance problem on half spaces, in Section~\ref{sec-proof-main-one-chart} we perform local analysis of the invariance problem on submanifolds with boundary, and in Section~\ref{sec-proof-main} we perform global analysis of the invariance problem on submanifolds with boundary and prove our main results. For convenience of the reader, the proofs of some technical auxiliary results are deferred to the appendix \cite{FTT-appendix}.

\section{Statement of the main results}\label{sec-statement-main}

In this section we introduce the necessary terminology and state our main results. We fix a filtered probability space $(\Omega, \mathcal{F}, (\mathcal{F}_t)_{t \geq 0}, \mathbb{P})$ satisfying the usual conditions and let $H$ be a separable Hilbert space.

Let $W$ be a $Q$-Wiener process (see \cite[pages~86, 87]{Da_Prato}) on some separable Hilbert space $\mathbb{H}$, where the covariance operator $Q$ is a trace class operator.

Let $(E,\mathcal{E})$ be a measurable space which we assume to be a
\textit{Blackwell space} (see \cite{Dellacherie,Getoor}). We remark
that every Polish space with its Borel $\sigma$-field is a Blackwell
space. Furthermore, let $\mu$ be a time-homogeneous Poisson random measure on $\mathbb{R}_+ \times E$, see \cite[Definition~II.1.20]{Jacod-Shiryaev}.
Then its compensator is of the form $dt \otimes F(dx)$, where $F$ is
a $\sigma$-finite measure on $(E,\mathcal{E})$.

In \cite{FTT-appendix} we review some basic facts about SPDEs of the type (\ref{SPDE}) and we recall the concepts of (local) strong, weak and mild solutions. In particular, equation (\ref{SPDE}) can be rewritten equivalently
\begin{align}\label{SPDE-j}
\left\{
\begin{array}{rcl}
dr_t & = & (Ar_t + \alpha(r_t))dt + \sum_{j \in \mathbb{N}} \sigma^j(r_t) d\beta_t^j
\\ && + \int_E \gamma(r_{t-},x) (\mu(dt,dx) - F(dx) dt)
\medskip
\\ r_0 & = & h_0,
\end{array}
\right.
\end{align}
where $(\beta^j)_{j \in \mathbb{N}}$ is a sequence of
real-valued independent standard Wiener processes.
We next formulate the concept of stochastic invariance.

\begin{definition}\label{def-invariance}
A non-empty Borel set $B \subset H$ is called \emph{prelocally (locally) invariant} for (\ref{SPDE-j}), if for all $h_0 \in B$ there exists a local mild solution $r = r^{(h_0)}$ to (\ref{SPDE-j}) with lifetime $\tau > 0$ such that up to an evanescent set\footnote[1]{A random set $A \subset \Omega \times \mathbb{R}_+$ is called \emph{evanescent} if the set $\{ \omega \in \Omega : (\omega,t) \in A \text{ for some } t \in \mathbb{R}_+ \}$ is a $\mathbb{P}$-nullset, cf. \cite[1.1.10]{Jacod-Shiryaev}.}
\begin{gather*}
(r^{\tau})_- \in B \text{ and } r^{\tau} \in \overline{B}
\\ \big( r^{\tau} \in B \big).
\end{gather*}
\end{definition}

The following standing assumptions prevail throughout this paper:
\begin{itemize}
\item $A$ generates a $C_0$-semigroup $(S_t)_{t \geq 0}$ on $H$.

\item The mapping $\alpha : H \rightarrow H$ is locally Lipschitz continuous, that is, for each $n \in \mathbb{N}$ there is a constant $L_n \geq 0$ such that
\begin{align}\label{Lipschitz-alpha-st}
\| \alpha(h_1) - \alpha(h_2) \| \leq L_n \| h_1 - h_2 \|, \quad h_1, h_2 \in H \text{ with } \| h_1 \|, \| h_2 \| \leq n.
\end{align}

\item For each $n \in \mathbb{N}$ there exists a sequence $(\kappa_n^j)_{j \in \mathbb{N}} \subset \mathbb{R}_+$ with $\sum_{j \in \mathbb{N}} (\kappa_n^j)^2 < \infty$ such that for all $j \in \mathbb{N}$ the mapping $\sigma^j : H \rightarrow H$ satisfies
\begin{align}\label{Lipschitz-sigma-st}
\| \sigma^j(h_1) - \sigma^j(h_2) \| &\leq \kappa_n^j \| h_1 - h_2 \|, \quad h_1,h_2 \in H \text{ with } \| h_1 \|, \| h_2 \| \leq n,
\\ \label{linear-growth-sigma} \| \sigma^j(h) \| &\leq \kappa_n^j, \quad h \in H \text{ with } \| h \| \leq n.
\end{align}
Consequently, for each $j \in \mathbb{N}$ the mapping $\sigma^j$ is locally Lipschitz continuous.

\item The mapping $\gamma : H \times E \rightarrow H$ is measurable, and for each $n \in \mathbb{N}$ there exists a measurable function $\rho_n : E \rightarrow \mathbb{R}_+$ with
\begin{align}\label{rho-integrable}
\int_E \big( \rho_n(x)^2 \vee \rho_n(x)^4 \big) F(dx) < \infty
\end{align}
such that for all $x \in E$ the mapping $\gamma(\bullet,x) : H \rightarrow H$ satisfies
\begin{align}\label{Lipschitz-gamma-rho}
\| \gamma(h_1,x) - \gamma(h_2,x) \| &\leq \rho_n(x) \| h_1 - h_2 \|, \quad h_1,h_2 \in H \text{ with } \| h_1 \|, \| h_2 \| \leq n,
\\ \label{linear-growth-rho} \| \gamma(h,x) \| &\leq \rho_n(x), \quad h \in H \text{ with } \| h \| \leq n.
\end{align}
Consequently, for each $x \in E$ the mapping $\gamma(\bullet,x)$ is locally Lipschitz continuous.

\item We assume that for each $j \in \mathbb{N}$ the mapping $\sigma^j : H \rightarrow H$ is continuously differentiable, that is
\begin{align}\label{sigma-C1}
\sigma^j \in C^1(H) \quad \text{for all $j \in \mathbb{N}$.}
\end{align}
\end{itemize}
The first four conditions ensure that we may apply the results about SPDEs from~\cite{FTT-appendix}. We furthermore assume that:
\begin{itemize}
\item $\mathcal{M}$ is a finite-dimensional $C^3$-submanifold with boundary of $H$; that is, for all $h \in
\mathcal{M}$ there exist an open neighborhood $U \subset H$ of $h$, an open
set $V \subset \mathbb{R}^m_+ = \mathbb{R}_+ \times \mathbb{R}^{m-1}$ (where $m \in \mathbb{N}$ is the dimension of $\mathcal{M}$) and a map $\phi \in C^3(V;H)$ (which we will call a \emph{parametrization} of $\mathcal{M}$ around $h$ and also denote as $\phi : V \subset \mathbb{R}^m_+ \rightarrow U \cap \mathcal{M}$) such
that
\begin{enumerate}
\item $\phi : V \rightarrow U \cap \mathcal{M}$ is a homeomorphism;
\item $D \phi(y)$ is one to one for all $y \in V$.
\end{enumerate} 
We refer to \cite[Section~3]{FTT-appendix} for further details.
\end{itemize}

\begin{remark}
We impose that $\mathcal{M}$ is of class $C^3$, because this ensures that the coefficients $a$, $(b^j)_{j \in \mathbb{N}}$, $c$ and $\Theta$, $(\Sigma^j)_{j \in \mathbb{N}}$, $\Gamma$ of the SDEs (\ref{SDE-abc}), (\ref{SDE-Theta}), which we will define in (\ref{def-a-rem})--(\ref{def-c-rem}) and (\ref{def-Theta-rem})--(\ref{def-Gamma-rem}), satisfy the regularity conditions (\ref{Lipschitz-alpha-st})--(\ref{linear-growth-sigma}) and (\ref{Lipschitz-gamma-rho})--(\ref{sigma-C1}) as well; see Lemma~\ref{lemma-Lipschitz-pb}.
\end{remark}

\begin{remark}
Similarly, instead of (\ref{rho-integrable}) one would expect the weaker condition
\begin{align}\label{rho-square-int}
\int_E \rho_n(x)^2 F(dx) < \infty.
\end{align}
The reason is that (\ref{rho-integrable}) is required in order to ensure that the above-mentioned coefficients also satisfy the regularity conditions (\ref{Lipschitz-alpha-st})--(\ref{linear-growth-sigma}) and (\ref{Lipschitz-gamma-rho})--(\ref{sigma-C1}), but with (\ref{rho-integrable}) being replaced by (\ref{rho-square-int}); see Lemma~\ref{lemma-Lipschitz-pb}.
\end{remark}

Our first main result now reads as follows.

\begin{theorem}\label{thm-main-1}
The following statements are equivalent:
\begin{enumerate}
\item $\mathcal{M}$ is prelocally invariant for (\ref{SPDE-j}).

\item We have (\ref{domain})--(\ref{jumps}) and (\ref{manifold-boundary})--(\ref{alpha-tangent-plus}).
\end{enumerate}
In either case, $A$ and the mapping in (\ref{alpha-tangent}) are continuous on $\mathcal{M}$, and for each $h_0 \in \mathcal{M}$ there is a local strong solution $r = r^{(h_0)}$ to (\ref{SPDE-j}). Moreover, if instead of (\ref{jumps}) we even have
\begin{align}\label{jumps-no-closure}
h + \gamma(h,x) \in \mathcal{M} \quad \text{for $F$-almost all $x \in E$,} \quad \text{for all $h \in  \mathcal{M}$,}
\end{align}
then $\mathcal{M}$ is locally invariant for (\ref{SPDE-j}).
\end{theorem}

\begin{remark}
It follows from Theorem~\ref{thm-main-1} that (pre-)local invariance of $\mathcal{M}$ is a property which only depends on the parameters $\{ \alpha, \sigma^j, \gamma, F\}$ -- that is, on the law of the solution to (\ref{SPDE-j}). It does not depend on the actual stochastic basis $\{ (\Omega, \mathcal{F}, (\mathcal{F}_t)_{t \geq 0}, \mathbb{P}), W, \mu \}$.
\end{remark}

Note that local invariance of $\mathcal{M}$ does not imply (\ref{jumps-no-closure}), as the following example illustrates:

\begin{example}
Let $H = \mathbb{R}$, $(E,\mathcal{E}) = (\mathbb{R},\mathcal{B}(\mathbb{R}))$, $\mathcal{M} = [0,1)$ and consider the SDE
\begin{align}\label{SDE-jumps}
\left\{
\begin{array}{rcl}
dr_t & = & dt + \int_{\mathbb{R}} \gamma(r_{t-},x) \mu(dt,dx) \medskip
\\ r_0 & = & h_0,
\end{array}
\right.
\end{align}
where the compensator $dt \otimes F(dx)$ of $\mu$ is given by the Dirac measure $F = \delta_1$ concentrated in $1$, and
\begin{align*}
\gamma : \mathbb{R} \times \mathbb{R} \rightarrow \mathbb{R}, \quad \gamma(h,x) = 1 - 2 h.
\end{align*}
Then $\mathcal{M}$ is locally invariant for (\ref{SDE-jumps}). Indeed, let $h_0 \in \mathcal{M}$ be arbitrary. There exists $\epsilon > 0$ with $h_0 + \epsilon < 1$. We define the stopping time $\tau > 0$ as
\begin{align*}
\tau := \inf \{ t \geq 0 : r_t = h_0 + \epsilon \} \wedge \inf \{ t \geq 0 : \mu( [0,t] \times \mathbb{R}) = 1 \}.
\end{align*}
Then we have $(r^{(h_0)})^{\tau} \in \mathcal{M}$ up to an evanescent set, because
\begin{align*}
h + \gamma(h,x) = 1 - h \in \mathcal{M}, \quad h \in (0,1)
\end{align*}
showing that $\mathcal{M}$ is locally invariant for (\ref{SDE-jumps}).
However, the jump condition (\ref{jumps-no-closure}) is not satisfied, because for $h = 0$ we have
\begin{align*}
h + \gamma(h,x) = 1 \notin \mathcal{M}.
\end{align*}
Nevertheless, we see that condition (\ref{jumps}) holds true, because $1 \in \overline{\mathcal{M}}$.
\end{example}

If $\mathcal{M}$ is a closed subset of $H$ and global Lipschitz conditions are satisfied, then we obtain global invariance. This is the content of our second main result, for which we recall the following definition:

\begin{definition}
The semigroup $(S_t)_{t \geq 0}$ is called \emph{pseudo-contractive}, if
\begin{align*}
\| S_t \| \leq e^{\omega t}, \quad t \geq 0
\end{align*}
for some constant $\omega \in \mathbb{R}$.
\end{definition}

Now our second main result reads as follows:

\begin{theorem}\label{thm-main-2}
Assume that the semigroup $(S_t)_{t \geq 0}$ is pseudo-contractive and that conditions (\ref{Lipschitz-alpha-st})--(\ref{linear-growth-rho}) hold globally, i.e. the coefficients $L_n$, $(\kappa_n^j)_{j \in \mathbb{N}}$, $\rho_n$ do not depend on $n \in \mathbb{N}$, and with the right-hand sides of (\ref{linear-growth-sigma}), (\ref{linear-growth-rho}) multiplied by $(1 + \| h \|)$. 
If $\mathcal{M}$ is a closed subset of $H$, then (\ref{domain})--(\ref{jumps}) and (\ref{manifold-boundary})--(\ref{alpha-tangent-plus}) imply that for any $h_0 \in \mathcal{M}$ there exists a unique strong solution $r = r^{(h_0)}$ to (\ref{SPDE-j}) and $r \in \mathcal{M}$ up to an evanescent set.
\end{theorem}

\begin{remark}
Let us comment on the pseudo-contractivity of the semigroup, which we have imposed for Theorem~\ref{thm-main-2}. Together with the global Lipschitz conditions, it ensures existence and uniqueness of mild solutions to the SPDE (\ref{SPDE-j}) with c\`{a}dl\`{a}g sample paths, which we require for the proof. In the general situation, where the semigroup fulfills the estimate
\begin{align*}
\| S_t \| \leq M e^{\omega t}, \quad t \geq 0
\end{align*}
for constants $M \geq 1$ and $\omega \in \mathbb{R}$, the global Lipschitz conditions ensure existence and uniqueness of mild solutions, but it is generally not known whether they have a c\`{a}dl\`{a}g version. However, we remark that in the continuous case $\gamma \equiv 0$ we obtain the existence of continuous mild solutions without the pseudo-contractivity of the semigroup; see, e.g., \cite{Da_Prato}.
\end{remark}

\begin{remark}
Note that we have not imposed the pseudo-contractivity of the semigroup for Theorem~\ref{thm-main-1}. Under the conditions of this result, the existence of locally invariant mild solutions to the SPDE (\ref{SPDE-j}) follows from the existence of locally invariant strong solutions to the finite dimensional SDEs (\ref{SDE-Theta}), (\ref{SDE-abc}), and this does not require assumptions on the semigroup.
\end{remark}

The above two theorems simplify in the case of jumps with finite variation:

\begin{theorem}\label{thm-main-3}
Assume that
\begin{align}\label{gamma-integrable-pre}
\int_E \| \gamma(h,x) \| F(dx) < \infty \quad \text{for all $h \in \mathcal{M}$.}
\end{align}
Then the following statements are true:
\begin{enumerate}
\item\label{st-1} Theorems~\ref{thm-main-1} and \ref{thm-main-2} remain true with (\ref{manifold-boundary})--(\ref{alpha-tangent-plus}) being replaced by (\ref{intro-drift}).

\item Suppose that even the following stronger condition than (\ref{gamma-integrable-pre}) is satisfied: For each $n \in \mathbb{N}$ there exists a measurable function $\theta_n : E \rightarrow \mathbb{R}_+$ with $\int_E \theta_n(x) F(dx) < \infty$ such that
\begin{align}\label{gamma-integrable}
\| \gamma(h,x) \| \leq \theta_n(x) \quad \text{for all $h \in \mathcal{M}$ with $\| h \| \leq n$ and all $x \in E$.}
\end{align}
Then, in addition to statement (\ref{st-1}), the mapping in (\ref{intro-drift}) is continuous on~$\mathcal{M}$.
\end{enumerate}
\end{theorem}

\section{Notation and auxiliary results about stochastic invariance}\label{sec-notation}

In this section, we provide some notation and auxiliary results about stochastic invariance which we will use for the proofs our main results. In the sequel, for $h_0 \in H$ and $\epsilon > 0$ we denote by $B_{\epsilon}(h_0)$ the open ball
\begin{align*}
B_{\epsilon}(h_0) = \{ h \in H : \| h - h_0 \| < \epsilon \}.
\end{align*}
For technical reasons, we will also need the following concept of prelocal invariance:

\begin{definition}\label{def-prelocal}
Let $B_1 \subset B_2 \subset H$ be two nonempty Borel sets. $B_1$ is called \emph{prelocally invariant} in $B_2$ for (\ref{SPDE-j}), if for all $h_0 \in B_1$ there exists a local mild solution $r = r^{(h_0)}$ to (\ref{SPDE-j}) with lifetime $\tau > 0$ such that $(r^{\tau})_- \in B_1$ and $r^{\tau} \in B_2$ up to an evanescent set.
\end{definition}

\begin{remark}
Note that any non-empty Borel set $B \subset H$ is prelocally invariant for (\ref{SPDE-j}) in the sense of Definition~\ref{def-invariance} if and only if $B$ is prelocally invariant in $\overline{B}$ for (\ref{SPDE-j}) in the sense of Definition~\ref{def-prelocal}.
\end{remark}

We proceed with some auxiliary results about stochastic invariance which we will use later on. For the proofs we refer to \cite[Lemmas~2.11--2.16]{FTT-appendix}.

\begin{lemma}\label{lemma-jumps-global-set}
Let $B_1 \subset B_2 \subset H$ be two Borel sets such that $B_1$ is prelocally invariant in $B_2$ for (\ref{SPDE-j}). Then we have
\begin{align*}
h + \gamma(h,x) \in \overline{B}_2 \quad \text{ for $F$-almost all $x \in E$,} \quad \text{for all $h \in B_1$.}
\end{align*}
\end{lemma}

\begin{lemma}\label{lemma-jump-global-help}
Let $B_1 \subset B_2 \subset H$ be two Borel sets such that
\begin{align*}
h + \gamma(h,x) \in B_2 \quad \text{ for $F$-almost all $x \in E$,} \quad \text{for all $h \in B_1$.}
\end{align*}
Let $h_0 : \Omega \rightarrow H$ be a $\mathcal{F}_0$-measurable random variable and let $r = r^{(h_0)}$ be a local mild solution to (\ref{SPDE-j}) with lifetime $\tau > 0$ such that $(r^{\tau})_- \in B_1$ and $r^{\tau} \mathbbm{1}_{[\![ 0,\tau [\![} \in B_2$ up to an evanescent set. Then we have $r^{\tau} \in B_2$ up to an evanescent set.
\end{lemma}

\begin{lemma}\label{lemma-jump-global-help-2}
Let $B \subset C \subset H$ be two Borel sets such that $C$ is closed in $H$ and
\begin{align*}
h + \gamma(h,x) \in C \quad \text{ for $F$-almost all $x \in E$,} \quad \text{for all $h \in B$.}
\end{align*}
Let $h_0 : \Omega \rightarrow H$ be a $\mathcal{F}_0$-measurable random variable and let $r = r^{(h_0)}$ be a local mild solution to (\ref{SPDE-j}) with lifetime $\tau > 0$ such that $(r^{\tau})_- \in B$ up to an evanescent set. Then we have $r^{\tau} \in C$ up to an evanescent set.
\end{lemma}

\begin{lemma}\label{lemma-jumps-exact}
Let $G_1,G_2$ be metric spaces such that $G_1$ is separable. Let
$B \subset G_1$ be a Borel set, let $C \subset G_2$ be a closed set and let $\delta : G_1 \times E \rightarrow G_2$ be a measurable mapping such that $\delta(\bullet,x) : G_1 \rightarrow G_2$ is continuous for all $x \in E$. Suppose that
\begin{align*}
\delta(h,x) \in C \quad \text{for $F$-almost all $x \in E$,} \quad \text{for all $h \in B$.}
\end{align*}
Then we even have
\begin{align*}
\delta(h,x) \in C \quad \text{for all $h \in B$,} \quad \text{ for $F$-almost all $x \in E$.}
\end{align*}
\end{lemma}

\begin{lemma}\label{lemma-int-in-convex}
Let $(G, \mathcal{G}, \nu)$ be a $\sigma$-finite measure space, let $C \subset H$ be a closed, convex cone and let $f \in \mathcal{L}^1(G; H)$ be such that $f(x) \in C$ for $\nu$-almost all $x \in G$. Then we have
\begin{align*}
\int_G f d\nu \in C.
\end{align*}
\end{lemma}

\begin{lemma}\label{lemma-int-mu-pos}
Let $C \subset H$ be a closed, convex cone and let $\delta : \Omega \times \mathbb{R}_+ \times E \rightarrow H$ be an optional process satisfying
\begin{align*}
\mathbb{P} \bigg( \int_0^t \int_E \| \delta(s,x) \| \mu(ds,dx) < \infty \bigg) = 1 \quad \text{for all $t \geq 0$}
\end{align*}
such that
\begin{align*}
\delta(\bullet,x) \in C \quad \text{up to an evanescent set,} \quad \text{for $F$-almost all $x \in E$.}
\end{align*}
Then we have $X \in C$ up to an evanescent set, where $X$ denotes the integral process
\begin{align*}
X_t := \int_0^{t} \int_E \delta(s,x) \mu(ds,dx), \quad t \geq 0.
\end{align*}
\end{lemma}

\section{Local analysis of the invariance problem on half spaces}\label{sec-proof-half-space}

As a first building block for the proof of Theorem~\ref{thm-main-1}, our goal of this section is the proof of Theorem~\ref{thm-SDE-half-space}, which provides a local version of Theorem~\ref{thm-main-1} in the particular situation where the manifold is an open subset of a half space. More precisely, fix an arbitrary $m \in \mathbb{N}$ and consider the $\mathbb{R}^m$-valued SDE
\begin{align}\label{SDE-Theta}
\left\{
\begin{array}{rcl}
d Y_t & = & \Theta(Y_t)dt + \sum_{j \in \mathbb{N}} \Sigma^j(Y_t)d\beta_t^{j} + \int_E \Gamma(Y_{t-},x) (\mu(dt,dx) - F(dx)dt) \medskip
\\ Y_0 & = & y_0.
\end{array}
\right.
\end{align}
We assume that the mappings $\Theta : \mathbb{R}^m \rightarrow \mathbb{R}^m$, $\Sigma^j : \mathbb{R}^m \rightarrow \mathbb{R}^m$, $j \in \mathbb{N}$ and $\Gamma : \mathbb{R}^m \times E \rightarrow \mathbb{R}^m$ satisfy the regularity conditions (\ref{Lipschitz-alpha-st})--(\ref{linear-growth-sigma}) and (\ref{Lipschitz-gamma-rho})--(\ref{sigma-C1}). Instead of (\ref{rho-integrable}), we only demand that the mappings $\rho_n : E \rightarrow \mathbb{R}_+$, $n \in \mathbb{N}$ appearing in (\ref{Lipschitz-gamma-rho}), (\ref{linear-growth-rho}) satisfy (\ref{rho-square-int}). 

Let $V$ be an open subset of the half space $\mathbb{R}_+^m = \mathbb{R}_+ \times \mathbb{R}^{m-1}$, on which we consider the relative topology. Let $\partial V = \{ y \in V : y_1 = 0 \}$ be the set of all boundary points of $V$. Let $O_V \subset C_V \subset V$ be subsets such that $O_V$ is open in $V$ and $C_V$ is compact. In the sequel, we equip $\mathbb{R}^m$ with the Euclidean inner product and denote by $e_1 = (1,0,\ldots,0) \in \mathbb{R}^m$ the first unit vector.

\begin{theorem}\label{thm-SDE-half-space}
The following statements are equivalent:
\begin{enumerate}
\item $O_V$ is prelocally invariant in $C_V$ for (\ref{SDE-Theta}).

\item We have
\begin{align}\label{cond-Sigma-Y}
&\Sigma^j(y) \in T_y \partial V, \quad y \in O_V \cap \partial V,
\quad \text{for all $j \in \mathbb{N}$,}
\\ \label{cond-Gamma-jumps-Y} &y + \Gamma(y,x) \in C_V \quad \text{for $F$-almost all $x \in E$,} \quad \text{for all $y \in O_V$,}
\\ \label{cond-Gamma-int-Y} &\int_{E} |\langle e_1, \Gamma(y,x) \rangle| F(dx) < \infty,
\quad y \in O_V \cap \partial V,
\\ \label{cond-Theta-Y} &\langle e_1,\Theta(y) \rangle
- \int_{E} \langle e_1,\Gamma(y,x) \rangle F(dx) \geq 0, \quad y \in O_V \cap \partial V.
\end{align}
\end{enumerate}
\end{theorem}

\begin{proof}
For the sake of simplicity, we agree to write $O := O_V$, $\partial O := O \cap \partial V$ and $C := C_V$ during the proof.

\noindent(1) $\Rightarrow$ (2): Let $y \in O$ be arbitrary. Since $O$ is prelocally invariant in $C$ for (\ref{SDE-Theta}), there exists a local strong solution $Y = Y^{(y)}$ to (\ref{SDE-Theta}) with lifetime $\tau > 0$ such that $(Y^{\tau})_- \in O$ and $Y^{\tau} \in C$ up to an evanescent set.
Thus, Lemma~\ref{lemma-jumps-global-set} yields (\ref{cond-Gamma-jumps-Y}), and for every finite stopping time $\varrho \leq \tau$ we have
\begin{align}\label{pos-contr}
\mathbb{P}( \langle e_1, Y_{\varrho} \rangle \geq 0) = 1.
\end{align}
From now on, we assume that $y \in \partial O$. Let $(\Phi^j)_{j \in \mathbb{N}} \subset \mathbb{R}$ be a sequence with $\Phi^j \neq 0$ for only finitely many $j \in \mathbb{N}$, and let $\Psi : E \rightarrow \mathbb{R}$ be a measurable
function of the form $\Psi = c \mathbbm{1}_B$ with $c > -1$ and $B
\in \mathcal{E}$ satisfying $F(B) < \infty$. Let $Z$ be the
Dol\'eans-Dade exponential
\begin{align*}
Z = \mathcal{E} \bigg( \sum_{j \in \mathbb{N}} \Phi^j \beta^j + \int_0^{\bullet}
\int_{E} \Psi(x) (\mu(ds,dx) - F(dx)ds) \bigg).
\end{align*}
By \cite[Theorem~I.4.61]{Jacod-Shiryaev} the process $Z$ is a solution
of
\begin{align*}
Z_t = 1 + \sum_{j \in \mathbb{N}} \Phi^j \int_0^t Z_s d\beta_s^j + \int_0^t
\int_{E} Z_{s-} \Psi(x) (\mu(ds,dx) - F(dx)ds), \quad t \geq 0
\end{align*}
and, since $\Psi > -1$, the process $Z$ is a strictly positive local
martingale. There exists a strictly positive stopping time $\tau_1$
such that $Z^{\tau_1}$ is a martingale. Integration by parts (see \cite[Theorem~I.4.52]{Jacod-Shiryaev}) yields
\begin{equation}\label{product}
\begin{aligned}
\langle e_1,Y_t \rangle Z_t &= \int_0^t \langle e_1,Y_{s-} \rangle dZ_s + \int_0^t Z_{s-} d \langle e_1,Y_s \rangle
\\ &\quad+
\langle \langle e_1,Y^{c} \rangle,Z^c \rangle_t + \sum_{s \leq t} \langle e_1,\Delta Y_s \rangle \Delta Z_s, \quad t \geq 0.
\end{aligned}
\end{equation}
Taking into account the dynamics (\ref{SDE-Theta}), we have
\begin{align}\label{rhs0}
\langle \langle e_1,Y^c \rangle, Z^c \rangle_t &= \sum_{j \in \mathbb{N}} \Phi^j \int_0^t Z_s
\langle e_1,\Sigma^j(Y_s) \rangle ds, \quad t \geq 0,
\\ \label{rhs} \sum_{s \leq t} \langle e_1,\Delta Y_s \rangle \Delta Z_s &= \int_0^t
\int_{E} Z_{s-} \Psi(x) \langle e_1,\Gamma(Y_{s-},x) \rangle \mu(ds,dx), \quad t \geq 0.
\end{align}
Incorporating (\ref{SDE-Theta}), (\ref{rhs0}) and (\ref{rhs}) into
(\ref{product}), we obtain
\begin{equation}\label{make-contradiction}
\begin{aligned}
\langle e_1,Y_t \rangle Z_t &= M_t + \int_0^t Z_{s-} \bigg( \langle e_1,\Theta(Y_{s-}) \rangle +
\sum_{j \in \mathbb{N}} \Phi^j \langle e_1,\Sigma^{j}(Y_{s-}) \rangle
\\ &\qquad\qquad\qquad\qquad + \int_{E} \Psi(x)
\langle e_1,\Gamma(Y_{s-},x) \rangle F(dx) \bigg) ds, \quad t \geq 0,
\end{aligned}
\end{equation}
where $M$ is a local martingale with $M_0 = 0$. There exists a
strictly positive stopping time $\tau_2$ such that $M^{\tau_2}$ is a
martingale.

By the continuity of $\Theta$ there exist a strictly positive stopping
time $\tau_3$ and a constant
$\tilde{\Theta} > 0$ such that
\begin{align*}
| \langle e_1,\Theta(Y_{(t \wedge \tau_3)-}) \rangle | \leq \tilde{\Theta}, \quad t
\geq 0.
\end{align*}
Suppose that $\Sigma^j(y) \notin T_y \partial V$, i.e. $\langle e_1,\Sigma^{j}(y) \rangle \neq 0$, for some $j \in \mathbb{N}$. By the
continuity of $\Sigma$ there exist
$\eta > 0$ and a strictly positive stopping time $\tau_4 \leq 1$ such that
\begin{align*}
| \langle e_1,\Sigma^{j}(Y_{(t \wedge \tau_4)-}) \rangle| \geq \eta, \quad t \geq 0.
\end{align*}
Let $(\Phi_k)_{k \in \mathbb{N}} \subset \mathbb{R}$ be the sequence given by
\begin{align*}
\Phi^k =
\begin{cases}
-{\rm sign}( \langle e_1,\Sigma^{k}(y) \rangle ) \frac{\tilde{\Theta} + 1}{\eta}, & k = j,
\\ 0, & k \neq j.
\end{cases}
\end{align*}
Furthermore, let $\Psi := 0$ and $\varrho := \tau \wedge \tau_1 \wedge \tau_2 \wedge \tau_3 \wedge \tau_4$. Taking expectation in
(\ref{make-contradiction}) yields $\mathbb{E}[ \langle e_1,Y_{\varrho} \rangle Z_{\varrho}] <
0$, implying $\mathbb{P}( \langle e_1,Y_{\varrho} \rangle < 0) > 0$, which contradicts
(\ref{pos-contr}). This proves (\ref{cond-Sigma-Y}).

Now suppose $\int_E |\langle e_1,\Gamma(y,x) \rangle| F(dx) = \infty$. By the Cauchy-Schwarz inequality, for all $B \in \mathcal{E}$ with $F(B) < \infty$ the map $y \mapsto \int_B \Gamma(y,x) F(dx)$
is continuous. Using the
$\sigma$-finiteness of $F$, there exist $B \in \mathcal{E}$ with
$F(B) < \infty$ and a strictly positive stopping time $\tau_4 \leq
1$ such that
\begin{align*}
-\frac{1}{2} \int_{B} |\langle e_1,\Gamma(Y_{(t \wedge \tau_4)-},x) \rangle| F(dx) \leq
-(\tilde{\Theta}+1), \quad t \geq 0.
\end{align*}
Let $\Phi := 0$, $\Psi := -\frac{1}{2} \mathbbm{1}_{B}$ and $\varrho :=
\tau \wedge \tau_1 \wedge \tau_2 \wedge \tau_3 \wedge \tau_4$. Taking expectation in
(\ref{make-contradiction}) we obtain $\mathbb{E}[ \langle e_1,Y_{\varrho} \rangle
Z_{\varrho}] < 0$, implying $\mathbb{P}( \langle e_1, Y_{\varrho} \rangle < 0) > 0$, which
contradicts (\ref{pos-contr}). This yields (\ref{cond-Gamma-int-Y}).

Since $F$ is $\sigma$-finite, there exists a sequence $(B_n)_{n \in
\mathbb{N}} \subset \mathcal{E}$ with $B_n \uparrow E$ and $F(B_n) <
\infty$, $n \in \mathbb{N}$. We shall show for all $n \in
\mathbb{N}$ the relation
\begin{align}\label{cond-drift-approx}
\langle e_1,\Theta(y) \rangle + \int_E \Psi_n(x) \langle e_1,\Gamma(y,x) \rangle F(dx) \geq 0,
\end{align}
where $\Psi_n := - (1 - \frac{1}{n}) \mathbbm{1}_{B_n}$. Suppose, on
the contrary, that (\ref{cond-drift-approx}) is not satisfied for
some $n \in \mathbb{N}$. Then there exist
$\eta > 0$ and a strictly positive stopping time $\tau_4 \leq 1$
such that
\begin{align*}
&\langle e_1,\Theta(Y_{(t \wedge \tau_4)-}) \rangle + \int_E \Psi_n(x) \langle e_1,\Gamma(Y_{(t
\wedge \tau_4)-},x) \rangle F(dx) \leq -\eta, \quad t \geq 0.
\end{align*}
Let $\Phi := 0$ and $\varrho := \tau \wedge \tau_1 \wedge \tau_2 \wedge \tau_3 \wedge \tau_4$. Taking
expectation in (\ref{make-contradiction}) we obtain
$\mathbb{E}[ \langle e_1,Y_{\varrho} \rangle Z_{\varrho}] < 0$, implying
$\mathbb{P}( \langle e_1,Y_{\varrho} \rangle < 0) > 0$, which contradicts
(\ref{pos-contr}). This yields (\ref{cond-drift-approx}). By
(\ref{cond-drift-approx}), (\ref{cond-Gamma-int-Y}) and Lebesgue's dominated convergence theorem,
we conclude (\ref{cond-Theta-Y}).

\noindent(2) $\Rightarrow$ (1): The metric projection $\Pi = \Pi_{\mathbb{R}_+^m} : \mathbb{R}^m \rightarrow \mathbb{R}_+^m$ on
the half space $\mathbb{R}_+^m$ is given by
\begin{align}\label{proj-half-space}
\Pi(y^1,y^2,\ldots,y^m) = ((y^1)^+,y^2,\ldots,y^m),
\end{align}
and therefore, it satisfies
\begin{align*}
\| \Pi(y_1) - \Pi(y_2) \| \leq \| y_1 - y_2 \| \quad \text{for all $y_1,y_2 \in \mathbb{R}^m$.}
\end{align*}
Consequently, the mappings $\Theta_{\Pi} : \mathbb{R}^m \rightarrow \mathbb{R}^m$, $\Sigma_{\Pi}^j : \mathbb{R}^m \rightarrow \mathbb{R}^m$, $j \in \mathbb{N}$ and $\Gamma_{\Pi} : \mathbb{R}^m \times E \rightarrow \mathbb{R}^m$ defined as
\begin{align*}
\Theta_{\Pi} := \Theta \circ \Pi, \quad \Sigma_{\Pi}^j := \Sigma^j \circ \Pi \quad \text{and} \quad \Gamma_{\Pi}(\bullet,x) := \Gamma(\bullet,x) \circ \Pi
\end{align*}
also satisfy the regularity conditions (\ref{Lipschitz-alpha-st})--(\ref{linear-growth-sigma}) and (\ref{Lipschitz-gamma-rho})--(\ref{sigma-C1}), which ensures existence and uniqueness of local strong solutions to the SDE
\begin{align}\label{SDE-projection}
\left\{
\begin{array}{rcl}
dY_t & = & \Theta_{\Pi}(Y_t)dt + \sum_{j \in \mathbb{N}} \Sigma_{\Pi}^j(Y_t) d\beta_t^j
\\ && + \int_E \Gamma_{\Pi}(Y_{t-},x) (\mu(dt,dx) - F(dx) dt) \medskip
\\ Y_0 & = & y_0.
\end{array}
\right.
\end{align}
Now, let $y_0 \in O$ be arbitrary. Then there exists a local strong solution $Y$ to (\ref{SDE-projection}) with $Y_0 = y_0$ and some lifetime $\tau > 0$. First, suppose that $y_0 \notin \partial O$. Then there exists $\epsilon > 0$ such that $\overline{B_{\epsilon}(y_0)} \subset O$. We define the strictly positive stopping time
\begin{align*}
\varrho := \inf \{ t \geq 0 : Y_t \notin B_{\epsilon}(y_0) \} \wedge \tau.
\end{align*}
Then we have
\begin{align*}
(Y^{\varrho})_- \in \overline{B_{\epsilon}(y_0)} \subset O.
\end{align*}
Using (\ref{cond-Gamma-jumps-Y}) and Lemma~\ref{lemma-jump-global-help} we obtain $Y^{\varrho} \in C$ up to an evanescent set.

From now on, we suppose that $y_0 \in \partial O$. Then there exists $\epsilon > 0$ such that $\overline{B_{\epsilon}(y_0)} \cap \mathbb{R}_+^m \subset O$. We define the strictly positive stopping time
\begin{align*}
\varrho := \inf \{ t \geq 0 : Y_t \notin B_{\epsilon}(y_0) \} \wedge \tau.
\end{align*}
Setting
\begin{align*}
P := \overline{B_{\epsilon}(y_0)} \quad \text{and} \quad \mathbb{R}_-^m := \{ y \in \mathbb{R}^m : y_1 \leq 0 \},
\end{align*}
by taking into account that the metric projection $\Pi$ on $\mathbb{R}_+^m$ is given by (\ref{proj-half-space}), we have
\begin{align}\label{pi-at-boundary}
\Pi(y) \in \partial O, \quad y \in P \cap \mathbb{R}_-^m.
\end{align}
By (\ref{proj-half-space}) and (\ref{cond-Gamma-jumps-Y}), for all $y \in P \cap \mathbb{R}_+^m$ we have
\begin{equation}\label{cond-gamma-D-2}
\begin{aligned}
&\langle e_1, y + \xi \Gamma_{\Pi}(y,x) \rangle = (1 - \xi) \langle e_1, y \rangle + \xi ( \langle e_1, y \rangle + \langle e_1, \Gamma_{\Pi}(y,x) \rangle)
\\ &= (1 - \xi) \langle e_1,y \rangle + \xi \langle e_1,y + \Gamma(y,x) \rangle
\geq 0 \quad \text{for all $\xi \in [0,1]$,}
\\ &\qquad\qquad\qquad\qquad\qquad\qquad\qquad\qquad\qquad\,\, \text{for $F$-almost all $x \in E$.}
\end{aligned}
\end{equation}
Furthermore, by (\ref{cond-Sigma-Y})--(\ref{cond-Theta-Y}) and (\ref{pi-at-boundary}), for all $y \in P \cap \mathbb{R}_-^m$ we have
\begin{align}\label{cond-sigma-D}
&\langle e_1, \Sigma_{\Pi}^j(y) \rangle = \langle e_1, \Sigma^j(\Pi(y)) \rangle = 0, \quad \text{for all $j \in \mathbb{N}$,}
\\ \label{cond-gamma-D-1} &\langle e_1, \Gamma_{\Pi}(y,x) \rangle = \langle e_1,\Pi(y) \rangle + \langle e_1, \Gamma(\Pi(y),x) \rangle
\\ \notag &= \langle e_1, \Pi(y) + \Gamma(\Pi(y),x) \rangle \geq 0, \quad \text{for $F$-almost all $x \in E$,}
\\\label{cond-int-D}  &\int_E | \langle e_1, \Gamma_{\Pi}(y) \rangle | F(dx) = \int_E | \langle e_1, \Gamma(\Pi(y)) \rangle | F(dx) < \infty,
\\ \label{cond-alpha-D} &\langle e_1, \Theta_{\Pi}(y) \rangle - \int_E \langle e_1, \Gamma_{\Pi}(y,x) \rangle F(dx)
\\ \notag &= \langle e_1, \Theta(\Pi(y)) \rangle - \int_E \langle e_1, \Gamma(\Pi(y),x) \rangle F(dx) \geq 0.
\end{align}
The function $\phi : \mathbb{R} \rightarrow \mathbb{R}$, $\phi(y) := (-y^3)^+$ is of class $C^2(\mathbb{R})$ and we have $\phi'(y) < 0$ for $y < 0$ and $\phi'(y) = \phi''(y) = 0$ for $y \geq 0$. By (\ref{cond-gamma-D-2})--(\ref{cond-alpha-D}) and Lemma~\ref{lemma-jumps-exact}, we obtain
\begin{align}\label{cond-with-phi-1}
&\phi'(\langle e_1,y \rangle) \bigg( \langle e_1,\Theta_{\Pi}(y) \rangle - \int_E \langle e_1,\Gamma_{\Pi}(y,x) \rangle F(dx) \bigg) \leq 0, \quad y \in P
\\ &\phi''(\langle e_1,y \rangle) |\langle e_1, \Sigma_{\Pi}^{j}(y) \rangle|^2 = 0, \quad y \in P, \quad \text{for all $j \in \mathbb{N}$}
\\ &\phi'(\langle e_1,y \rangle) \langle e_1,\Sigma_{\Pi}^{j}(y) \rangle = 0, \quad \quad y \in P, \quad \text{for all $j \in \mathbb{N}$}
\\ \label{cond-with-phi-4} &\bigg( \int_0^1 \phi'(\langle e_1,y + \xi \Gamma_{\Pi}(y,x) \rangle) d\xi \bigg) \langle e_1,\Gamma_{\Pi}(y,x) \rangle \leq 0 \quad \text{for all $y \in P$,}
\\ \notag & \quad \text{for $F$-almost all $x \in E$.}
\end{align}
Applying It\^{o}'s formula (see \cite[Theorem~I.4.57]{Jacod-Shiryaev}) yields $\mathbb{P}$-almost surely
\begin{align*}
&\phi(\langle e_1,Y_{t \wedge \varrho} \rangle) = \phi(\langle e_1,y_0 \rangle)
\\ &\quad + \int_0^{t \wedge \varrho} \bigg( \phi'(\langle e_1,Y_s \rangle) \langle e_1,\Theta_{\Pi}(Y_s) \rangle
+ \frac{1}{2} \sum_{j \in \mathbb{N}} \phi''(\langle e_1,Y_s \rangle) |\langle e_1,\Sigma_{\Pi}^{j}(Y_s) \rangle|^2
\\ &\quad + \int_E \big( \phi(\langle e_1,Y_s + \Gamma_{\Pi}(Y_s,x) \rangle ) - \phi(\langle e_1,Y_s \rangle)
\\ &\quad\quad\quad\quad - \phi'(\langle e_1, Y_s \rangle) \langle e_1,\Gamma_{\Pi}(Y_s,x) \rangle \big) F(dx) \bigg) ds
\\ &\quad + \sum_{j \in \mathbb{N}} \int_0^{t \wedge \varrho} \phi'(\langle e_1,Y_s \rangle) \langle e_1,\Sigma_{\Pi}^{j}(Y_s) \rangle d\beta_s^j
\\ &\quad + \int_0^{t \wedge \varrho} \int_E \big( \langle e_1, \phi( \langle e_1,Y_{s-} + \Gamma_{\Pi}(Y_{s-},x) \rangle ) - \phi( \langle e_1, Y_{s-} \rangle) \big)
\\ &\quad\quad\quad\quad\quad\quad\quad (\mu(ds,dx) - F(dx)ds), \quad t \geq 0.
\end{align*}
By (\ref{cond-int-D}) and Taylor's theorem we obtain $\mathbb{P}$-almost surely
\begin{align*}
&\phi(\langle e_1,Y_{t \wedge \varrho} \rangle)
\\ &= \int_0^{t \wedge \varrho} \bigg[ \phi'( \langle e_1,Y_s \rangle) \bigg( \langle e_1,\Theta_{\Pi}(Y_s) \rangle - \int_E \langle e_1,\Gamma_{\Pi}(Y_s,x) \rangle F(dx) \bigg)
\\ &\quad\quad\quad\quad\quad + \frac{1}{2} \sum_{j \in \mathbb{N}} \phi''( \langle e_1,Y_s \rangle ) |\langle e_1,\Sigma_{\Pi}^{j}(Y_s) \rangle|^2 \bigg] ds
\\ &\quad + \sum_{j \in \mathbb{N}} \int_0^{t \wedge \varrho} \phi'(\langle e_1,Y_s \rangle) \langle e_1,\Sigma_{\Pi}^{j}(Y_s) \rangle d\beta_s^j
\\ &\quad + \int_0^{t \wedge \varrho} \int_E \bigg( \int_0^1 \phi'(\langle e_1, Y_{s-} + \xi \Gamma_{\Pi}(Y_{s-},x) \rangle) d\xi \bigg) \langle e_1,\Gamma_{\Pi}(Y_{s-},x) \rangle
\\ &\qquad\qquad\qquad\quad\,\,\,\, \mu(ds,dx), \quad t \geq 0.
\end{align*}
By (\ref{cond-with-phi-1})--(\ref{cond-with-phi-4}) and Lemmas~\ref{lemma-int-in-convex} and \ref{lemma-int-mu-pos}, we deduce that $\phi(\langle e_1,Y^{\varrho} \rangle) \leq 0$ up to an evanescent set.
Therefore, we obtain on up to an evanescent set
\begin{align*}
(Y^{\tau})_- \in \overline{B_{\epsilon}(y_0)} \cap \mathbb{R}_+^m \subset O.
\end{align*}
Using (\ref{cond-Gamma-jumps-Y}) and Lemma~\ref{lemma-jump-global-help-2} we obtain $Y^{\tau} \in C$ up to an evanescent set.
Since $\Theta|_C = \Theta_{\Pi}|_C$, $\Sigma^j|_C = \Sigma_{\Pi}^j|_C$ for all $j \in \mathbb{N}$ and $\Gamma(\bullet,x)|_C = \Gamma_{\Pi}(\bullet,x)|_C$ for all $x \in E$, the process $Y$ is also a local strong solution to (\ref{SDE-Theta}) with lifetime $\varrho$, proving that $O$ is prelocally invariant in $C$ for (\ref{SDE-Theta}).
\end{proof}

Note that $V$ is a $m$-dimensional $C^3$-submanifold with boundary of $\mathbb{R}^m$, and that for $y \in \partial V$ the inward pointing normal vector to $\partial V$ at $y$ is given by the first unit vector $e_1 = (1,0,\ldots,0) \in \mathbb{R}^m$. In order to see that for the submanifold $V$ conditions (\ref{cond-Sigma-Y})--(\ref{cond-Theta-Y}) resemble conditions (\ref{domain})--(\ref{jumps}) and (\ref{manifold-boundary})--(\ref{alpha-tangent-plus}), we require the following auxiliary result.

\begin{lemma}\label{lemma-stratonovich}
Suppose that (\ref{cond-Sigma-Y}) is satisfied. Then for all $j \in \mathbb{N}$ we have
\begin{align*}
\langle e_1,D\Sigma^j(y)\Sigma^j(y) \rangle = 0, \quad y \in O_V \cap \partial V.
\end{align*}
\end{lemma}

\begin{proof}
The statement is a consequence of \cite[Lemma~3.13]{FTT-appendix}.
\end{proof}

\section{Local analysis of the invariance problem on submanifolds with boundary}\label{sec-proof-main-one-chart}

As next building block for the proof of Theorem~\ref{thm-main-1}, our goal of this section is the proof of Theorem~\ref{thm-invariance-partition}, which provides a local version of Theorem~\ref{thm-main-1}. We assume that  for the $m$-dimensional $C^3$-submanifold $\mathcal{M}$ with boundary of $H$ there exist
\begin{itemize}
\item a $m$-dimensional $C^3$-submanifold $\mathcal{N}$ with boundary of $\mathbb{R}^m$,

\item parametrizations $\phi : V \subset \mathbb{R}_+^m \rightarrow \mathcal{M}$ and $\psi : V \subset \mathbb{R}_+^m  \rightarrow \mathcal{N}$,

\item and elements $\zeta_1,\ldots,\zeta_m \in \mathcal{D}(A^*)$
such that the mapping $f := \phi \circ \psi^{-1} : \mathcal{N} \rightarrow \mathcal{M}$ has the inverse
\begin{align}\label{f-inverse-conv}
f^{-1} : \mathcal{M} \rightarrow \mathcal{N}, \quad f^{-1}(h) = \langle \zeta,h \rangle := (\langle \zeta_1,h \rangle,\ldots,\langle \zeta_m,h \rangle).
\end{align}
\end{itemize}
In other words, the diagram
\begin{equation}\label{diagram}
\begin{aligned}
\begin{xy}
  \xymatrix{
      \mathcal{N} \subset \mathbb{R}^m \ar@<2pt>[rr]^f  &     &  \mathcal{M} \subset H \ar@<2pt>[ll]^{\langle \zeta,\bullet \rangle} \\
                             & \ar[ul]_{\psi} V \subset \mathbb{R}_+^m \ar[ru]^{\phi} &
  }
\end{xy}
\end{aligned}
\end{equation}
commutes.

\begin{remark}
According to \cite[Proposition~3.11]{FTT-appendix}, for an arbitrary $C^3$-submanifold $\mathcal{M}$ with boundary of $H$ and an arbitrary point $h_0 \in \mathcal{M}$ there always exists a neighborhood of $h_0$ such that a diagram of form (\ref{diagram}) exists and commutes. We will use this result for the global analysis of the invariance problem in Section~\ref{sec-proof-main}.
\end{remark}

\begin{remark}
For a $C^3$-submanifold $\mathcal{M}$ without boundary there even exist local parametrizations $\phi : V \subset \mathbb{R}^m \rightarrow U \cap \mathcal{M}$ with inverses being of the form $\langle \zeta,\bullet \rangle$ for some $\zeta_1,\ldots,\zeta_m \in \mathcal{D}(A^*)$, see \cite{fillnm}. In the present situation, where $\mathcal{M}$ is a submanifold with boundary, this is generally not possible, and thus, we consider the situation where the diagram (\ref{diagram}) commutes.
\end{remark}

Let $O_{\mathcal{M}} \subset C_{\mathcal{M}} \subset \mathcal{M}$ be subsets. We assume that $O_{\mathcal{M}}$ is open in $\mathcal{M}$ and $C_{\mathcal{M}}$ is compact. Our announced main result of this section reads as follows.

\begin{theorem}\label{thm-invariance-partition}
The following statements are equivalent:
\begin{enumerate}
\item $O_{\mathcal{M}}$ is prelocally invariant in $C_{\mathcal{M}}$ for (\ref{SPDE-j}).

\item The following conditions are satisfied:
\begin{align}\label{domain-local}
&O_{\mathcal{M}} \subset \mathcal{D}(A),
\\ \label{sigma-tangent-local} &\sigma^j(h) \in T_h \mathcal{M}, \quad h \in O_{\mathcal{M}}, \quad j \in \mathbb{N},
\\ \label{sigma-tangent-local-2} &\sigma^j(h) \in T_h \partial \mathcal{M}, \quad h \in O_{\mathcal{M}} \cap \partial \mathcal{M}, \quad j \in \mathbb{N},
\\ \label{jumps-local}
&h + \gamma(h,x) \in C_{\mathcal{M}} \quad \text{for $F$-almost all $x \in E$,} \quad \text{for all $h \in O_{\mathcal{M}}$,}
\\ \label{manifold-boundary-local} &\int_{E} |\langle \eta_h, \gamma(h,x) \rangle| F(dx) < \infty, \quad h \in O_{\mathcal{M}} \cap \partial \mathcal{M},
\\ \label{alpha-tangent-local} &Ah + \alpha(h) - \frac{1}{2} \sum_{j \in \mathbb{N}} D\sigma^j(h)\sigma^j(h)
\\ \notag &\quad -\int_E \Pi_{(T_h \mathcal{M})^{\perp}} \gamma(h,x) F(dx) \in T_h
\mathcal{M}, \quad h \in O_{\mathcal{M}},
\\ \label{alpha-tangent-plus-local} &\langle \eta_h,Ah + \alpha(h) \rangle - \frac{1}{2} \sum_{j \in \mathbb{N}} \langle \eta_h, D \sigma^j(h) \sigma^j(h) \rangle
\\ \notag &\quad - \int_{E} \langle \eta_h,\gamma(h,x) \rangle F(dx) \geq 0, \quad h \in O_{\mathcal{M}} \cap \partial
\mathcal{M}.
\end{align}
\end{enumerate}
In either case, $A$ and the mapping in (\ref{alpha-tangent-local}) are
continuous on $O_{\mathcal{M}}$.
\end{theorem}

Our strategy for proving Theorem~\ref{thm-invariance-partition} can be divided into the following steps:
\begin{itemize}
\item Define the $\mathbb{R}^m$-valued SDE (\ref{SDE-abc}), whose coefficients $a,b^j,c$ are given by pull-backs in terms of $\alpha,\sigma^j,\gamma$.

\item Define the $\mathbb{R}^m$-valued SDE (\ref{SDE-Theta}), whose coefficients $\Theta,\Sigma^j,\Gamma$ are given by pull-backs in terms of $a,b^j,c$.

\item Provide conditions (\ref{cond-Sigma-Y})--(\ref{cond-Theta-Y}) for invariance of $V$ for the SDE (\ref{SDE-Theta}); this has already been established in Theorem~\ref{thm-SDE-half-space}.

\item Translate these conditions into conditions (\ref{sigma-tangent-local-D-1})--(\ref{alpha-tangent-plus-local-D}) regarding invariance of $\mathcal{N}$ for the SDE (\ref{SDE-abc}).

\item Translate these conditions into conditions (\ref{domain-local})--(\ref{alpha-tangent-plus-local}) regarding invariance of $\mathcal{M}$ for the original SPDE (\ref{SPDE-j}).
\end{itemize}
Now, we start with the formal proofs. First, we prepare an auxiliary result.

\begin{lemma}\label{lemma-cont-statement}
The following statements are true:
\begin{enumerate}
\item For each $h \in H$ we have
\begin{align*}
\sum_{j \in \mathbb{N}} \| D \sigma^j(h) \sigma^j(h) \| < \infty,
\end{align*}
and the mapping
\begin{align*}
H \rightarrow H, \quad h \mapsto \sum_{j \in \mathbb{N}} D \sigma^j(h) \sigma^j(h)
\end{align*}
is continuous.

\item If (\ref{jumps-local}) is satisfied, then for each $h \in O_{\mathcal{M}}$ we have
\begin{align*}
\int_E \| \Pi_{(T_h \mathcal{M})^{\perp}} \gamma(h,x) \| F(dx) < \infty,
\end{align*}
and the mapping
\begin{align*}
O_{\mathcal{M}} \rightarrow H, \quad h \mapsto \int_E \Pi_{(T_h \mathcal{M})^{\perp}} \gamma(h,x) F(dx)
\end{align*}
is continuous.
\end{enumerate}
\end{lemma}

\begin{proof}
This follows from \cite[Lemma~2.17 and Corollary~3.28]{FTT-appendix}.
\end{proof}

Let $G$ be another separable Hilbert space. For any $k \in \mathbb{N}$ we denote by $C_b^k(G;H)$ the linear space consisting of all $f \in C^k(G;H)$ such that $D^i f$ is bounded for all $i=1,\ldots,k$. In particular, for each $f \in C_b^k(G;H)$ the mappings $D^i f$, $i=0,\ldots,k-1$ are Lipschitz continuous. We do not demand that $f$ itself is bounded, as this would exclude continuous linear operators $f \in L(G;H)$.

\begin{definition}\label{def-pull-back}
Let $\alpha : H \rightarrow H$, $\sigma^j : H \rightarrow H$, $j \in \mathbb{N}$ and $\gamma : H \times E \rightarrow H$ be mappings satisfying
\begin{align}\label{sum-pull}
\sum_{j \in \mathbb{N}} \| \sigma^j(h) \|^2 < \infty \quad \text{and} \quad \int_E \| \gamma(h,x) \|^2 F(dx) < \infty
\end{align}
for all $h \in H$,
and let $f : G \rightarrow H$ and $g \in C_b^2(H;G)$ be mappings.
We define the mappings $(f,g)_{\lambda}^{\star} \alpha : G \rightarrow G$, $(f,g)_W^{\star}\sigma^j : G \rightarrow G$, $j \in \mathbb{N}$ and $(f,g)_{\mu}^{\star}\gamma : G \times E \rightarrow G$ as
\begin{align}\label{star-1}
((f,g)_{\lambda}^{\star} \alpha)(z) &:= Dg(h) \alpha(h) + \frac{1}{2} \sum_{j \in \mathbb{N}} D^2 g(h)(\sigma^j(h),\sigma^j(h))
\\ \notag &\quad + \int_E \big( g(h + \gamma(h,x)) - g(h)
- Dg(h) \gamma(h,x) \big) F(dx),
\\ \label{star-2} ((f,g)_W^{\star} \sigma^j)(z) &:= Dg(h) \sigma^j(h),
\\ \label{star-3} ((f,g)_{\mu}^{\star} \gamma)(z,x) &:= g(h + \gamma(h,x)) - g(h),
\end{align}
where $h = f(z)$.
\end{definition}

The following results show that the mappings from Definition~\ref{def-pull-back} may be regarded as pull-backs for jump-diffusions. First, we provide sufficient conditions which ensure that the regularity conditions (\ref{Lipschitz-alpha-st})--(\ref{linear-growth-sigma}) and (\ref{Lipschitz-gamma-rho})--(\ref{sigma-C1}) are preserved.

\begin{lemma}\label{lemma-Lipschitz-pb}
Let $\alpha : H \rightarrow H$, $\sigma^j : H \rightarrow H$, $j \in \mathbb{N}$ and $\gamma : H \times E \rightarrow H$ be mappings satisfying the regularity conditions (\ref{Lipschitz-alpha-st})--(\ref{linear-growth-sigma}) and (\ref{Lipschitz-gamma-rho})--(\ref{sigma-C1}). Furthermore, let $f \in C_b^1(G;H)$ and $g \in C_b^3(H;G)$ be arbitrary. Then the following statements are true:
\begin{enumerate}
\item The mappings $(f,g)_{\lambda}^{\star} \alpha$, $((f,g)_W^{\star} \sigma^j)_{j \in \mathbb{N}}$ and $(f,g)_{\mu}^{\star} \gamma$ also fulfill the regularity conditions (\ref{Lipschitz-alpha-st})--(\ref{linear-growth-sigma}) and (\ref{Lipschitz-gamma-rho})--(\ref{sigma-C1}), but with the mappings $\rho_n : E \rightarrow \mathbb{R}_+$, $n \in \mathbb{N}$ appearing in (\ref{Lipschitz-gamma-rho}), (\ref{linear-growth-rho}) only satisfying (\ref{rho-square-int}) instead of (\ref{rho-integrable}).

\item If $g \in L(H;G)$, then the mappings $\rho_n : E \rightarrow \mathbb{R}_+$, $n \in \mathbb{N}$ appearing in (\ref{Lipschitz-gamma-rho}), (\ref{linear-growth-rho}) even satisfy (\ref{rho-integrable}).
\end{enumerate}
\end{lemma}

\begin{proof}
See \cite[Lemma~2.24]{FTT-appendix}.
\end{proof}

Recall that $\mathcal{M}$ denotes a $C^3$-submanifold with boundary of the  separable Hilbert space $H$.
Let $\mathcal{N}$ be a $C^3$-submanifold with boundary of $G$. We assume there exist parametrizations $\phi : V \rightarrow \mathcal{M}$ and $\psi : V \rightarrow \mathcal{N}$. Let $f := \phi \circ \psi^{-1} : \mathcal{N} \rightarrow \mathcal{M}$ and $g := f^{-1} : \mathcal{M} \rightarrow \mathcal{N}$. Then the diagram
\begin{align*}
\begin{xy}
  \xymatrix{
      \mathcal{N} \subset G \ar@<2pt>[rr]^f  &     &  \mathcal{M} \subset H \ar@<2pt>[ll]^{g} \\
                             & \ar[ul]_{\psi} V \subset \mathbb{R}_+^m \ar[ru]^{\phi} &
  }
\end{xy}
\end{align*}
commutes. We assume that $\phi$, $\psi$, $\Phi := \phi^{-1}$, $\Psi := \psi^{-1}$ have extensions $\phi \in C_{b}^3(\mathbb{R}^m;H)$, $\psi \in C_{b}^3(\mathbb{R}^m;G)$, $\Phi \in C_{b}^3(H;\mathbb{R}^m)$, $\Psi \in C_{b}^3(G;\mathbb{R}^m)$. Consequently, the mappings $f$, $g$ have extensions $f \in C_{b}^3(G;H)$, $g \in C_{b}^3(H;G)$.

We define the subsets $O_{\mathcal{N}} \subset C_{\mathcal{N}} \subset \mathcal{N}$ by $O_{\mathcal{N}} := g(O_{\mathcal{M}})$ and $C_{\mathcal{N}} := g(C_{\mathcal{M}})$.

\begin{definition}
Let $\beta : O_{\mathcal{M}} \rightarrow H$, $\sigma^j : O_{\mathcal{M}} \rightarrow H$, $j \in \mathbb{N}$ and $\gamma : O_{\mathcal{M}} \times E \rightarrow H$ be mappings satisfying (\ref{sum-pull}) for all $h \in O_{\mathcal{M}}$.
We define the mappings $f_{\lambda}^{\star} \beta : O_{\mathcal{N}} \rightarrow G$, $f_W^{\star}\sigma^j : O_{\mathcal{N}} \rightarrow G$, $j \in \mathbb{N}$ and $f_{\mu}^{\star} \gamma : O_{\mathcal{N}} \times E \rightarrow G$ as
\begin{align*}
(f_{\lambda}^{\star} \beta)(z) &:= ((f,g)_{\lambda}^{\star} \beta)(z),
\\ (f_{W}^{\star} \sigma^j)(z) &:= ((f,g)_W^{\star} \sigma^j)(z),
\\ (f_{\mu}^{\star} \gamma)(z,x) &:= ((f,g)_{\mu}^{\star} \gamma)(z,x)
\end{align*}
according to (\ref{star-1})--(\ref{star-3}).
\end{definition}

Let $a : G \rightarrow G$, $b^j : G \rightarrow G$, $j \in \mathbb{N}$ and $c : G \times E \rightarrow G$ be mappings satisfying the regularity conditions (\ref{Lipschitz-alpha-st})--(\ref{linear-growth-sigma}) and (\ref{Lipschitz-gamma-rho})--(\ref{sigma-C1}). In the sequel, for $z \in \partial \mathcal{N}$ the vector $\xi_z$ denotes the inward pointing normal vector to $\partial \mathcal{N}$ at $z$.

The following result shows how the invariance conditions of Theorem~\ref{thm-invariance-partition} translate when we change to another manifold, and how this is related to the just defined pull-backs.

\begin{proposition}\label{prop-cond-imply-cond}
Suppose we have (\ref{domain-local}) and define $\beta : O_{\mathcal{M}} \rightarrow H$ as
\begin{align*}
\beta(h) := Ah + \alpha(h), \quad h \in O_{\mathcal{M}}.
\end{align*}
Moreover, we suppose that
\begin{align}\label{a-f-star-beta}
a(z) &= (f_{\lambda}^{\star} \beta)(z), \quad z \in O_{\mathcal{N}},
\\ \label{b-f-star-sigma} b^j(z) &= (f_W^{\star} \sigma^j)(z), \quad j \in \mathbb{N} \text{ and } z \in O_{\mathcal{N}},
\\ \label{c-f-star-gamma} c(z,x) &= (f_{\mu}^{\star} \gamma)(z,x) \quad \text{for $F$-almost all $x \in E$,} \quad \text{for all $z \in O_{\mathcal{N}}$.}
\end{align}
Then the following statements are true:
\begin{enumerate}
\item If conditions (\ref{sigma-tangent-local})--(\ref{alpha-tangent-plus-local}) are satisfied, then we also have
\begin{align}\label{sigma-tangent-local-D-1}
&b^j(z) \in T_z \mathcal{N}, \quad z \in O_{\mathcal{N}}, \quad j \in \mathbb{N},
\\ \label{sigma-tangent-local-D-2} &b^j(z) \in T_z \partial \mathcal{N}, \quad z \in O_{\mathcal{N}} \cap \partial \mathcal{N}, \quad j \in \mathbb{N},
\\ \label{jumps-local-D}
&z + c(z,x) \in C_{\mathcal{N}} \quad \text{for $F$-almost all $x \in E$,} \quad \text{for all $z \in O_{\mathcal{N}}$,}
\\ \label{manifold-boundary-local-D} &\int_{E} |\langle \xi_z, c(z,x) \rangle| F(dx) < \infty, \quad z \in O_{\mathcal{N}} \cap \partial \mathcal{N},
\\ \label{alpha-tangent-local-D} &a(z) - \frac{1}{2} \sum_{j \in \mathbb{N}} Db^j(z)b^j(z)
\\ \notag &\quad- \int_E \Pi_{(T_z \mathcal{N})^{\perp}} c(z,x) F(dx) \in T_z
\mathcal{N}, \quad z \in O_{\mathcal{N}},
\\ \label{alpha-tangent-plus-local-D} &\langle \xi_z,a(z) \rangle - \frac{1}{2} \sum_{j \in \mathbb{N}} \langle \xi_z, D b^j(z) b^j(z) \rangle
\\ \notag &\quad - \int_{E} \langle \xi_z,c(z,x) \rangle F(dx) \geq 0, \quad z \in O_{\mathcal{N}} \cap \partial \mathcal{N}.
\end{align}
\item If we have (\ref{sigma-tangent-local}), (\ref{jumps-local}) and (\ref{alpha-tangent-local}), then we also have
\begin{align}\label{beta-umg}
\beta(h) &= (g_{\lambda}^{\star} a)(h), \quad h \in O_{\mathcal{M}},
\\ \label{sigma-umg} \sigma^j(h) &= (g_W^{\star} b^j)(h), \quad j \in \mathbb{N} \text{ and } h \in O_{\mathcal{M}},
\\ \label{gamma-umg} \gamma(h,x) &= (g_{\mu}^{\star} c)(h,x) \quad \text{for $F$-almost all $x \in E$}, \quad \text{for all $h \in O_{\mathcal{M}}$.}
\end{align}
\end{enumerate}
\end{proposition}

\begin{proof}
This follows from \cite[Propositions~3.23 and 3.33]{FTT-appendix}.
\end{proof}

Now, we consider the $G$-valued SDE
\begin{align}\label{SDE-abc}
\left\{
\begin{array}{rcl}
d Z_t & = & a(Z_t)dt + \sum_{j \in \mathbb{N}} b^j(Z_t)d\beta_t^{j} + \int_E c(Z_{t-},x) (\mu(dt,dx) - F(dx)dt) \medskip
\\ Z_{0} & = & z_0.
\end{array}
\right.
\end{align}
For our subsequent analysis, the following technical definition will be useful.

\begin{definition}\label{def-preloc-f}
The set $O_{\mathcal{M}}$ is called \emph{prelocally invariant in $C_{\mathcal{M}}$ for (\ref{SPDE-j}) with solutions given by (\ref{SDE-abc}) and $f$}, if for all $h_0 \in O_{\mathcal{M}}$ there exists a local strong solution $Z = Z^{(g(h_0))}$ to (\ref{SDE-abc}) with lifetime $\tau > 0$ such that $(Z^{\tau})_- \in O_{\mathcal{N}}$ and $Z^{\tau} \in C_{\mathcal{N}}$ up to an evanescent set and $f(Z)$ is a local mild solution to (\ref{SPDE-j}) with initial condition $h_0$ and lifetime $\tau$.
\end{definition}


\begin{lemma}\label{lemma-inv-transfer}
Suppose $O_{\mathcal{M}}$ is prelocally invariant in $C_{\mathcal{M}}$ for (\ref{SPDE-j}) with solutions given by (\ref{SDE-abc}) and $f$. Then the following statements are true:
\begin{enumerate}
\item $O_{\mathcal{M}}$ is prelocally invariant in $C_{\mathcal{M}}$ for (\ref{SPDE-j}).

\item $O_{\mathcal{N}}$ is prelocally invariant in $C_{\mathcal{N}}$ for (\ref{SDE-abc}).
\end{enumerate}
\end{lemma}

\begin{proof}
This is an immediate consequence of Definitions~\ref{def-prelocal} and \ref{def-preloc-f}.
\end{proof}

The following result shows how the coefficients of locally invariant jump-diffusions translate when we change to another manifold; they are given by the respective pull-backs.

\begin{proposition}\label{prop-Ito}
Let $Z$ be a local strong solution to (\ref{SDE-abc}) for some initial condition $z_0 \in O_{\mathcal{N}}$ with lifetime $\tau > 0$ such that $(Z^{\tau})_- \in O_{\mathcal{N}}$ and $Z^{\tau} \in C_{\mathcal{N}}$ up to an evanescent set. Then $r := f(Z)$ is a local strong solution to the SDE
\begin{align}\label{SDE}
\left\{
\begin{array}{rcl}
dr_t & = & (g_{\lambda}^* a)(r_t)dt + \sum_{j \in \mathbb{N}} (g_W^* b^j)(r_t) d\beta_t^j
\\ && + \int_E (g_{\mu}^* c)(r_{t-},x) (\mu(dt,dx) - F(dx) dt)
\medskip
\\ r_0 & = & h_0
\end{array}
\right.
\end{align}
with initial condition $h_0 = f(z_0)$ and lifetime $\tau$.
\end{proposition}

\begin{proof}
This follows from It\^o's formula for jump-diffusions in infinite dimension; see \cite[Proposition~2.25]{FTT-appendix}.
\end{proof}

If the generator $A$ is continuous, then the just introduced invariance concept transfers to the sets $O_{\mathcal{N}}$ and $C_{\mathcal{N}}$.

\begin{lemma}\label{lemma-inv-sets-equiv}
Suppose $A \in L(H)$. Then the following statements are equivalent:
\begin{enumerate}
\item $O_{\mathcal{M}}$ is prelocally invariant in $C_{\mathcal{M}}$ for (\ref{SPDE-j}) with solutions given by (\ref{SDE-abc}) and $f$.

\item $O_{\mathcal{N}}$ is prelocally invariant in $C_{\mathcal{N}}$ for (\ref{SDE-abc}) with solutions given by (\ref{SPDE-j}) and $g$.
\end{enumerate}
\end{lemma}

\begin{proof}
(1) $\Rightarrow$ (2): Let $z_0 \in O_{\mathcal{N}}$ be arbitrary and set $h_0 := f(z_0) \in O_{\mathcal{M}}$. There exists a local strong solution $Z = Z^{(g(h_0))} = Z^{(z_0)}$ to (\ref{SDE-abc}) with lifetime $\tau > 0$ such that $(Z^{\tau})_- \in O_{\mathcal{N}}$ and $Z^{\tau} \in C_{\mathcal{N}}$ up to an evanescent set, and, since $A \in L(H)$, the process $r = f(Z)$ is a local strong solution to (\ref{SPDE-j}) with initial condition $h_0 = f(z_0)$. Therefore, we have $(r^{\tau})_- \in O_{\mathcal{M}}$ and $r^{\tau} \in C_{\mathcal{M}}$ up to an evanescent set, and $g(r)$ is a local strong solution to (\ref{SDE-abc}) with initial condition $z_0$ and lifetime $\tau$, because $Z^{\tau} = g(r^{\tau})$.

\noindent(2) $\Rightarrow$ (1): This implication is proven analogously.
\end{proof}

\begin{proposition}\label{prop-part-pre}
The following statements are equivalent:
\begin{enumerate}
\item $O_{\mathcal{M}}$ is prelocally invariant in $C_{\mathcal{M}}$ for (\ref{SPDE-j}) with solutions given by (\ref{SDE-abc}) and $f$.

\item $O_{\mathcal{N}}$ is prelocally invariant in $C_{\mathcal{N}}$ for (\ref{SDE-abc}) and we have
\begin{align}\label{dom-local-pre}
O_{\mathcal{M}} &\subset \mathcal{D}(A),
\\ \label{cond-alpha-pb-pre} (A + \alpha)(h) &= (g_{\lambda}^{\star} a)(h) \quad \text{for all $h \in O_{\mathcal{M}}$,}
\\ \label{cond-sigma-pb-pre} \sigma^j(h) &= (g_W^{\star} b^j)(h) \quad \text{for all $j \in \mathbb{N}$,} \quad \text{for all $h \in O_{\mathcal{M}}$,}
\\ \label{cond-gamma-pb-pre} \gamma(h,x) &= (g_{\mu}^{\star} c)(h,x) \quad \text{for $F$-almost all $x \in E$,} \quad \text{for all $h \in O_{\mathcal{M}}$.}
\end{align}
\end{enumerate}
In either case, $A$ is continuous on $O_{\mathcal{M}}$.
\end{proposition}

\begin{proof}
(1) $\Rightarrow$ (2): By Lemma~\ref{lemma-inv-transfer} the set $O_{\mathcal{N}}$ is prelocally invariant in $C_{\mathcal{N}}$ for (\ref{SDE-abc}). Let $h \in O_{\mathcal{M}}$ be arbitrary. Since $O_{\mathcal{M}}$ is prelocally invariant in $C_{\mathcal{M}}$ for (\ref{SPDE-j}) with solutions given by (\ref{SDE-abc}) and $f$, there exists a local strong solution $Z = Z^{(g(h))}$ to (\ref{SDE-abc}) with lifetime $\tau > 0$ such that $(Z^{\tau})_- \in O_{\mathcal{N}}$ and $Z^{\tau} \in C_{\mathcal{N}}$ up to an evanescent set and $r := f(Z)$ is a local mild solution to (\ref{SPDE-j}) with initial condition $h$ and lifetime $\tau$. By Proposition~\ref{prop-Ito} the process $r$ is a local strong solution to (\ref{SDE}) with initial condition $h = f(z)$ and lifetime $\tau$.

Let $\zeta \in \mathcal{D}(A^*)$ be arbitrary. Since $r$ is also a local weak solution to (\ref{SPDE-j}) with lifetime $\tau$, we have $\mathbb{P}$-almost surely
\begin{align*}
\langle \zeta, r_{t \wedge \tau} \rangle &= \langle \zeta,h
\rangle + \int_0^{t \wedge \tau} ( \langle A^* \zeta,
r_s \rangle + \langle \zeta, \alpha(r_s) \rangle ) ds
\\ &\quad + \sum_{j \in \mathbb{N}} \int_0^{t \wedge \tau} \langle \zeta,\sigma^j(r_s) \rangle d\beta_s^j
\\ &\quad + \int_0^{t \wedge \tau} \int_E \langle \zeta,
\gamma(r_{s-},x) \rangle (\mu(ds,dx) -
F(dx)ds), \quad t \geq 0.
\end{align*}
Therefore, we get up to an evanescent set
\begin{align*}
B + M^c + M^d = 0,
\end{align*}
where the processes $B$, $M^c$, $M^d$ are given by
\begin{align*}
B_t &:= \int_0^{t \wedge \tau} \big( \langle A^* \zeta, r_s
\rangle + \langle \zeta, \alpha(r_s) - (g_{\lambda}^{\star} a)(r_s) \rangle \big) ds,
\\ M_t^c &:= \sum_{j \in \mathbb{N}} \int_0^{t \wedge \tau} \langle \zeta,\sigma^j(r_s)
- (g_W^{\star} b^j)(r_s) \rangle
d\beta_s^j,
\\ M_t^d &:= \int_0^{t \wedge \tau} \int_E \langle \zeta,
\gamma(r_{s-},x) -
(g_{\mu}^{\star} c)(r_{s-},x) \rangle (\mu(ds,dx) - F(dx)ds).
\end{align*}
The process $B$ is a finite variation process which is continuous, and hence predictable, $M^c$ is a continuous square-integrable martingale and $M^d$ is a purely discontinuous square-integrable martingale. Therefore $B + M^c + M^d$ is a special semimartingale. Since the decomposition $B + M$ of a special semimartingale into a finite variation process $B$ and a local martingale $M$ is unique (see \cite[Corollary~I.3.16]{Jacod-Shiryaev}) and the decomposition of a local martingale $M = M^c + M^d$ into a continuous local martingale $M^c$ and a purely discontinuous local martingale $M^d$ is unique (see \cite[Theorem~I.4.18]{Jacod-Shiryaev}), we deduce that $B = M^c = M^d = 0$ up to an evanescent set. By the It\^{o} isometry, we obtain $\mathbb{P}$-almost surely
\begin{align}\label{integrand-1}
&\int_0^{t \wedge \tau} \big( \langle A^* \zeta, r_s
\rangle + \langle \zeta, \alpha(r_s) - (g_{\lambda}^{\star} a)(r_s) \rangle \big) ds = 0, \quad t \geq 0,
\\ \label{integrand-2} &\int_0^{t \wedge \tau} \bigg( \sum_{j \in \mathbb{N}} |\langle \zeta,\sigma^j(r_s)
- (g_{W}^{\star} b^j)(r_s) \rangle|^2 \bigg)
ds = 0, \quad t \geq 0,
\\ \label{integrand-3} &\int_0^{t \wedge \tau} \bigg( \int_E |\langle \zeta,
\gamma(r_{s-},x) -
(g_{\mu}^{\star} c)(r_{s-},x) \rangle|^2
F(dx) \bigg) ds = 0, \quad t \geq 0.
\end{align}
Since the process $r$ is c\`{a}dl\`{a}g, by Lemma~\ref{lemma-Lipschitz-pb} and Lebesgue's dominated convergence theorem (applied to the sum $\sum_{j \in \mathbb{N}}$ and to the integral $\int_E$) the integrands appearing in (\ref{integrand-1})--(\ref{integrand-3}) are continuous in $s = 0$, and hence, we get
\begin{align}\label{id-alpha}
&\langle A^* \zeta, h
\rangle + \langle \zeta, \alpha(h) - (g_{\lambda}^{\star} a)(h) \rangle = 0,
\\ \label{id-sigma} &\sum_{j \in \mathbb{N}} |\langle \zeta,\sigma^j(h)
- (g_W^{\star} b^j)(h) \rangle|^2 = 0,
\\ \label{id-gamma} &\int_E |\langle \zeta,
\gamma(h,x) -
(g_{\mu}^{\star} c)(h,x) \rangle|^2 F(dx) = 0.
\end{align}
Identity (\ref{id-alpha}) shows that $\zeta \mapsto \langle A^* \zeta, h \rangle$ is continuous on $\mathcal{D}(A^*)$, proving $h \in \mathcal{D}(A^{**})$. Since $A = A^{**}$, see \cite[Theorem~13.12]{Rudin}, we obtain $h \in \mathcal{D}(A)$, which yields
(\ref{dom-local-pre}). Using the identity $\langle A^* \zeta, h \rangle = \langle \zeta, Ah \rangle$, we obtain
\begin{align*}
&\langle \zeta, A h + \alpha(h) - (g_{\lambda}^{\star} a)(h) \rangle = 0 \quad \text{for all $\zeta \in \mathcal{D}(A^*)$,}
\end{align*}
and hence (\ref{cond-alpha-pb-pre}). For an arbitrary $j \in \mathbb{N}$ we obtain, by using (\ref{id-sigma}),
\begin{align*}
\langle \zeta,\sigma^j(h)
- (g_W^{\star} b^j)(h) \rangle = 0 \quad \text{for all $\zeta \in \mathcal{D}(A^*)$,}
\end{align*}
showing (\ref{cond-sigma-pb-pre}). By (\ref{id-gamma}), for all $\zeta \in \mathcal{D}(A^*)$ we have
\begin{align*}
\langle \zeta,
\gamma(h,x) -
(g_{\mu}^{\star} c)(h,x) \rangle = 0 \quad \text{for $F$-almost all $x \in E$.}
\end{align*}
Using Lemma~\ref{lemma-jumps-exact}, for $F$-almost all $x \in E$ we obtain
\begin{align*}
\langle \zeta,
\gamma(h,x) -
(g_{\mu}^{\star} c)(h,x) \rangle = 0 \quad \text{for all $\zeta \in \mathcal{D}(A^*)$,}
\end{align*}
which proves (\ref{cond-gamma-pb-pre}).

\noindent(2) $\Rightarrow$ (1): Let $h_0 \in O_{\mathcal{M}}$ be arbitrary. Since $O_{\mathcal{N}}$ is prelocally invariant in $C_{\mathcal{N}}$ for (\ref{SDE-abc}), there exists a local strong solution $Z = Z^{(g(h_0))}$ to (\ref{SDE-abc}) with lifetime $\tau > 0$ such that $(Z^{\tau})_- \in O_{\mathcal{N}}$ and $Z^{\tau} \in C_{\mathcal{N}}$ up to an evanescent set. By Proposition~\ref{prop-Ito} and conditions (\ref{dom-local-pre})--(\ref{cond-gamma-pb-pre}), the process $r := f(Z)$ is a local strong solution to
(\ref{SPDE-j}) with initial condition $h_0$ and lifetime $\tau$, showing that $O_{\mathcal{M}}$ is prelocally invariant in $C_{\mathcal{M}}$ for (\ref{SPDE-j}) with solutions given by (\ref{SDE-abc}) and $f$.

\noindent\emph{Additional Statement:} If conditions (\ref{dom-local-pre}), (\ref{cond-alpha-pb-pre}) are satisfied, then we have
\begin{align*}
Ah = (g_{\lambda}^{\star} a)(h) - \alpha(h), \quad h \in O_{\mathcal{M}},
\end{align*}
and hence, the continuity of $A$ on $O_{\mathcal{M}}$ follows from Lemma~\ref{lemma-Lipschitz-pb}.
\end{proof}

For the rest of this section, let $G = \mathbb{R}^m$, where $m \in \mathbb{N}$ denotes the dimension of the submanifold $\mathcal{M}$. We assume there exist elements $\zeta_1,\ldots,\zeta_m \in \mathcal{D}(A^*)$
such that the mapping $f : \mathcal{N} \rightarrow \mathcal{M}$ has the inverse (\ref{f-inverse-conv}), that is, diagram (\ref{diagram}) commutes.

We define the subsets $O_{V} \subset C_{V} \subset V$ by $O_{V} := \psi^{-1}(O_{\mathcal{N}})$ and $C_{V} := \psi^{-1}(C_{\mathcal{N}})$.
Recall that $O_{\mathcal{M}}$ is open in $\mathcal{M}$ and $C_{\mathcal{M}}$ is compact. Since $f : \mathcal{N} \rightarrow \mathcal{M}$ is a homeomorphism, $O_{\mathcal{N}}$ is open in $\mathcal{N}$ and $C_{\mathcal{N}}$ is compact. Furthermore, since $\psi : V \rightarrow \mathcal{N}$ is a homeomorphism, $O_{V}$ is open in $V$ and $C_{V}$ is compact.
We define the mappings for the $\mathbb{R}^m$-valued SDE (\ref{SDE-abc}) as
\begin{align}\label{def-a-rem}
a &:= \langle A^* \zeta, f \rangle + (f,\langle \zeta,\bullet \rangle)_{\lambda}^{\star} \alpha : \mathbb{R}^m \rightarrow \mathbb{R}^m,
\\ \label{def-b-rem} b^j &:= (f,\langle \zeta,\bullet \rangle)_W^{\star} \sigma^j : \mathbb{R}^m \rightarrow \mathbb{R}^m \quad \text{for $j \in \mathbb{N}$,}
\\ \label{def-c-rem} c &:= (f,\langle \zeta,\bullet \rangle)_{\mu}^{\star} \gamma : \mathbb{R}^m \times E \rightarrow \mathbb{R}^m,
\end{align}
where $\langle A^* \zeta, f \rangle := (\langle A^* \zeta_1, f \rangle,\ldots,\langle A^* \zeta_m, f \rangle)$. Then for each $h \in O_{\mathcal{M}}$ we have
\begin{align}\label{a-linear}
a(z) &= \langle A^* \zeta, h \rangle + \langle \zeta, \alpha(h) \rangle,
\\ \label{b-linear} b^j(z) &= \langle \zeta, \sigma^j(h) \rangle, \quad j \in \mathbb{N}
\\ \label{c-linear} c(z,x) &= \langle \zeta, \gamma(h,x) \rangle, \quad x \in E
\end{align}
where $z = \langle \zeta,h \rangle \in O_{\mathcal{N}}$.
Furthermore, we define the mappings
\begin{align}\label{def-Theta-rem}
\Theta &:= (\psi,\Psi)_{\lambda}^{\star} a : \mathbb{R}^m \rightarrow \mathbb{R}^m,
\\ \label{def-Sigma-rem} \Sigma^j &:= (\psi,\Psi)_{W}^{\star} b^j : \mathbb{R}^m \rightarrow \mathbb{R}^m, \quad \text{for $j \in \mathbb{N}$,}
\\ \label{def-Gamma-rem} \Gamma &:= (\psi,\Psi)_{\mu}^{\star} c : \mathbb{R}^m \times E \rightarrow \mathbb{R}^m
\end{align}
and consider the $\mathbb{R}^m$-valued SDE (\ref{SDE-Theta}).
According to Lemma~\ref{lemma-Lipschitz-pb}, the mappings $a$, $(b^j)_{j \in \mathbb{N}}$, $c$ as well as $\Theta$, $(\Sigma^j)_{j \in \mathbb{N}}$, $\Gamma$ satisfy the regularity conditions (\ref{Lipschitz-alpha-st})--(\ref{linear-growth-sigma}) and (\ref{Lipschitz-gamma-rho})--(\ref{sigma-C1}). Note that
\begin{align}\label{Theta-trans}
\Theta(y) &= (\psi_{\lambda}^{\star} a)(y), \quad y \in O_V
\\ \label{Sigma-trans} \Sigma^j(y) &= (\psi_W^{\star} b^j)(y), \quad j \in \mathbb{N} \text{ and } y \in O_V
\\ \label{Gamma-trans} \Gamma(y,x) &= (\psi_{\mu}^{\star} c)(y,x), \quad x \in E \text{ and } y \in O_V.
\end{align}
Note that $V$ is a $m$-dimensional $C^3$-submanifold with boundary of $\mathbb{R}^m$, and that for $y \in \partial V$ the inward pointing normal vector to $\partial V$ at $y$ is given by the first unit vector $e_1 = (1,0,\ldots,0) \in \mathbb{R}^m$. Therefore, Theorem~\ref{thm-SDE-half-space} together with Lemma~\ref{lemma-stratonovich} provides the statement of Theorem~\ref{thm-invariance-partition} for the particular case, where the submanifold is an open subset in the half space $\mathbb{R}_+^m$.

\begin{lemma}\label{lemma-weak-preloc}
Suppose that $O_{\mathcal{M}}$ is prelocally invariant in $C_{\mathcal{M}}$ for (\ref{SPDE-j}). Then the set $O_{\mathcal{M}}$ is prelocally invariant in $C_{\mathcal{M}}$ for (\ref{SPDE-j}) with solutions given by (\ref{SDE-abc}) and~$f$.
\end{lemma}

\begin{proof}
Let $h_0 \in O_{\mathcal{M}}$ be arbitrary. Since $O_{\mathcal{M}}$ is prelocally invariant in $C_{\mathcal{M}}$ for (\ref{SDE-abc}), there exists a local mild solution $r = r^{(h_0)}$ to (\ref{SPDE-j}) with lifetime $\tau > 0$ such that $(r^{\tau})_- \in O_{\mathcal{M}}$ and $r^{\tau} \in C_{\mathcal{M}}$ up to an evanescent set. Since $\zeta_1,\ldots,\zeta_m \in
\mathcal{D}(A^*)$ and $r$ is also a local weak solution to (\ref{SPDE-j}), setting $Z := \langle \zeta,r \rangle$ we have, by taking into account (\ref{a-linear})--(\ref{c-linear}), $\mathbb{P}$-almost surely
\begin{align*}
Z_{t \wedge \tau} &= \langle \zeta, r_{t \wedge \tau} \rangle = \langle
\zeta,h_0 \rangle + \int_{0}^{t \wedge \tau} ( \langle A^* \zeta,
r_s \rangle + \langle \zeta, \alpha(r_s)
\rangle ) ds
\\ &\quad + \sum_{j \in \mathbb{N}} \int_{0}^{t \wedge \tau} \langle \zeta,\sigma^j(r_s) \rangle d\beta_s^{j}
+ \int_{0}^{t \wedge \tau} \int_E \langle \zeta,
\gamma(r_{s-},x) \rangle (\mu(ds,dx) -
F(dx)ds)
\\ &= \langle
\zeta,h_0 \rangle + \int_{0}^{t \wedge \tau} a(Z_s)ds +
\sum_{j \in \mathbb{N}} \int_{0}^{t \wedge \tau} b^j(Z_s)d\beta_s^{j}
\\ &\quad + \int_{0}^{t \wedge \tau} \int_E c(Z_{s-},x) (\mu(ds,dx) - F(dx)ds), \quad t \geq 0.
\end{align*}
Therefore, the process $Z$ is a local strong solution to (\ref{SDE-abc}) with initial condition $\langle
\zeta,h_0 \rangle$ and lifetime $\tau$ such that $(Z^{\tau})_- \in O_{\mathcal{N}}$ and $Z^{\tau} \in C_{\mathcal{N}}$ up to an evanescent set. By (\ref{f-inverse-conv}) we have $f(Z^{\tau}) = r^{\tau}$, and hence, the process $f(Z)$ is a local mild solution to (\ref{SPDE-j}) with initial condition $h_0$ and lifetime $\tau$.
\end{proof}

Now, we are ready to provide the proof of Theorem~\ref{thm-invariance-partition}.

\begin{proof}[Proof of Theorem~\ref{thm-invariance-partition}]
(1) $\Rightarrow$ (2): By Lemma~\ref{lemma-weak-preloc}, the set $O_{\mathcal{M}}$ is prelocally invariant in $C_{\mathcal{M}}$ for (\ref{SPDE-j}) with solutions given by (\ref{SDE-abc}) and $f$. Therefore, we have two implications:
\begin{itemize}
\item Proposition~\ref{prop-part-pre} yields (\ref{domain-local}) and
\begin{align}\label{cond-alpha-pb}
(A + \alpha)(h) &= (\langle \zeta,\bullet \rangle_{\lambda}^{\star} a)(h), \quad h \in O_{\mathcal{M}},
\\ \label{cond-sigma-pb} \sigma^j(h) &= (\langle \zeta,\bullet \rangle_W^{\star} b^j)(h) \quad \text{for all $j \in \mathbb{N}$,} \quad \text{for all $h \in O_{\mathcal{M}}$,}
\\ \label{cond-gamma-pb} \gamma(h,x) &= (\langle \zeta,\bullet \rangle_{\mu}^{\star} c)(h,x) \quad \text{for $F$-almost all $x \in E$,} \quad \text{for all $h \in O_{\mathcal{M}}$.}
\end{align}
\item By Lemma~\ref{lemma-inv-transfer}, the set $O_{\mathcal{N}}$ is prelocally invariant in $C_{\mathcal{N}}$ for (\ref{SDE-abc}). Hence, by (\ref{Theta-trans})--(\ref{Gamma-trans}) and Proposition~\ref{prop-part-pre}, the set $O_V$ is prelocally invariant in $C_V$ for (\ref{SDE-Theta}) with solutions given by (\ref{SDE-abc}) and $\Psi$.
\end{itemize}
The latter statement has two further consequences:
\begin{itemize}
\item By Lemma~\ref{lemma-inv-sets-equiv}, the set $O_{\mathcal{N}}$ is prelocally invariant in $C_{\mathcal{N}}$ for (\ref{SDE-abc}) with solutions given by (\ref{SDE-Theta}) and $\psi$. Thus, Proposition~\ref{prop-part-pre} yields
\begin{align}\label{a-Psi-trans}
a(z) &= (\Psi_{\lambda}^{\star} \Theta)(z), \quad z \in O_{\mathcal{N}},
\\ \label{b-Psi-trans} b^j(z) &= (\Psi_W^{\star} \Sigma^j)(z), \quad j \in \mathbb{N} \text{ and } z \in O_{\mathcal{N}},
\\ \label{c-Psi-trans} c(z,x) &= (\Psi_{\mu}^{\star} \Gamma)(z,x) \quad \text{for $F$-almost all $x \in E$,} \quad \text{for all $z \in O_{\mathcal{N}}$.}
\end{align}
\item By Lemma~\ref{lemma-inv-transfer}, the set $O_{V}$ is prelocally invariant in $C_{V}$ for (\ref{SDE-Theta}). Theorem~\ref{thm-SDE-half-space} implies that conditions (\ref{cond-Sigma-Y})--(\ref{cond-Theta-Y}) are satisfied.
\end{itemize}
In view of (\ref{cond-Sigma-Y})--(\ref{cond-Theta-Y}), Lemma~\ref{lemma-stratonovich}, identities (\ref{a-Psi-trans})--(\ref{c-Psi-trans}) and Proposition~\ref{prop-cond-imply-cond} we obtain (\ref{sigma-tangent-local-D-1})--(\ref{alpha-tangent-plus-local-D}), where $\xi_z$ denotes the inward pointing normal vector to $\partial \mathcal{N}$ at $z$.
Taking into account (\ref{cond-alpha-pb})--(\ref{cond-gamma-pb}), applying Proposition~\ref{prop-cond-imply-cond} we arrive at (\ref{sigma-tangent-local})--(\ref{alpha-tangent-plus-local}).

\noindent(2) $\Rightarrow$ (1): Suppose that conditions (\ref{domain-local})--(\ref{alpha-tangent-plus-local}) are satisfied. By (\ref{domain-local}) and (\ref{a-linear}), for all $z \in O_{\mathcal{N}}$ we obtain
\begin{align*}
a(z) = \langle A^* \zeta, h \rangle + \langle \zeta, \alpha(h) \rangle = \langle \zeta, Ah + \alpha(h) \rangle = (f_{\lambda}^{\star}(A + \alpha))(z),
\end{align*}
where $h = f(z) \in O_{\mathcal{M}}$. Thus, we have
\begin{align}\label{a-zeta}
a(z) &= (f_{\lambda}^{\star} (A + \alpha))(z), \quad z \in O_{\mathcal{N}},
\\ \label{b-zeta} b^j(z) &= (f_W^{\star} \sigma^j)(z), \quad j \in \mathbb{N} \text{ and } z \in O_{\mathcal{N}},
\\ \label{c-zeta} c(z,x) &= (f_{\mu}^{\star} \gamma)(z,x), \quad x \in E \text{ and } z \in O_{\mathcal{N}},
\end{align}
which has two implications:
\begin{itemize}
\item By (\ref{sigma-tangent-local}), (\ref{jumps-local}), (\ref{alpha-tangent-local}) and Proposition~\ref{prop-cond-imply-cond} we obtain (\ref{cond-alpha-pb})--(\ref{cond-gamma-pb}).

\item By (\ref{sigma-tangent-local})--(\ref{alpha-tangent-plus-local}) and Proposition~\ref{prop-cond-imply-cond} we have (\ref{sigma-tangent-local-D-1})--(\ref{alpha-tangent-plus-local-D}).
\end{itemize}
In view of (\ref{Theta-trans})--(\ref{Gamma-trans}), we obtain the following consequences:
\begin{itemize}
\item By (\ref{sigma-tangent-local-D-1}), (\ref{jumps-local-D}), (\ref{alpha-tangent-local-D}) and Proposition~\ref{prop-cond-imply-cond} we obtain (\ref{a-Psi-trans})--(\ref{c-Psi-trans}).

\item By (\ref{sigma-tangent-local-D-1})--(\ref{alpha-tangent-plus-local-D}), Proposition~\ref{prop-cond-imply-cond} and Lemma~\ref{lemma-stratonovich} we have (\ref{cond-Sigma-Y})--(\ref{cond-Theta-Y}).
\end{itemize}
Therefore, by Theorem~\ref{thm-SDE-half-space}, the set $O_V$ is prelocally invariant in $C_V$ for (\ref{SDE-Theta}). By (\ref{a-Psi-trans})--(\ref{c-Psi-trans}) and Proposition~\ref{prop-part-pre}, the set $O_{\mathcal{N}}$ is prelocally invariant in $C_{\mathcal{N}}$ for (\ref{SDE-abc}) with with solutions given by (\ref{SDE-Theta}) and $\psi$. According to Lemma~\ref{lemma-inv-transfer}, the set $O_{\mathcal{N}}$ is prelocally invariant in $C_{\mathcal{N}}$ for (\ref{SDE-abc}). By (\ref{domain-local}), (\ref{cond-alpha-pb})--(\ref{cond-gamma-pb}) and Proposition~\ref{prop-part-pre}, the set $O_{\mathcal{M}}$ is prelocally invariant in $C_{\mathcal{M}}$ for (\ref{SPDE-j}) with solutions given by (\ref{SDE-abc}) and $f$.

\noindent\emph{Additional Statement:} If $O_{\mathcal{M}}$ is prelocally invariant in $C_{\mathcal{M}}$ for (\ref{SPDE-j}) with solutions given by (\ref{SDE-abc}) and $f$, then Proposition~\ref{prop-part-pre} implies that $A$ is continuous on $O_{\mathcal{M}}$. Using Lemma~\ref{lemma-cont-statement}, we obtain that the mapping in (\ref{alpha-tangent-local}) is continuous on $O_{\mathcal{M}}$.
\end{proof}

\section{Global analysis of the invariance problem on submanifolds with boundary and proofs of the main results}\label{sec-proof-main}

In this section, we perform global analysis of the invariance problem and prove our main results. The idea is to localize the invariance problem and to apply Theorem~\ref{thm-invariance-partition} from the previous section. In order to realize this idea, we will switch between the original SPDE (\ref{SPDE-j}) and the SPDE (\ref{SPDE-cutted-app-j}), which only makes sufficiently small jumps. 

Before we start with the proofs of our main results, we prepare some auxiliary results. Let $B \in \mathcal{E}$ be a set with $F(B^c) < \infty$.

\begin{lemma}
The mappings $\alpha^{B} : H \rightarrow H$ and $\gamma^{B} : H \times E \rightarrow H$ defined as
\begin{align}\label{def-alpha-B-app}
\alpha^{B}(h) &:= \alpha(h) - \int_{B^c}
\gamma(h,x) F(dx),
\\ \label{def-gamma-B-app} \gamma^{B}(h,x) &:= \gamma(h,x) \mathbbm{1}_{B}(x)
\end{align}
also satisfy the regularity conditions (\ref{Lipschitz-alpha-st}), (\ref{Lipschitz-gamma-rho}), (\ref{linear-growth-rho}).
\end{lemma}

\begin{proof}
See \cite[Lemma~2.18]{FTT-appendix}.
\end{proof}

Now, we consider the SPDE
\begin{align}\label{SPDE-cutted-app-j}
\left\{
\begin{array}{rcl}
dr_t^{B} & = & ( A r_t^{B} + \alpha^{B}(r_t^{B}) )dt +
\sum_{j \in \mathbb{N}} \sigma^j(r_t^{B})d\beta_t^{j}
\\ && + \int_{E}
\gamma^{B}(r_{t-}^{B},x) (\mu(dt,dx) - F(dx)dt)
\medskip
\\ r_0^{B} & = & h_0.
\end{array}
\right.
\end{align}
We define $\varrho^B$ as the first time where the Poisson random measure makes a jump outside $B$; that is
\begin{align*}
\varrho^B = \inf \{ t \geq 0 : \mu([0,t] \times B^c) = 1 \}
\end{align*}

\begin{lemma}
The mapping $\varrho^B$ is a strictly positive stopping time.
\end{lemma}

\begin{proof}
See \cite[Lemma~2.20]{FTT-appendix}.
\end{proof}

The following result shows that the SPDEs (\ref{SPDE-j}) and (\ref{SPDE-cutted-app-j}) locally have the same mild solutions.

\begin{proposition}\label{prop-r-equal-rB}
Let $h_0 : \Omega \rightarrow H$ be a $\mathcal{F}_0$-measurable random variable, let $B \in \mathcal{E}$ be a set with $F(B^c) < \infty$, and let $0 < \tau \leq \varrho^B$ be a stopping time. Then the following statements are true:
\begin{enumerate}
\item If there exists a local mild solution $r$ to (\ref{SPDE-j}) with lifetime $\tau$, then there also exists a local mild solution $r^B$ to (\ref{SPDE-cutted-app-j}) with lifetime $\tau$ such that
\begin{align}\label{r-rB-stoch-interval}
r^{\tau} \mathbbm{1}_{[\![ 0,\tau [\![ } = (r^B)^{\tau} \mathbbm{1}_{[\![ 0,\tau [\![ }.
\end{align}

\item If there exists a local mild solution $r^B$ to (\ref{SPDE-cutted-app-j}) with lifetime $\tau$, then there also exists a local mild solution $r$ to (\ref{SPDE-j}) with lifetime $\tau$ such that (\ref{r-rB-stoch-interval}) is satisfied.
\end{enumerate}
In particular, in either case we have $(r^{\tau})_- = ((r^B)^{\tau})_-$.
\end{proposition}

\begin{proof}
See \cite[Proposition~2.21]{FTT-appendix}.
\end{proof}

Recall that $\mathcal{M}$ denotes a $C^3$-submanifold with boundary of $H$. The following result shows how the invariance conditions regarding $\alpha,\sigma^j,\gamma$ and $\alpha^B,\sigma^j,\gamma^B$ are related.

\begin{proposition}\label{prop-inv-cond-B}
Let $O_{\mathcal{M}} \subset \mathcal{M}$ be a subset which is open in $\mathcal{M}$, and suppose that
\begin{align*}
&O_{\mathcal{M}} \subset \mathcal{D}(A),
\\ &h + \gamma(h,x) \in \overline{\mathcal{M}} \quad \text{for $F$-almost all $x \in E$}, \quad \text{for all $h \in O_{\mathcal{M}}$.}
\end{align*}
Then the following statements are true:
\begin{enumerate}
\item We have (\ref{manifold-boundary-local})--(\ref{alpha-tangent-plus-local}) if and only if
\begin{align}\label{manifold-boundary-local-B}
&\int_{E} |\langle \eta_h, \gamma^B(h,x) \rangle| F(dx) < \infty, \quad h \in O_{\mathcal{M}} \cap \partial \mathcal{M},
\\ \label{alpha-tangent-local-B} &Ah + \alpha^B(h) - \frac{1}{2} \sum_{j \in \mathbb{N}} D\sigma^j(h)\sigma^j(h)
\\ \notag &\quad -\int_E \Pi_{(T_h \mathcal{M})^{\perp}} \gamma^B(h,x) F(dx) \in T_h
\mathcal{M}, \quad h \in O_{\mathcal{M}},
\\ \label{alpha-tangent-plus-local-B} &\langle \eta_h,Ah + \alpha^B(h) \rangle - \frac{1}{2} \sum_{j \in \mathbb{N}} \langle \eta_h, D \sigma^j(h) \sigma^j(h) \rangle
\\ \notag &\quad - \int_{E} \langle \eta_h,\gamma^B(h,x) \rangle F(dx) \geq 0, \quad h \in O_{\mathcal{M}} \cap \partial
\mathcal{M}.
\end{align}

\item The mapping in (\ref{alpha-tangent-local}) is continuous on $O_{\mathcal{M}}$ if and only if the mapping in (\ref{alpha-tangent-local-B}) is continuous on $O_{\mathcal{M}}$.
\end{enumerate}
\end{proposition}

\begin{proof}
This follows from \cite[Lemma~3.27 and Proposition~3.19]{FTT-appendix}.
\end{proof}

The following auxiliary result shows that for each $h_0 \in \mathcal{M}$ there exists a neighborhood of $h_0$ such that the assumptions from Section~\ref{sec-proof-main-one-chart} are fulfilled, and that the global jump condition (\ref{jumps}) can be localized by choosing the set $B \in \mathcal{E}$ for $\gamma^B$ appropriately.

\begin{proposition}\label{prop-jumps-small}
Suppose that condition (\ref{jumps}) is satisfied. Then, for all $h_0 \in \mathcal{M}$ there exist
\begin{enumerate}
\item[(i)] a constant $\epsilon > 0$ such that $B_{\epsilon}(h_0) \cap \mathcal{M}$ is a submanifold as in Section~\ref{sec-proof-main-one-chart}, i.e., diagram (\ref{diagram}) commutes,

\item[(ii)] subsets $O_{\mathcal{M}} \subset C_{\mathcal{M}} \subset B_{\epsilon}(h_0) \cap \mathcal{M}$ with $h_0 \in O_{\mathcal{M}}$ as in Section~\ref{sec-proof-main-one-chart}, i.e., $O_{\mathcal{M}}$ is open in $B_{\epsilon}(h_0) \cap \mathcal{M}$ and $C_{\mathcal{M}}$ is compact,

\item[(iii)] and a set $B \in \mathcal{E}$ with $F(B^c) < \infty$
\end{enumerate}
such that we have
\begin{align}\label{jumps-in-C-in-Lemma}
h + \gamma^B(h,x) \in C_{\mathcal{M}} \quad \text{for $F$-almost all $x \in E$,} \quad \text{for all $h \in O_{\mathcal{M}}$.}
\end{align}
\end{proposition}

\begin{proof}
This follows from \cite[Proposition~3.11 and Lemma~3.15]{FTT-appendix}.
\end{proof}

Finally, we require the following result about the existence of strong solutions to (\ref{SPDE-j}) under stochastic invariance.

\begin{lemma}\label{lemma-strong-solution}
Suppose that $\mathcal{M} \subset \mathcal{D}(A)$ and that $A$
is continuous on $\mathcal{M}$. Let $h_0 \in \mathcal{M}$ be arbitrary, and let $r = r^{(h_0)}$ be a local weak solution to (\ref{SPDE-j}) with initial condition $h_0$ lifetime $\tau > 0$ such that $(r^{\tau})_- \in \mathcal{M}$ up to an evanescent set. Then $r$ is also a local strong solution to (\ref{SPDE-j}) with lifetime $\tau$.
\end{lemma}

\begin{proof}
See \cite[Lemma~2.7]{FTT-appendix}.
\end{proof}

Now, we are ready to provide the proofs of our main results.

\begin{proof}[Proof of Theorem~\ref{thm-main-1}]
(1) $\Rightarrow$ (2): We will prove that prelocal invariance of $\mathcal{M}$ for (\ref{SPDE-j}) implies conditions (\ref{domain})--(\ref{jumps}), (\ref{manifold-boundary})--(\ref{alpha-tangent-plus}), the continuity of $A$ and the mapping in (\ref{alpha-tangent}) on $\mathcal{M}$, and and that for each $h_0 \in \mathcal{M}$ there is a local strong solution $r = r^{(h_0)}$ to (\ref{SPDE-j}). 

According to Lemma~\ref{lemma-jumps-global-set} we have (\ref{jumps}). Let $h_0 \in \mathcal{M}$ be arbitrary. By Proposition~\ref{prop-jumps-small} there exist quantities as in (i)--(iii) such that condition (\ref{jumps-in-C-in-Lemma}) is satisfied. 

We will show that $O_{\mathcal{M}}$ is prelocally invariant in $C_{\mathcal{M}}$ for (\ref{SPDE-cutted-app-j}). Indeed, let $g_0 \in O_{\mathcal{M}}$ be arbitrary. Since $O_{\mathcal{M}}$ is open in $B_{\epsilon}(h_0) \cap \mathcal{M}$, there exists $\delta > 0$ such that $\overline{B_{\delta}(g_0)} \cap \mathcal{M} \subset O_{\mathcal{M}}$. Since $\mathcal{M}$ is prelocally invariant for (\ref{SPDE-j}), there exist a local mild solution $r = r^{(g_0)}$ to (\ref{SPDE-j}) with lifetime $0 < \tau \leq \varrho^B$ such that $(r^{\tau})_- \in \mathcal{M}$ up to an evanescent set. According to Proposition~\ref{prop-r-equal-rB}, there exists a local mild solution $r^B = r^{B,(g_0)}$ to (\ref{SPDE-cutted-app-j}) with lifetime $\tau$ such that $(r^{\tau})_- = ((r^B)^{\tau})_-$. The mapping
\begin{align*}
\varrho :=
\inf \{ t \geq 0 : r_t \notin B_{\delta}(g_0) \} \wedge \tau
\end{align*}
is a strictly positive stopping time, and we obtain up to an evanescent set
\begin{align*}
((r^B)^{\varrho})_- = (r^{\varrho})_- \in \overline{B_{\delta}(g_0)} \cap \mathcal{M} \subset O_{\mathcal{M}}.
\end{align*}
Furthermore, using (\ref{jumps-in-C-in-Lemma}) and Lemma~\ref{lemma-jump-global-help-2} we obtain
$(r^B)^{\varrho} \in C_{\mathcal{M}}$ up to an evanescent set.
Hence, the set $O_{\mathcal{M}}$ is prelocally invariant in $C_{\mathcal{M}}$ for (\ref{SPDE-cutted-app-j}).

Theorem~\ref{thm-invariance-partition}, applied to the SPDE (\ref{SPDE-cutted-app-j}), yields (\ref{domain-local})--(\ref{sigma-tangent-local-2}), (\ref{manifold-boundary-local-B})--(\ref{alpha-tangent-plus-local-B}) and that $A$ and the mapping in (\ref{alpha-tangent-local-B}) are continuous on $O_{\mathcal{M}}$. Since (\ref{domain-local}) and (\ref{jumps}) are satisfied, by Proposition~\ref{prop-inv-cond-B} we also have (\ref{manifold-boundary-local})--(\ref{alpha-tangent-plus-local}) and the mapping in (\ref{alpha-tangent-local}) is continuous on $O_{\mathcal{M}}$. Since $h_0 \in \mathcal{M}$ was arbitrary, we deduce (\ref{domain}), (\ref{sigma-tangent}), (\ref{manifold-boundary})--(\ref{alpha-tangent-plus}) and that $A$ and the mapping in (\ref{alpha-tangent}) are continuous on $\mathcal{M}$. By Lemma~\ref{lemma-strong-solution}, for each $h_0 \in \mathcal{M}$ there is a local strong solution $r = r^{(h_0)}$ to (\ref{SPDE-j}).

\noindent(2) $\Rightarrow$ (1): Now, we will prove that conditions (\ref{domain})--(\ref{jumps}) and (\ref{manifold-boundary})--(\ref{alpha-tangent-plus}) imply prelocal invariance of $\mathcal{M}$ for (\ref{SPDE-j}) and the statement regarding local invariance. 

Let $h_0 \in \mathcal{M}$ be arbitrary. By Proposition~\ref{prop-jumps-small} there exist quantities as in (i)--(iii) such that condition (\ref{jumps-in-C-in-Lemma}) is satisfied. 

We will show that $C_{\mathcal{M}}$ is prelocally invariant in $O_{\mathcal{M}}$ for (\ref{SPDE-cutted-app-j}). By (\ref{domain}), (\ref{sigma-tangent}) and (\ref{manifold-boundary})--(\ref{alpha-tangent-plus}) we have (\ref{domain-local})--(\ref{sigma-tangent-local-2}) and (\ref{manifold-boundary-local})--(\ref{alpha-tangent-plus-local}). Since (\ref{domain-local}) and (\ref{jumps}) are satisfied, by Proposition~\ref{prop-inv-cond-B} we also have (\ref{manifold-boundary-local-B})--(\ref{alpha-tangent-plus-local-B}). Consequently, by (\ref{domain-local})--(\ref{sigma-tangent-local-2}), (\ref{jumps-in-C-in-Lemma}), (\ref{manifold-boundary-local-B})--(\ref{alpha-tangent-plus-local-B}) and Theorem~\ref{thm-invariance-partition}, the set $C_{\mathcal{M}}$ is prelocally invariant in $O_{\mathcal{M}}$ for (\ref{SPDE-cutted-app-j}).

Now, we will show that $\mathcal{M}$ is prelocally invariant for (\ref{SPDE-j}).
Since $C_{\mathcal{M}}$ is prelocally invariant in $O_{\mathcal{M}}$ for (\ref{SPDE-cutted-app-j}), there exists a local mild solution $r^{B}$ to (\ref{SPDE-cutted-app-j}) with lifetime $0 < \tau \leq \varrho^B$ such that up to an evanescent set
\begin{align*}
((r^{B})^{\tau})_- \in O_{\mathcal{M}} \quad \text{and} \quad (r^{B})^{\tau} \in C_{\mathcal{M}}.
\end{align*}
According to Proposition~\ref{prop-r-equal-rB}, there exists a local mild solution $r$ to (\ref{SPDE-j}) with lifetime $\tau$ such that $(r)^{\tau} \mathbbm{1}_{[\![ 0,\tau [\![ } = (r^{B})^{\tau} \mathbbm{1}_{[\![ 0,\tau [\![ }$. We obtain up to an evanescent set
\begin{align*}
(r^{\tau})_- = ((r^{B})^{\tau})_- \in O_{\mathcal{M}} \subset \mathcal{M}
\end{align*}
as well as
\begin{align*}
r^{\tau} \mathbbm{1}_{[\![ 0,\tau [\![} = (r^{B})^{\tau} \mathbbm{1}_{[\![ 0,\tau [\![} \in C_{\mathcal{M}} \subset \mathcal{M}.
\end{align*}
Using Lemma~\ref{lemma-jump-global-help}, by (\ref{jumps}) we obtain
$r^{\tau} \in \overline{\mathcal{M}}$ up to an evanescent set,
proving that $\mathcal{M}$ is prelocally invariant for (\ref{SPDE-j}). 

If even condition (\ref{jumps-no-closure}) is satisfied, then by Lemma~\ref{lemma-jump-global-help} we obtain
$r^{\tau} \in \mathcal{M}$ up to an evanescent set,
and hence, $\mathcal{M}$ is locally invariant for (\ref{SPDE-j}).
\end{proof}

\begin{proof}[Proof of Theorem~\ref{thm-main-2}]
Let $h_0 \in \mathcal{M}$ be arbitrary. Then there exists a unique mild and weak solution $r = r^{(h_0)}$ to (\ref{SPDE-j}); see, e.g., \cite[Corollary~10.9]{SPDE}. Defining the stopping time
\begin{align}\label{def-tau-0}
\tau := \inf \{ t \geq 0 : r_t \notin \mathcal{M} \},
\end{align}
we claim that
\begin{align}\label{claim-tau0}
\mathbb{P}( \tau = \infty ) = 1.
\end{align}
Suppose, on the contrary, that (\ref{claim-tau0}) is not satisfied. Then there exists $N \in \mathbb{N}$ such that $\mathbb{P}( \tau \leq N ) > 0$.
We define the bounded stopping
time $\tau_0 := \tau \wedge N$. By the closedness of $\mathcal{M}$ in $H$, we have
$(r^{\tau_0})_- \in \mathcal{M}$
up to an evanescent set.
Therefore, by relation (\ref{jumps}) and Lemma~\ref{lemma-jump-global-help-2} we obtain $r^{\tau_0} \in \mathcal{M}$
up to an evanescent set. We define the filtration $\mathbb{F}^{(\tau_0)} := (\mathcal{F}_{\tau_0 + t})_{t \geq 0}$, the sequence $(\beta^{(\tau_0),j})_{j \in \mathbb{N}}$ of real-valued processes by
\begin{align}
\beta_t^{(\tau_0),j} := \beta_{\tau_0 + t}^j - \beta_{\tau_0}^j, \quad t \geq 0,
\end{align}
and the random measure ${\mu}^{(\tau_0)}$ on $\mathbb{R}_+ \times E$ by
\begin{align}\label{mu-new-tau}
{\mu}^{(\tau_0)}(\omega;B) := \mu(\omega;B_{\tau_0(\omega)}), \quad \omega \in \Omega \text{ and } B \in \mathcal{B}(\mathbb{R}_+) \otimes \mathcal{E},
\end{align}
where we use the notation
\begin{align*}
B_{\tau_0} := \{ (t + \tau_0,x) \in \mathbb{R}_+ \times E : (t,x) \in B
\}.
\end{align*}
According to \cite[Lemma~4.6]{Positivity}, the sequence $(\beta^{(\tau_0),j})_{j \in \mathbb{N}}$ is a sequence of
real-valued independent standard Wiener processes, adapted to $\mathbb{F}^{(\tau_0)}$, and $\mu^{(\tau_0)}$ is a time-homogeneous Poisson random measure relative to the filtration $\mathbb{F}^{(\tau_0)}$ with compensator $dt \otimes F(dx)$. The process $r_{\tau_0 + \bullet}$ is a weak solution to the time-shifted SPDE
\begin{align}\label{SPDE-j-shifted}
\left\{
\begin{array}{rcl}
dr_t & = & (Ar_t + \alpha(r_t))dt + \sum_{j \in \mathbb{N}} \sigma^j(r_t) d\beta_t^{(\tau_0),j}
\\ && + \int_E \gamma(r_{t-},x) (\mu^{(\tau_0)}(dt,dx) - F(dx) dt)
\medskip
\\ r_0 & = & h_0
\end{array}
\right.
\end{align}
with initial condition $r_{\tau_0}$, because for each $\zeta \in \mathcal{D}(A^*)$ we have $\mathbb{P}$-almost surely
\begin{align*}
&\langle \zeta, r_{\tau_0 + t} \rangle = \langle \zeta, r_{\tau_0} \rangle + \langle \zeta, r_{\tau_0 + t} - r_{\tau_0} \rangle
\\ &= \langle \zeta, r_{\tau_0} \rangle + \int_{\tau_0}^{\tau_0 + t} \big( \langle A^* \zeta, r_s \rangle + \langle \zeta, \alpha(r_s) \rangle \big) ds + \sum_{j \in \mathbb{N}} \int_{\tau_0}^{\tau_0 + t} \langle \zeta, \sigma^j(r_s) \rangle d\beta_s^j
\\ &\quad + \int_{\tau_0}^{\tau_0 + t} \int_E \langle \zeta, \gamma(r_{s-},x) \rangle (\mu(ds,dx) - F(dx)ds)
\\ &= \langle \zeta, r_{\tau_0} \rangle + \int_{0}^{t} \big( \langle A^* \zeta, r_{\tau_0 + s} \rangle + \langle \zeta, \alpha(r_{\tau_0 + s}) \rangle \big) ds + \sum_{j \in \mathbb{N}} \int_{0}^{t} \langle \zeta, \sigma(r_{\tau_0 + s}) \rangle d\beta_s^{(\tau_0),j}
\\ &\quad + \int_{0}^{t} \int_E \langle \zeta, \gamma(r_{(\tau_0 + s)-},x) \rangle (\mu^{(\tau_0)}(ds,dx) - F(dx)ds), \quad t \geq 0.
\end{align*}
There exists $K \in \mathbb{N}$ such that $\mathbb{P}(\Gamma) > 0$, where
\begin{align*}
\Gamma := \{ \tau \leq N \} \cap \{ \| r_{\tau_0} \| \leq K \}.
\end{align*}
By choosing a suitable covering $\mathcal{M} = \bigcup_{k \in \mathbb{N}} \mathcal{M}_k$ according to Lindel\"of's Lemma \cite[Lemma~1.1.6]{Abraham} and arguing as in the second part of the proof of Theorem~\ref{thm-main-1}, there exists a local weak solution $r^K$ to the time-shifted SPDE (\ref{SPDE-j-shifted}) with the $\mathcal{F}_{\tau_0}$-measurable initial condition $r_{\tau_0} \mathbbm{1}_{ \{ \| r_{\tau_0} \| \leq K \} }$ and lifetime $\varrho > 0$ such that $(r^K)^{\varrho} \in \mathcal{M}$ up to an evanescent set.
Noting that $\{ \tau \leq N \} = \{ \tau = \tau_0 \}$, by the uniqueness of weak solutions to (\ref{SPDE-j-shifted}) we obtain up to an evanescent set
\begin{align*}
(r_{\tau + \bullet})^{\varrho} \mathbbm{1}_{\Gamma} = (r_{\tau_0 + \bullet})^{\varrho} \mathbbm{1}_{\Gamma} = (r^K)^{\varrho} \mathbbm{1}_{\Gamma} \in \mathcal{M},
\end{align*}
which contradicts the definition (\ref{def-tau-0}) of $\tau$. Therefore, relation (\ref{claim-tau0}) is satisfied and we obtain $r \in \mathcal{M}$ up to an evanescent set. Hence, Lemma~\ref{lemma-strong-solution} implies that $r$ is a strong solution to (\ref{SPDE-j}).
\end{proof}

For the proof of Theorem~\ref{thm-main-3} we prepare an auxiliary result.

\begin{lemma}\label{lemma-tang-plus}
For all $h \in \partial \mathcal{M}$ we have $(T_h \mathcal{M})_+ =  T_h \mathcal{M} \cap \{ \eta_h \}^+$, where
\begin{align*}
\{ \eta_h \}^+ = \{ g \in H : \langle \eta_h,g \rangle \geq 0 \}
\end{align*}
\end{lemma}

\begin{proof}
See \cite[Lemma~3.7]{FTT-appendix}.
\end{proof}

\begin{proof}[Proof of Theorem~\ref{thm-main-3}]
Relation (\ref{gamma-integrable-pre}) implies (\ref{manifold-boundary}). Furthermore, presuming (\ref{domain}), we have (\ref{intro-drift}) if and only if (\ref{alpha-tangent}), (\ref{alpha-tangent-plus}) are satisfied. Indeed, noting that
\begin{equation}\label{rel-alpha-beta}
\begin{aligned}
&Ah + \alpha(h) - \frac{1}{2} \sum_{j \in \mathbb{N}} D\sigma^j(h)\sigma^j(h) - \int_E \gamma(h,x) F(dx)
\\ &= Ah + \alpha(h) - \frac{1}{2} \sum_{j \in \mathbb{N}} D\sigma^j(h)\sigma^j(h) - \int_E \Pi_{(T_h \mathcal{M})^{\perp}} \gamma(h,x) F(dx)
\\ &\quad - \Pi_{T_h \mathcal{M}} \int_E \gamma(h,x) F(dx), \quad h \in \mathcal{M},
\end{aligned}
\end{equation}
we have (\ref{alpha-tangent}) if and only if
\begin{align*}
Ah + \alpha(h) - \frac{1}{2} \sum_{j \in \mathbb{N}} D\sigma^j(h)\sigma^j(h) - \int_E \gamma(h,x) F(dx) \in T_h \mathcal{M}, \quad h \in \mathcal{M},
\end{align*}
and, by Lemma~\ref{lemma-tang-plus}, we have (\ref{alpha-tangent-plus}) if and only if
\begin{align*}
Ah + \alpha(h) - \frac{1}{2} \sum_{j \in \mathbb{N}} D\sigma^j(h)\sigma^j(h) - \int_E \gamma(h,x) F(dx) \in (T_h \mathcal{M})_+, \quad h \in \partial \mathcal{M},
\end{align*}
showing that condition (\ref{intro-drift}) is equivalent to (\ref{alpha-tangent}), (\ref{alpha-tangent-plus}).

Now, suppose that even condition (\ref{gamma-integrable}) is satisfied. Since, by Theorem~\ref{thm-main-1}, the mapping in (\ref{alpha-tangent}) is continuous on $\mathcal{M}$, identity (\ref{rel-alpha-beta}) together with relations (\ref{Lipschitz-gamma-rho}), (\ref{gamma-integrable}) and Lebesgue's dominated convergence theorem shows that the mapping in (\ref{intro-drift}) is continuous $\mathcal{M}$.
\end{proof}

\section*{Acknowledgement}

We are grateful to the editor and to several anonymous referees for valuable comments and suggestions. Stefan Tappe and Josef Teichmann gratefully acknowledge the support from the ETH Z\"{u}rich Foundation.


\begin{thebibliography}{20}

\bibitem{Abraham} Abraham, R., Marsden, J.~E., Ratiu, T. (1988):
  \textit{Manifolds, tensor analysis, and applications}. Springer, New York.

\bibitem{Bj-Ch} Bj\"ork, T., Christensen, B.~J. (1999):
  Interest rate dynamics and consistent forward rate
  curves.
  \textit{Mathematical Finance} {\bf 9}(4), 323--348.

\bibitem{Bj-La} Bj{\"o}rk, T., Land{\'e}n, C. (2002):
  On the construction of finite dimensional realizations for nonlinear forward rate models.
  \textit{Finance and Stochastics} {\bf 6}(3), 303--331.

\bibitem{Bj-Sv} Bj{\"o}rk, T., Svensson, L. (2001):
    On the existence of finite dimensional realizations for nonlinear forward rate models.
    \textit{Mathematical Finance} {\bf 11}(2), 205--243.

\bibitem{bqrt:10} Buckdahn, R., Quincampoix, M., Rainer, C., Teichmann, J. (2010):
Another proof for the equivalence between invariance of closed sets with respect to stochastic and deterministic systems.
\textit{Bulletin des Sciences Math\'{e}matiques} {\bf 134}(2), 207--214.

\bibitem{Da_Prato} Da~Prato, G., Zabczyk, J. (1992):
  \textit{Stochastic equations in infinite dimensions.} Cambridge University Press, Cambridge.

\bibitem{Dellacherie} Dellacherie, C., Meyer, P.~A. (1978):
  \textit{Probabilities and potential.}
  Hermann: Paris.

\bibitem{Filipovic-note} Filipovi\'c, D. (1999):
  A note on the Nelson--Siegel family.
  \textit{Mathematical Finance} {\bf 9}(4), 349--359.

  \bibitem{Filipovic-exponential} Filipovi\'c, D. (2000):
  Exponential-polynomial families and the term structure of interest rates.
  \textit{Bernoulli} {\bf 6}(6), 1081--1107.

    \bibitem{Filipovic-inv} Filipovi\'c, D. (2000):
  Invariant manifolds for weak solutions to stochastic equations.
  \textit{Probability Theory and Related Fields} {\bf 118}(3), 323--341.

\bibitem{fillnm} Filipovi\'c, D. (2001):
  \textit{Consistency problems for Heath--Jarrow--Morton interest rate models.}
  Springer, Berlin.

\bibitem{SPDE} Filipovi\'c, D., Tappe, S., Teichmann, J. (2010):
  Jump-diffusions in Hilbert spaces: Existence, stability and numerics.
  \textit{Stochastics} {\bf 82}(5), 475--520.

\bibitem{Positivity} Filipovi\'c, D., Tappe, S., Teichmann, J. (2010):
  Term structure models driven by Wiener processes and Poisson measures: Existence and positivity.
  \textit{SIAM Journal on Financial Mathematics} {\bf 1}(1), 523--554.

\bibitem{FTT-appendix} Filipovi\'c, D., Tappe, S., Teichmann, J. (2014):
  \textit{Stochastic partial differential equations and submanifolds in Hilbert spaces.} Appendix of this file.
  
\bibitem{Filipovic-Teichmann} Filipovi\'c, D., Teichmann, J. (2003):
  Existence of invariant manifolds for stochastic equations in infinite dimension.
  \textit{Journal of Functional Analysis} {\bf 197}(2), 398--432.

\bibitem{Filipovic-Teichmann-royal} Filipovi\'c, D., Teichmann, J. (2004):
  On the geometry of the term structure of interest rates.
  \textit{Proceedings of The Royal Society of London. Series A. Mathematical, Physical and Engineering Sciences}
  {\bf 460}(2041), 129--167.

\bibitem{Getoor} Getoor, R.~K. (1975):
  On the construction of kernels.
  \textit{S\'eminaire de Probabilit\'es} IX, Lecture Notes in Mathematics {\bf 465}, 443--463.

\bibitem{Jacod-Shiryaev}
  Jacod, J., Shiryaev, A.~N. (2003):
  \textit{Limit theorems for stochastic processes}.
  Springer, Berlin.

\bibitem{Nakayama-Support} Nakayama, T. (2004):
  Support theorem for mild solutions of SDE's in Hilbert spaces.
  \textit{J. Math. Sci. Univ. Tokyo} {\bf 11}(3), 245--311.

\bibitem{Nakayama} Nakayama, T. (2004):
  Viability Theorem for SPDE's including HJM framework.
  \textit{J. Math. Sci. Univ. Tokyo} {\bf 11}(3), 313--324.

\bibitem{Rudin} Rudin, W. (1991): \textit{Functional Analysis}.
Second Edition, McGraw-Hill.

\bibitem{sim:00} Simon, T. (2000):
Support theorem for jump processes.
\textit{Stochastic Processes and Their Applications} {\bf 89}(1), 1--30.
		
\bibitem{tz:06} Tessitore, G., Zabczyk, J. (2006):
Wong-{Z}akai approximations of stochastic evolution equations.
\textit{Journal of Evolution Equations} {\bf 6}(4), 621--655.

\end{thebibliography}
\end{document}


\hyphenation{ma-ni-folds Swit-zer-land Lem-ma}

\title[SPDEs and submanifolds in Hilbert spaces]{Stochastic partial differential equations and submanifolds in Hilbert spaces}
\author{Damir Filipovi\'c \and Stefan Tappe \and Josef Teichmann}
\date{20 June 2014}

\address{EPFL and Swiss Finance Institute, Quartier UNIL-Dorigny, Extranef 218, CH-1015 Lausanne, Switzerland}
\email{damir.filipovic@epfl.ch}
\address{Leibniz Universit\"{a}t Hannover, Institut f\"{u}r Mathematische Stochastik, Welfengarten 1, D-30167 Hannover, Germany}
\email{tappe@stochastik.uni-hannover.de}
\address{ETH Z\"urich, Department of Mathematics, R\"amistrasse 101, CH-8092 Z\"urich, Switzerland}
\email{josef.teichmann@math.ethz.ch}
\begin{abstract}
The goal of this appendix is to provide results about stochastic partial differential equations driven by Wiener processes and Poisson measures and results about submanifolds in Hilbert spaces. It should serve as a reference for auxiliary results that we require in \cite{FTT-manifolds}.
\end{abstract}
\keywords{Stochastic partial differential equation, jump-diffusion,
submanifold with boundary, tangent space}
\subjclass[2010]{60H15, 60G17}
\maketitle

\section{Introduction}

In \cite{FTT-manifolds}, we have provided necessary and sufficient conditions for stochastic invariance of finite dimensional submanifolds with boundary in Hilbert spaces for stochastic partial differential equations (SPDEs) of the form 
\begin{align}\label{SPDE}
\left\{
\begin{array}{rcl}
dr_t & = & (Ar_t + \alpha(r_t))dt + \sigma(r_t) dW_t
+ \int_E \gamma(r_{t-},x) (\mu(dt,dx) - F(dx) dt)
\medskip
\\ r_0 & = & h_0
\end{array}
\right.
\end{align}
driven by Wiener processes and Poisson random measures. The goal of this appendix is to serve as a reference for auxiliary results that we require for the proofs in \cite{FTT-manifolds}. 

In Section~\ref{app-SDE} we provide results about SPDEs driven by Wiener processes and Poisson random measures, and in Section~\ref{app-submanifolds} we provide results about finite dimensional submanifolds with boundary in Hilbert spaces.

\section{SPDEs driven by Wiener processes and Poisson random measures}\label{app-SDE}

In this section, we provide results about
SPDEs driven by Wiener processes and Poisson random measures. References on this topic are, e.g., \cite{Ruediger-mild,Marinelli-Prevot-Roeckner,SPDE}. Furthermore, we mention \cite{P-Z-book} regarding SPDEs driven by L\'{e}vy processes, and \cite{Da_Prato, Prevot-Roeckner, Atma-book} regarding SPDEs driven by Wiener processes.

In the sequel, $(\Omega, \mathcal{F}, (\mathcal{F}_t)_{t \geq 0}, \mathbb{P})$ denotes a filtered probability space satisfying the usual conditions.
Let $H$ be a separable Hilbert space and let $(S_t)_{t \geq 0}$ be a $C_0$-semigroup on $H$ with
infinitesimal generator $A : \mathcal{D}(A) \subset H \rightarrow
H$. We denote by $A^* : \mathcal{D}(A^*) \subset H \rightarrow H$
the adjoint operator of $A$. Recall that the domains
$\mathcal{D}(A)$ and $\mathcal{D}(A^*)$ are dense in $H$, see, e.g.,
\cite[Theorems~13.35.c and 13.12]{Rudin}.

Let $\mathbb{H}$ be another separable Hilbert space and let $Q \in L(\mathbb{H})$ be a nuclear, self-adjoint, positive definite linear operator. Then, there exist an orthonormal basis $(e_j)_{j \in \mathbb{N}}$ of $\mathbb{H}$ and a sequence $(\lambda_j)_{j \in \mathbb{N}} \subset (0,\infty)$ with $\sum_{j \in \mathbb{N}} \lambda_j < \infty$ such that
\begin{align*}
Qu = \sum_{j \in \mathbb{N}} \lambda_j \langle u,e_j \rangle_{\mathbb{H}} \, e_j, \quad u \in \mathbb{H},
\end{align*}
namely, the $\lambda_j$ are the eigenvalues of $Q$, and each $e_j$
is an eigenvector corresponding to $\lambda_j$. The space $\mathbb{H}_0 := Q^{1/2}(\mathbb{H})$, equipped with the inner product
\begin{align*}
\langle u,v \rangle_{\mathbb{H}_0} := \langle Q^{-1/2} u, Q^{-1/2} v \rangle_{\mathbb{H}},
\end{align*}
is another separable Hilbert space
and $( \sqrt{\lambda_j} e_j )_{j \in \mathbb{N}}$ is an orthonormal basis.
Let $W$ be a $\mathbb{H}$-valued $Q$-Wiener process, see \cite[pages~86, 87]{Da_Prato}. We denote by $L_2^0(H) := L_2(\mathbb{H}_0,H)$ the space of Hilbert-Schmidt
operators from $\mathbb{H}_0$ into $H$, which, endowed with the
Hilbert-Schmidt norm
\begin{align*}
\| \Phi \|_{L_2^0(H)} := \bigg( \sum_{j \in \mathbb{N}} \lambda_j \| \Phi e_j \|^2 \bigg)^{1/2},
\quad \Phi \in L_2^0(H)
\end{align*}
itself is a separable Hilbert space.
According to \cite[Proposition~4.1]{Da_Prato}, the sequence of stochastic
processes $( \beta^j )_{j \in \mathbb{N}}$ defined as
\begin{align*}
\beta^j := \frac{1}{\sqrt{\lambda_j}} \langle W, e_j \rangle, \quad j \in \mathbb{N}
\end{align*}
is a sequence of
real-valued independent standard Wiener processes and we
have the expansion
\begin{align*}
W = \sum_{j \in \mathbb{N}} \sqrt{\lambda _j} \beta^j e_j.
\end{align*}
Note that $L_2^0(H) \cong \ell^2(H)$, because
\begin{align}\label{isom-isom}
L_2^0(H) \rightarrow \ell^2(H), \quad \Phi \mapsto (\Phi^j)_{j \in \mathbb{N}} \quad \text{with $\Phi^j := \sqrt{\lambda_j} \Phi e_j$, $j \in \mathbb{N}$}
\end{align}
is an isometric isomorphism. According to \cite[Theorem~4.3]{Da_Prato}, for every predictable process $\Phi : \Omega \times
\mathbb{R}_+ \rightarrow L_2^0(H)$ satisfying
\begin{align*}
\mathbb{P} \bigg( \int_0^t \| \Phi_s \|_{L_2^0(H)}^2 ds < \infty
\bigg) = 1 \quad \text{for all $t \geq 0$}
\end{align*}
we have the identity
\begin{align}\label{Wiener-int-expansion}
\int_0^t \Phi_s dW_s = \sum_{j \in \mathbb{N}} \int_0^t \Phi_s^j d\beta_s^j, \quad t \geq 0.
\end{align}
Let $(E,\mathcal{E})$ be a measurable space which we assume to be a
\textit{Blackwell space} (see \cite{Dellacherie,Getoor}). We remark
that every Polish space with its Borel $\sigma$-field is a Blackwell
space.
Furthermore, let $\mu$ be a time-homogeneous Poisson random measure on
$\mathbb{R}_+ \times E$, see \cite[Definition~II.1.20]{Jacod-Shiryaev}.
Then its compensator is of the form $dt \otimes F(dx)$, where $F$ is
a $\sigma$-finite measure on $(E,\mathcal{E})$.

We shall now focus on SPDEs of the type (\ref{SPDE}). Let $\alpha : H \rightarrow H$, $\sigma : H \rightarrow L_2^0(H)$ and $\gamma : H \times E \rightarrow H$ be measurable mappings.

\begin{definition}\label{def-solutions}
Let $h_0 : \Omega \rightarrow H$ be a $\mathcal{F}_0$-measurable random variable. Furthermore, let $r = r^{(h_0)}$ be a $H$-valued c\`{a}dl\`{a}g adapted process and let $\tau > 0$ a stopping time such that for all $t \geq 0$ we have
\begin{align*}
\mathbb{P} \bigg( \int_{0}^{t \wedge \tau} \bigg( \| r_s \| + \| \alpha(r_s) \| + \| \sigma(r_s) \|_{L_2^0(H)}^2 + \int_E \| \gamma(r_s,x) \|^2 F(dx) \bigg) ds < \infty \bigg) = 1.
\end{align*}
Then the process $r$ is called
\begin{itemize}
\item a \emph{local strong solution} to (\ref{SPDE}), if
\begin{align}\label{f-meant-0}
&r_{t \wedge \tau} \in \mathcal{D}(A) \quad \text{for almost all $t \in \mathbb{R}_+$,} \quad \text{$\mathbb{P}$-almost surely,}
\\ \label{f-meant-1} &\mathbb{P} \bigg( \int_{0}^{t \wedge \tau} \| A r_{s} \| ds < \infty \bigg) = 1 \quad \text{ for all } t \geq 0
\end{align}
and we have $\mathbb{P}$-almost surely
\begin{equation}\label{f-meant-2}
\begin{aligned}
r_{t \wedge \tau} &= h_0 + \int_{0}^{t \wedge \tau} \big( A r_s + \alpha(r_s) \big) ds + \int_{0}^{t \wedge \tau} \sigma(r_s) dW_s
\\ &\quad + \int_{0}^{t \wedge \tau} \int_E \gamma(r_{s-},x) (\mu(ds,dx) - F(dx)ds), \quad t \geq 0.
\end{aligned}
\end{equation}
\item a \emph{local weak solution} to (\ref{SPDE}), if for all $\zeta \in \mathcal{D}(A^*)$ we have $\mathbb{P}$-almost surely
\begin{align*}
\langle \zeta,r_{t \wedge \tau} \rangle &= \langle \zeta,h_0 \rangle + \int_{0}^{t \wedge \tau} \big( \langle A^* \zeta, r_s \rangle + \langle \zeta, \alpha(r_s) \rangle \big) ds + \int_{0}^{t \wedge \tau} \langle \zeta, \sigma(r_s) \rangle dW_s
\\ &\quad + \int_{0}^{t \wedge \tau} \int_E \langle \zeta, \gamma(r_{s-},x) \rangle (\mu(ds,dx) - F(dx)ds), \quad t \geq 0.
\end{align*}
\item a \emph{local mild solution} to (\ref{SPDE}), if we have $\mathbb{P}$-almost surely
\begin{align*}
r_{t \wedge \tau} &= S_{t \wedge \tau} h_0 + \int_{0}^{t \wedge \tau} S_{(t \wedge \tau) - s} \alpha(r_s) ds + \int_{0}^{t \wedge \tau} S_{(t \wedge \tau) - s} \sigma(r_s) dW_s
\\ &\quad + \int_{0}^{t \wedge \tau} \int_E S_{(t \wedge \tau) - s} \gamma(r_{s-},x) (\mu(ds,dx) - F(dx)ds), \quad t \geq 0.
\end{align*}
\end{itemize}
We call $\tau$ the \emph{lifetime} of $r$. If $\tau = \infty$, then we call $r$ a \emph{strong}, \emph{weak} or \emph{mild solution} to (\ref{SPDE}), respectively.
\end{definition}

\begin{remark}\label{remark-domain}
Since the process $r$ is c\`{a}dl\`{a}g, we have
\begin{align*}
r_t = r_{t-} \quad \text{for almost all $t \in \mathbb{R}_+$,} \quad \text{$\mathbb{P}$-almost surely,}
\end{align*}
and hence, relation (\ref{f-meant-0}) implies
\begin{align*}
r_{(t \wedge \tau)-} \in \mathcal{D}(A) \quad \text{for almost all $t \in \mathbb{R}_+$,} \quad \text{$\mathbb{P}$-almost surely.}
\end{align*}
According to \cite[Lemma~2.4.2]{fillnm}, the process $f$ defined by
\begin{align}\label{def-f-Ar}
f_t :=
\begin{cases}
A r_{t-}, & \text{if } r_{t-} \in \mathcal{D}(A)
\\ 0, & \text{otherwise}
\end{cases}
\end{align}
is predictable. By slight abuse of notation, we have written $A r$ instead of $f$ in (\ref{f-meant-1}) and (\ref{f-meant-2}).
\end{remark}

\begin{remark}
The following results are well-known:
\begin{itemize}
\item Every (local) strong solution to (\ref{SPDE}) is also a (local) weak solution to (\ref{SPDE}).

\item Every (local) weak solution to (\ref{SPDE}) is also a (local) mild solution to (\ref{SPDE}).

\item If $A$ is bounded, i.e. generates a norm-continuous semigroup $(S_t)_{t \geq 0}$, then the concepts of (local) strong, weak and mild solutions to (\ref{SPDE}) are equivalent.
\end{itemize}
\end{remark}

\begin{definition}
The semigroup $(S_t)_{t \geq 0}$ is called \emph{pseudo-contractive}, if
\begin{align*}
\| S_t \| \leq e^{\omega t}, \quad t \geq 0
\end{align*}
for some constant $\omega \in \mathbb{R}$.
\end{definition}

\begin{definition}
The mappings $(\alpha,\sigma,\gamma)$ are called \emph{(locally) Lipschitz continuous}, if:
\begin{itemize}
\item $\alpha : H \rightarrow H$ is (locally) Lipschitz continuous.

\item $\sigma : H \rightarrow L_2^0(H)$ is (locally) Lipschitz continuous.

\item For each $h \in H$ we have $\gamma(h,\bullet) \in L^2(F)$ and $\gamma : H \rightarrow L^2(F)$ is (locally) Lipschitz continuous, where we use the notation $L^2(F) := L^2(E,\mathcal{E},F;H)$.
\end{itemize}
\end{definition}

\begin{remark}\label{remark-ex-SDE}
The following results are well-known:
\begin{itemize}
\item If $(S_t)_{t \geq 0}$ is pseudo-contractive and $(\alpha,\sigma,\gamma)$ are Lipschitz continuous, then we have existence and uniqueness of mild and weak solutions to (\ref{SPDE}).

\item If $(S_t)_{t \geq 0}$ is pseudo-contractive and $(\alpha,\sigma,\gamma)$ are locally Lipschitz continuous, then we have existence and uniqueness of local mild and weak solutions to (\ref{SPDE}).

\item If $(\alpha,\sigma,\gamma)$ are locally Lipschitz continuous, then we have uniqueness of local mild solutions to (\ref{SPDE}).
\end{itemize}
\end{remark}

\begin{lemma}\label{lemma-strong-solution}
Let $\mathcal{M} \subset \mathcal{D}(A)$ be a subset such that $A$
is continuous on $\mathcal{M}$, and let $r = r^{(h_0)}$ be a local weak solution to (\ref{SPDE}) with lifetime $\tau > 0$ for some $\mathcal{F}_0$-measurable random variable $h_0 : \Omega \rightarrow H$ such that $(r^{\tau})_- \in \mathcal{M}$ up to an evanescent set.
Then $r$ is also a local strong solution to (\ref{SPDE}) with lifetime $\tau$.
\end{lemma}

\begin{proof}
Condition (\ref{f-meant-0}) is satisfied, because $(r^{\tau})_- \in \mathcal{M} \subset \mathcal{D}(A)$ up to an evanescent set, and condition (\ref{f-meant-1}) is satisfied due to the continuity of $A$ on $\mathcal{M}$. Taking into account Remark~\ref{remark-domain}, we obtain for each $\zeta \in \mathcal{D}(A^*)$ that $\mathbb{P}$-almost surely
\begin{align*}
\langle \zeta,r_{t \wedge \tau} \rangle &= \langle \zeta,h_0 \rangle + \int_0^{t \wedge \tau}
( \langle A^* \zeta, r_s \rangle + \langle \zeta, \alpha(r_s) \rangle ) ds + \int_0^{t \wedge \tau}
\langle \zeta, \sigma(r_s) \rangle dW_s
\\ &\quad + \int_0^{t \wedge \tau} \int_{E} \langle \zeta, \gamma(r_{s-},x) \rangle
(\mu(ds,dx) - F(dx)ds)
\\ &= \Big\langle \zeta, h_0 + \int_0^{t \wedge \tau} (A r_s + \alpha(r_s))ds + \int_0^{t \wedge \tau}
\sigma(r_s) dW_s
\\ &\quad + \int_0^{t \wedge \tau} \int_{E} \gamma(r_{s-},x) (\mu(ds,dx) - F(dx)ds) \Big\rangle, \quad t \geq 0.
\end{align*}
Since $\mathcal{D}(A^*)$ is dense in $H$, we get $\mathbb{P}$-almost surely
\begin{align*}
r_{t \wedge \tau} &= h_0 + \int_{0}^{t \wedge \tau} \big( A r_s + \alpha(r_s) \big) ds + \int_{0}^{t \wedge \tau} \sigma(r_s) dW_s
\\ &\quad + \int_{0}^{t \wedge \tau} \int_E \gamma(r_{s-},x) (\mu(ds,dx) - F(dx)ds), \quad t \geq 0,
\end{align*}
showing that $r$ is a local strong solution to (\ref{SPDE}) with lifetime $\tau$.
\end{proof}

\begin{remark}\label{remark-mu}
According to \cite[Proposition~II.1.14]{Jacod-Shiryaev}, there exist a sequence $(\tau_n)_{n \in \mathbb{N}}$ of finite stopping times with
$[\![ \tau_n ]\!] \cap [\![ \tau_m ]\!] = \emptyset$ for $n \neq m$ and an $E$-valued optional process $\xi$ such that for every optional process $\delta : \Omega \times \mathbb{R}_+ \times E \rightarrow H$ with
\begin{align}\label{gamma-mu-integrable}
\mathbb{P} \bigg( \int_0^t \int_E \| \delta(s,x) \| \mu(ds,dx) < \infty \bigg) = 1 \quad \text{for all $t \geq 0$}
\end{align}
we have
\begin{align}\label{integrate-mu}
\int_0^t \int_E \delta(s,x) \mu(ds,dx) = \sum_{n \in \mathbb{N}} \delta(\tau_n,\xi_{\tau_n}) \mathbbm{1}_{\{ \tau_n \leq t \}}, \quad t \geq 0.
\end{align}
Furthermore, for every predictable process $\delta : \Omega \times \mathbb{R}_+ \times E \rightarrow H$ with
\begin{align*}
\mathbb{P} \bigg( \int_0^t \int_E \| \delta(s,x) \|^2 F(dx)ds < \infty \bigg) = 1 \quad \text{for all $t \geq 0$}
\end{align*}
the jumps of the integral process
\begin{align*}
X_t := \int_0^t \int_E \delta(s,x) (\mu(ds,dx) - F(dx)ds), \quad t \geq 0
\end{align*}
are given by
\begin{align}\label{jumps-of-integral}
\Delta X_t = \delta(t,\xi_t) \sum_{n \in \mathbb{N}} \mathbbm{1}_{\{ \tau_n = t \}}, \quad t \geq 0,
\end{align}
see \cite[Section~II.1.d]{Jacod-Shiryaev}.
\end{remark}

\begin{lemma}\label{lemma-jumps-of-SPDE}
Let $r = r^{(h_0)}$ be a local weak solution to (\ref{SPDE}) with lifetime $\tau > 0$ for some $\mathcal{F}_0$-measurable random variable $h_0 : \Omega \rightarrow H$. Then the following statements are true:
\begin{enumerate}
\item The jumps of the stopped process $r^{\tau}$ are given by
\begin{align*}
\Delta r_{t \wedge \tau} = \gamma(r_{(t \wedge \tau)-},\xi_{t \wedge \tau}) \sum_{n \in \mathbb{N}} \mathbbm{1}_{\{ \tau_n = t \wedge \tau \}}, \quad t \geq 0.
\end{align*}

\item For each $n \in \mathbb{N}$ we have
\begin{align*}
\Delta r_{\tau_n} \mathbbm{1}_{\{ \tau_n \leq \tau \}} = \gamma(r_{\tau_n -}, \xi_{\tau_n}) \mathbbm{1}_{\{ \tau_n \leq \tau \}}.
\end{align*}
\end{enumerate}
\end{lemma}

\begin{proof}
Let $X$ be the process
\begin{align*}
X_t := \int_0^{t \wedge \tau} \int_E \langle \zeta, \gamma(r_{s-},x) \rangle ( \mu(ds,dx) - F(dx)ds ), \quad t \geq 0.
\end{align*}
Since $r$ is a local weak solution to (\ref{SPDE}), for every $\zeta \in \mathcal{D}(A^*)$ we have, by using (\ref{jumps-of-integral}),
\begin{align*}
\langle \zeta, \Delta r_{t \wedge \tau} \rangle &= \Delta \langle \zeta, r_{t \wedge \tau} \rangle = \Delta X_{t \wedge \tau} = \langle \zeta, \gamma(r_{(t \wedge \tau)-},\xi_{t \wedge \tau}) \rangle \sum_{n \in \mathbb{N}} \mathbbm{1}_{\{ \tau_n = t \wedge \tau \}}
\\ &= \Big\langle \zeta, \gamma(r_{(t \wedge \tau)-},\xi_{t \wedge \tau}) \sum_{n \in \mathbb{N}} \mathbbm{1}_{\{ \tau_n = t \wedge \tau \}} \Big\rangle, \quad t \geq 0.
\end{align*}
Taking into account that $\mathcal{D}(A^*)$ is dense in $H$, the first statement follows. Since $[\![ \tau_n ]\!] \cap [\![ \tau_m ]\!] = \emptyset$ for $n \neq m$, we deduce that
\begin{align*}
\Delta r_{\tau_n} \mathbbm{1}_{\{ \tau_n \leq \tau \}} &= \Delta r_{\tau_n \wedge \tau} \mathbbm{1}_{\{ \tau_n \leq \tau \}} = \bigg( \gamma(r_{(\tau_n \wedge \tau)-},\xi_{\tau_n \wedge \tau}) \sum_{m \in \mathbb{N}} \mathbbm{1}_{\{ \tau_m = \tau_n \wedge \tau \}} \bigg) \mathbbm{1}_{\{ \tau_n \leq \tau \}}
\\ &= \bigg( \gamma(r_{\tau_n-},\xi_{\tau_n}) \sum_{m \in \mathbb{N}} \mathbbm{1}_{\{ \tau_m = \tau_n\}} \bigg) \mathbbm{1}_{\{ \tau_n \leq \tau \}} = \gamma(r_{\tau_n -}, \xi_{\tau_n}) \mathbbm{1}_{\{ \tau_n \leq \tau \}}
\end{align*}
for each $n \in \mathbb{N}$, establishing the second statement.
\end{proof}

From now on, we fix mappings $\alpha : H \rightarrow H$, $\sigma^j : H \rightarrow H$, $j \in \mathbb{N}$, $\gamma : H \times E \rightarrow H$ satisfying the following regularity conditions:
\begin{itemize}
\item The mapping $\alpha : H \rightarrow H$ is locally Lipschitz continuous, that is, for each $n \in \mathbb{N}$ there is a constant $L_n \geq 0$ such that
\begin{align}\label{Lipschitz-alpha-st}
\| \alpha(h_1) - \alpha(h_2) \| \leq L_n \| h_1 - h_2 \|, \quad h_1, h_2 \in H \text{ with } \| h_1 \|, \| h_2 \| \leq n.
\end{align}

\item For each $n \in \mathbb{N}$ there exists a sequence $(\kappa_n^j)_{j \in \mathbb{N}} \subset \mathbb{R}_+$ with $\sum_{j \in \mathbb{N}} (\kappa_n^j)^2 < \infty$ such that for all $j \in \mathbb{N}$ the mapping $\sigma^j : H \rightarrow H$ satisfies
\begin{align}\label{Lipschitz-sigma-st}
\| \sigma^j(h_1) - \sigma^j(h_2) \| &\leq \kappa_n^j \| h_1 - h_2 \|, \quad h_1,h_2 \in H \text{ with } \| h_1 \|, \| h_2 \| \leq n,
\\ \label{linear-growth-sigma} \| \sigma^j(h) \| &\leq \kappa_n^j, \quad h \in H \text{ with } \| h \| \leq n.
\end{align}
Consequently, for each $j \in \mathbb{N}$ the mapping $\sigma^j$ is locally Lipschitz continuous.

\item The mapping $\gamma : H \times E \rightarrow H$ is measurable, and for each $n \in \mathbb{N}$ there exists a measurable function $\rho_n : E \rightarrow \mathbb{R}_+$ with
\begin{align}\label{rho-square-int}
\int_E \rho_n(x)^2 F(dx) < \infty.
\end{align}
such that for all $x \in E$ the mapping $\gamma(\bullet,x) : H \rightarrow H$ satisfies
\begin{align}\label{Lipschitz-gamma-rho}
\| \gamma(h_1,x) - \gamma(h_2,x) \| &\leq \rho_n(x) \| h_1 - h_2 \|, \quad h_1,h_2 \in H \text{ with } \| h_1 \|, \| h_2 \| \leq n,
\\ \label{linear-growth-rho} \| \gamma(h,x) \| &\leq \rho_n(x), \quad h \in H \text{ with } \| h \| \leq n.
\end{align}
Consequently, for each $x \in E$ the mapping $\gamma(\bullet,x)$ is locally Lipschitz continuous.

\item We assume that for each $j \in \mathbb{N}$ the mapping $\sigma^j : H \rightarrow H$ is continuously differentiable, that is
\begin{align}\label{sigma-C1}
\sigma^j \in C^1(H) \quad \text{for all $j \in \mathbb{N}$.}
\end{align}
\end{itemize}
Using the identification $L_2^0(H) \cong \ell^2(H)$, which holds true by the isometric isomorphism defined in (\ref{isom-isom}), we can identify the sequence $(\sigma^j)_{j \in \mathbb{N}}$ of mappings $\sigma^j : H \rightarrow H$ with a locally Lipschitz continuous mapping $\sigma : H \rightarrow L_2^0(H)$, and, in view of (\ref{Wiener-int-expansion}), equation (\ref{SPDE}) can be rewritten equivalently
\begin{align}\label{SPDE-j}
\left\{
\begin{array}{rcl}
dr_t & = & (Ar_t + \alpha(r_t))dt + \sum_{j \in \mathbb{N}} \sigma^j(r_t) d\beta_t^j
\\ && + \int_E \gamma(r_{t-},x) (\mu(dt,dx) - F(dx) dt)
\medskip
\\ r_0 & = & h_0.
\end{array}
\right.
\end{align}

\begin{definition}\label{def-prelocal}
Let $B_1 \subset B_2 \subset H$ be two nonempty Borel sets. $B_1$ is called \emph{prelocally invariant} in $B_2$ for (\ref{SPDE-j}), if for all $h_0 \in B_1$ there exists a local mild solution $r = r^{(h_0)}$ to (\ref{SPDE-j}) with lifetime $\tau > 0$ such that $(r^{\tau})_- \in B_1$ and $r^{\tau} \in B_2$ up to an evanescent set.
\end{definition}

\begin{lemma}\label{lemma-jumps-global-set}
Let $B_1 \subset B_2 \subset H$ be two Borel sets such that $B_1$ is prelocally invariant in $B_2$ for (\ref{SPDE}). Then we have
\begin{align*}
h + \gamma(h,x) \in \overline{B}_2 \quad \text{ for $F$-almost all $x \in E$,} \quad \text{for all $h \in B_1$.}
\end{align*}
\end{lemma}

\begin{proof}
We denote by
\begin{align*}
d_{B_2} : H \rightarrow \mathbb{R}_+, \quad d_{B_2}(h) := \inf_{g \in B_2} \| h-g \|
\end{align*}
the distance function of the set $B_2$. Since
\begin{align*}
|d_{B_2}(h_1) - d_{B_2}(h_2)| \leq \| h_1 - h_2 \| \quad \text{for all $h_1,h_2 \in H$,}
\end{align*}
by the linear growth condition (\ref{linear-growth-rho}), for all $n \in \mathbb{N}$, all $h \in \overline{B}_2$ with $\| h \| \leq n$ and all $x \in E$ we have
\begin{align}\label{distance-Lip-1}
| d_{B_2}(h + \gamma(h,x)) | = | d_{B_2}(h + \gamma(h,x)) - d_{B_2}(h) | \leq \| \gamma(h,x) \| \leq \rho_n(x).
\end{align}
Thus, by (\ref{Lipschitz-gamma-rho}) and Lebesgue's dominated convergence theorem, the mapping
\begin{align}\label{mapping-int-dist}
\overline{B}_2 \rightarrow \mathbb{R}, \quad h \mapsto \int_E |d_{B_2}(h + \gamma(h,x))|^2 F(dx)
\end{align}
is continuous. Now, let $h \in B_1$ be arbitrary. Since $B_1$ is prelocally invariant in $B_2$ for (\ref{SPDE}), there exists a local mild solution $r = r^{(h)}$ to (\ref{SPDE}) with lifetime $\tau > 0$ such that $(r^{\tau})_- \in B_1$ and $r^{\tau} \in B_2$ up to an evanescent set. Taking into account \cite[Theorem~II.1.8]{Jacod-Shiryaev}, identity (\ref{integrate-mu}) and Lemma~\ref{lemma-jumps-of-SPDE}, we obtain
\begin{align*}
&\mathbb{E} \bigg[ \int_0^{\tau} \int_E | d_{B_2}(r_{s-} + \gamma(r_{s-},x)) |^2 F(dx)ds \bigg]
\\ &= \mathbb{E} \bigg[ \int_0^{\tau} \int_E | d_{B_2}(r_{s-} + \gamma(r_{s-},x)) |^2 \mu(ds,dx) \bigg]
\\ &= \mathbb{E} \bigg[ \sum_{n \in \mathbb{N}} | d_{B_2}(r_{\tau_n -} + \gamma(r_{\tau_n -},\xi_{\tau_n})) |^2 \mathbbm{1}_{\{ \tau_n \leq \tau \}} \bigg]
\\ &= \mathbb{E} \bigg[ \sum_{n \in \mathbb{N}} | d_{B_2}(r_{\tau_n -} + \Delta r_{\tau_n} ) |^2 \mathbbm{1}_{\{ \tau_n \leq \tau \}} \bigg]
= \mathbb{E} \bigg[ \sum_{n \in \mathbb{N}} | d_{B_2}(r_{\tau_n}) |^2 \mathbbm{1}_{\{ \tau_n \leq \tau \}} \bigg] = 0.
\end{align*}
Therefore, we have $\mathbb{P}$-almost surely
\begin{align}\label{int-distance}
\int_0^{\tau} \bigg( \int_E | d_{B_2}(r_{s-} + \gamma(r_{s-},x)) |^2 F(dx) \bigg) ds = 0, \quad t \geq 0.
\end{align}
Since the process $r$ is c\`{a}dl\`{a}g with $(r^{\tau})_- \in B_1$ up to an evanescent set and the mapping (\ref{mapping-int-dist}) is continuous, the integrand appearing in (\ref{int-distance}) is continuous in $s=0$. Thus, we deduce that
\begin{align*}
\int_E | d_{B_2}(h + \gamma(h,x)) |^2 F(dx) = 0.
\end{align*}
This provides
\begin{align*}
d_{B_2}(h + \gamma(h,x)) = 0 \quad \text{for $F$-almost all $x \in E$,}
\end{align*}
and hence
\begin{align*}
h + \gamma(h,x) \in \overline{B}_2 \quad \text{for $F$-almost all $x \in E$,}
\end{align*}
completing the proof.
\end{proof}

\begin{lemma}\label{lemma-jump-global-help}
Let $B_1 \subset B_2 \subset H$ be two Borel sets such that
\begin{align}\label{jumps-provided-suff}
h + \gamma(h,x) \in B_2 \quad \text{ for $F$-almost all $x \in E$,} \quad \text{for all $h \in B_1$.}
\end{align}
Let $h_0 : \Omega \rightarrow H$ be a $\mathcal{F}_0$-measurable random variable and let $r = r^{(h_0)}$ be a local mild solution to (\ref{SPDE}) with lifetime $\tau > 0$ such that $(r^{\tau})_- \in B_1$ and $r^{\tau} \mathbbm{1}_{[\![ 0,\tau [\![} \in B_2$ up to an evanescent set. Then we have $r^{\tau} \in B_2$ up to an evanescent set.
\end{lemma}

\begin{proof}
Since $r^{\tau} \mathbbm{1}_{[\![ 0,\tau [\![} \in B_2$ up to an evanescent set, it suffices to prove that
\begin{align}\label{suffices-jumps}
\mathbb{P}(r_{\tau} \mathbbm{1}_{\{ \tau < \infty \}} \in B_2) = 1.
\end{align}
By (\ref{integrate-mu}), \cite[Theorem~II.1.8]{Jacod-Shiryaev} and (\ref{jumps-provided-suff}) we obtain
\begin{align*}
&\mathbb{E} \bigg[ \sum_{n \in \mathbb{N}} \mathbbm{1}_{\{ r_{\tau_n -} + \gamma(r_{\tau_n -}, \xi_{\tau_n}) \notin B_2 \} } \bigg] = \mathbb{E} \bigg[ \int_0^{\infty} \int_E \mathbbm{1}_{\{ r_{s-} + \gamma(r_{s-}, x) \notin B_2 \} } \mu(ds,dx) \bigg]
\\ &= \mathbb{E} \bigg[ \int_0^{\infty} \int_E \mathbbm{1}_{\{ r_{s-} + \gamma(r_{s-}, x) \notin B_2 \} } F(dx)ds \bigg] = 0,
\end{align*}
which yields
\begin{align*}
\mathbb{P} ( r_{\tau_n -} + \gamma(r_{\tau_n -}, \xi_{\tau_n}) \in B_2 \text{ for all } n \in \mathbb{N} ) = 1.
\end{align*}
Therefore, by Lemma~\ref{lemma-jumps-of-SPDE} we obtain $\mathbb{P}$-almost surely
\begin{align*}
r_{\tau} \mathbbm{1}_{\{ \tau < \infty \}} = ( r_{\tau -} + \Delta r_{\tau} ) \mathbbm{1}_{\{ \tau < \infty \}} = \bigg( r_{\tau -} + \gamma(r_{\tau-},\xi_{\tau}) \sum_{n \in \mathbb{N}} \mathbbm{1}_{ \{ \tau_n = \tau \} } \bigg) \mathbbm{1}_{\{ \tau < \infty \}} \in B_2,
\end{align*}
proving (\ref{suffices-jumps}).
\end{proof}

\begin{lemma}
Let $B \subset C \subset H$ be two Borel sets such that $C$ is closed in $H$ and
\begin{align*}
h + \gamma(h,x) \in C \quad \text{ for $F$-almost all $x \in E$,} \quad \text{for all $h \in B$.}
\end{align*}
Let $h_0 : \Omega \rightarrow H$ be a $\mathcal{F}_0$-measurable random variable and let $r = r^{(h_0)}$ be a local mild solution to (\ref{SPDE}) with lifetime $\tau > 0$ such that $(r^{\tau})_- \in B$ up to an evanescent set. Then we have $r^{\tau} \in C$ up to an evanescent set.
\end{lemma}

\begin{proof}
By the closedness of $C$ in $H$, we have $r^{\tau} \mathbbm{1}_{[\![ 0,\tau [\![} \in C$ up to an evanescent set. Thus, the statement follows from Lemma~\ref{lemma-jump-global-help}.
\end{proof}

\begin{lemma}\label{lemma-jumps-exact}
Let $G_1,G_2$ be metric spaces such that $G_1$ is separable. Let
$B \subset G_1$ be a Borel set, let $C \subset G_2$ be a closed set and let $\delta : G_1 \times E \rightarrow G_2$ be a measurable mapping such that $\delta(\bullet,x) : G_1 \rightarrow G_2$ is continuous for all $x \in E$. Suppose that
\begin{align}\label{jumps-for-proof}
\delta(h,x) \in C \quad \text{for $F$-almost all $x \in E$,} \quad \text{for all $h \in B$.}
\end{align}
Then we even have
\begin{align}\label{jumps-provided}
\delta(h,x) \in C \quad \text{for all $h \in B$,} \quad \text{ for $F$-almost all $x \in E$.}
\end{align}
\end{lemma}

\begin{proof}
By separability of $G_1$ there exists a countable set $D$, which is dense in $B$. By (\ref{jumps-for-proof}), for each $h \in D$ there exists a $F$-nullset $N_h$ such that
\begin{align*}
\delta(h,x) \in C \quad \text{for all $x \in N_h^c$.}
\end{align*}
The set $N := \bigcup_{h \in D} N_h$ is also a $F$-nullset. Now, let $h \in B$ be arbitrary. Then there exists a sequence $(h_n)_{n \in \mathbb{N}} \subset D$ with $h_n \rightarrow h$, and we obtain
\begin{align*}
\delta(h_n,x) \in C \quad \text{for all $n \in \mathbb{N}$ and $x \in N^c$.}
\end{align*}
Since $\delta(\bullet,x)$ is continuous for all $x \in E$ and the set $C$ is closed in $G_2$, we deduce
\begin{align*}
\delta(h,x) = \lim_{n \rightarrow \infty} \delta(h_n,x) \in C \quad \text{for all $x \in N^c$,}
\end{align*}
providing (\ref{jumps-provided}).
\end{proof}

Recall that a closed, convex cone $C$ is a nonempty, closed subset $C \subset H$  such that $h+g \in C$ for all $h,g \in C$ and $\lambda h \in C$ for all $\lambda \geq 0$ and $h \in C$.

\begin{lemma}\label{lemma-int-in-convex}
Let $(G, \mathcal{G}, \nu)$ be a $\sigma$-finite measure space, let $C \subset H$ be a closed, convex cone and let $f \in \mathcal{L}^1(G; H)$ be such that $f(x) \in C$ for $\nu$-almost all $x \in G$. Then we have
\begin{align*}
\int_G f d\nu \in C.
\end{align*}
\end{lemma}

\begin{proof}
First, we assume that $f \in \mathcal{L}^1(G; H)$ is a simple function of the form
\begin{align}\label{f-simple}
f = \sum_{k=1}^m c_k \mathbbm{1}_{A_k}
\end{align}
with $c_k \in C$ and $A_k \in \mathcal{G}$ satisfying $\nu(A_k) < \infty$ for $k = 1,\ldots,m$. Then we have
\begin{align*}
\int_G f d\nu = \sum_{k=1}^m c_k \nu(A_k) \in C.
\end{align*}
Now, let $f \in \mathcal{L}^1(G; H)$ be an arbitrary function such that $f(x) \in C$ for $\nu$-almost all $x \in G$. Arguing as in the proof of \cite[Lemma~1.1]{Da_Prato}, there exists a a sequence $(f_n)_{n \in \mathbb{N}}$ of simple functions of the form (\ref{f-simple}) such that $f_n \rightarrow f$ in $\mathcal{L}^1(G;H)$. Therefore, we get
\begin{align*}
\int_G f d\nu = \lim_{n \rightarrow \infty} \int_G f_n d\nu \in C,
\end{align*}
finishing the proof.
\end{proof}

\begin{lemma}\label{lemma-int-mu-pos}
Let $C \subset H$ be a closed, convex cone and let $\delta : \Omega \times \mathbb{R}_+ \times E \rightarrow H$ be an optional process satisfying (\ref{gamma-mu-integrable}) such that
\begin{align*}
\delta(\bullet,x) \in C \quad \text{up to an evanescent set,} \quad \text{for $F$-almost all $x \in E$.}
\end{align*}
Then we have $X \in C$ up to an evanescent set, where $X$ denotes the integral process
\begin{align*}
X_t := \int_0^{t} \int_E \delta(s,x) \mu(ds,dx), \quad t \geq 0.
\end{align*}
\end{lemma}

\begin{proof}
By assumption, there is a $F$-nullset $N$ such that
\begin{align*}
\delta(\bullet,x) \in C \quad \text{up to an evanescent set,} \quad \text{for all $x \in N^c$.}
\end{align*}
Using identity (\ref{integrate-mu}) and \cite[Theorem~II.1.8]{Jacod-Shiryaev} we obtain
\begin{align*}
\mathbb{E} \bigg[ \sum_{n \in \mathbb{N}} \mathbbm{1}_{\{ \xi_{\tau_n} \in N \}} \bigg] &= \mathbb{E} \bigg[ \int_0^{\infty} \int_E \mathbbm{1}_{\{ x \in N \}} \mu(ds,dx) \bigg]
\\ &= \mathbb{E} \bigg[ \int_0^{\infty} \int_E \mathbbm{1}_{\{ x \in N \}} F(dx)ds \bigg] = 0,
\end{align*}
which gives us $\mathbb{P} ( \xi_{\tau_n} \notin N \text{ for all } n \in \mathbb{N} ) = 1$.
Using (\ref{integrate-mu}) we obtain $\mathbb{P}$-almost surely
\begin{align*}
X_t = \sum_{n \in \mathbb{N}} \delta(\tau_n,\xi_{\tau_n}) \mathbbm{1}_{\{ \tau_n \leq t \}} \in C \quad \text{for all $t \geq 0$,}
\end{align*}
finishing the proof.
\end{proof}

\begin{lemma}\label{lemma-Damir}
The following statements are true:
\begin{enumerate}
\item For each $h \in H$ we have
\begin{align}\label{series-D-abs-conv}
\sum_{j \in \mathbb{N}} \| D \sigma^j(h) \sigma^j(h) \| < \infty.
\end{align}

\item The mapping
\begin{align}\label{continuity-Damir}
H \rightarrow H, \quad h \mapsto \sum_{j \in \mathbb{N}} D \sigma^j(h) \sigma^j(h)
\end{align}
is continuous.
\end{enumerate}
\end{lemma}

\begin{proof}
Let $j \in \mathbb{N}$ be arbitrary. Furthermore, let $h \in H$ be arbitrary. There exists $n \in \mathbb{N}$ such that $\| h \| \leq n$. By estimates (\ref{Lipschitz-sigma-st}), (\ref{linear-growth-sigma}) we have
\begin{align}\label{est-d-sigma-sigma}
\| D \sigma^j(h) \sigma^j(h) \| \leq \| D \sigma^j(h) \| \, \| \sigma^j(h) \| \leq (\kappa_n^j)^2.
\end{align}
Since $\sum_{j \in \mathbb{N}} (\kappa_n^j)^2 < \infty$, we have (\ref{series-D-abs-conv}), showing that the first statement holds true. For each $j \in \mathbb{N}$ the mapping
\begin{align*}
H \mapsto H, \quad D \sigma^j(h) \sigma^j(h)
\end{align*}
is continuous, because for all $h_1,h_2 \in H$ we have
\begin{align*}
&\| D \sigma^j(h_1) \sigma^j(h_1) - D \sigma^j(h_2) \sigma^j(h_2) \|
\\ &\leq \| D \sigma^j(h_1) \sigma^j(h_1) - D \sigma^j(h_1) \sigma^j(h_2) \| + \| D \sigma^j(h_1) \sigma^j(h_2) - D \sigma^j(h_2) \sigma^j(h_2) \|
\\ &\leq \| D \sigma^j(h_1) \| \, \| \sigma^j(h_1) - \sigma^j(h_2) \| + \| D \sigma^j(h_1) - D \sigma^j(h_2) \| \, \| \sigma^j(h_2) \|.
\end{align*}
Denoting by $\nu$ the counting measure on $(\mathbb{N},\mathfrak{P}(\mathbb{N}))$ given by $\nu(\{j\}) = 1$ for all $j \in \mathbb{N}$, we can express the mapping (\ref{continuity-Damir}) as
\begin{align*}
\sum_{j \in \mathbb{N}} D \sigma^j(h) \sigma^j(h) = \int_{\mathbb{N}} D \sigma^j(h) \sigma^j(h) \nu(dj).
\end{align*}
Taking into account estimate (\ref{est-d-sigma-sigma}), Lebesgue's dominated convergence theorem yields the continuity of the mapping (\ref{continuity-Damir}).
\end{proof}

\begin{lemma}\label{lemma-gamma-Bc-cont}
Let $B \in \mathcal{E}$ be a set with $F(B^c) < \infty$.
\begin{enumerate}
\item For each $h \in H$ we have
\begin{align}\label{int-Bc-fin-app}
\int_{B^c} \| \gamma(h,x) \| F(dx) < \infty.
\end{align}
\item The mappings $\alpha^{B} : H \rightarrow H$ and $\gamma^{B} : H \times E \rightarrow H$ defined as
\begin{align}\label{def-alpha-B-app}
\alpha^{B}(h) &:= \alpha(h) - \int_{B^c}
\gamma(h,x) F(dx),
\\ \label{def-gamma-B-app} \gamma^{B}(h,x) &:= \gamma(h,x) \mathbbm{1}_{B}(x)
\end{align}
also satisfy the regularity conditions (\ref{Lipschitz-alpha-st}), (\ref{Lipschitz-gamma-rho}), (\ref{linear-growth-rho}).
\end{enumerate}
\end{lemma}

\begin{proof}
Let $h \in H$ be arbitrary. There exists $n \in \mathbb{N}$ with $\| h \| \leq n$. By the Cauchy-Schwarz inequality and (\ref{linear-growth-rho}), (\ref{rho-square-int}) we have
\begin{align*}
\int_{B^c} \| \gamma(h,x) \| F(dx) &\leq F(B^c)^{1/2} \bigg( \int_E \| \gamma(h,x) \|^2 F(dx) \bigg)^{1/2}
\\ &\leq F(B^c)^{1/2} \bigg( \int_E \rho_n(x)^2 F(dx) \bigg)^{1/2} < \infty,
\end{align*}
showing (\ref{int-Bc-fin-app}). Now, let  $n \in \mathbb{N}$ and $h_1,h_2 \in H$ with $\| h_1 \|, \| h_2 \| \leq n$ be arbitrary. By the Cauchy-Schwarz inequality and (\ref{Lipschitz-gamma-rho}) we obtain
\begin{align*}
&\bigg\| \int_{B^c} \gamma(h_1,x) F(dx) - \int_{B^c} \gamma(h_2,x) F(dx)  \bigg\| \leq \int_{B^c} \| \gamma(h_1,x) - \gamma(h_2,x) \| F(dx)
\\ &\leq F(B^c)^{1/2} \bigg( \int_E \| \gamma(h_1,x) - \gamma(h_2,x) \|^2 F(dx) \bigg)^{1/2}
\\ &\leq F(B^c)^{1/2} \bigg( \int_E \rho_n(x)^2 F(dx) \bigg)^{1/2} \| h_1 - h_2 \|,
\end{align*}
which, in view of (\ref{rho-square-int}), proves that $\alpha^B$ also satisfies (\ref{Lipschitz-alpha-st}). Furthermore, the mapping $\gamma^B$ also satisfies (\ref{Lipschitz-gamma-rho}), (\ref{linear-growth-rho}), which directly follows from its Definition (\ref{def-gamma-B-app}).
\end{proof}

\begin{lemma}\label{lemma-Nt}
For every set $B \in \mathcal{E}$ with $F(B^c) < \infty$ the process
\begin{align*}
N_t := \mu([0,t] \times B^c), \quad t \geq 0
\end{align*}
is a c\`{a}dl\`{a}g, adapted process with $N_0 = 0$, $N \in \mathbb{N}_0$ and $\Delta N \in \{ 0,1 \}$ up to an evanescent set, and we have the representation
\begin{align}\label{repr-Nt}
N_t = \sum_{n \in \mathbb{N}} \mathbbm{1}_{\{ \xi_{\tau_n} \notin B \}} \mathbbm{1}_{\{ \tau_n \leq t \}}, \quad t \geq 0.
\end{align}
\end{lemma}

\begin{proof}
We have $N_0 = 0$, because $\mu(\omega; \{ 0 \} \times E) = 0$ for all $\omega \in \Omega$ by the definition of a random measure, see \cite[Definition~II.1.3]{Jacod-Shiryaev}.
By (\ref{integrate-mu}) we have
\begin{align*}
N_t &= \mu([0,t] \times B^c) = \int_0^t \int_E \mathbbm{1}_{\{ x \notin B \}} \mu(ds,dx)
\\ &= \sum_{n \in \mathbb{N}} \mathbbm{1}_{\{ \xi_{\tau_n} \notin B \}} \mathbbm{1}_{\{ \tau_n \leq t \}}, \quad t \geq 0
\end{align*}
which provides the representation (\ref{repr-Nt}) and shows that $N \in \overline{\mathbb{N}}_0$. Since
\begin{align*}
\mathbb{E}[ N_t ] = \mathbb{E}[ \mu([0,t] \times B^c) ] = t F(B^c) < \infty \quad \text{for all $t \geq 0$,}
\end{align*}
we deduce that $\mathbb{P}(N_t < \infty) = 1$ for all $t \geq 0$. Therefore, the representation (\ref{repr-Nt}) shows that the process $N$ is c\`{a}dl\`{a}g, adapted with $N \in \mathbb{N}_0$ up to an evanescent set. Since $\mu(\omega; \{ t \} \times E) \leq 1$ for all $(\omega,t) \in \Omega \times \mathbb{R}_+$ by the definition of an integer-valued random measure, see \cite[Definition~II.1.13]{Jacod-Shiryaev}, we obtain $\Delta N \in \{ 0,1 \}$.
\end{proof}

For any set $B \in \mathcal{E}$ we define the mapping $\varrho^B : \Omega \rightarrow \overline{\mathbb{R}}_+$ as
\begin{align*}
\varrho^{B} := \inf \{ t \geq 0 : \mu([0,t] \times B^c) = 1 \}.
\end{align*}
For the representation (\ref{repr-tau-B}) below we recall that for any stopping time $\tau$ and any set $A \in \mathcal{F}_{\tau}$ the mapping $\tau^A : \Omega \rightarrow \overline{\mathbb{R}}_+$ given by
\begin{align}\label{def-tau-A}
\tau^A(\omega) :=
\begin{cases}
\tau(\omega), & \omega \in A
\\ \infty, & \omega \notin A
\end{cases}
\end{align}
is also a stopping time.

\begin{lemma}\label{lemma-integrate-tau-B}
For every set $B \in \mathcal{E}$ with $F(B^c) < \infty$ the mapping $\varrho^B$ is a strictly positive stopping time and we have the representation
\begin{align}\label{repr-tau-B}
\varrho^B = \min_{n \in \mathbb{N}} \tau_n^{\{ \xi_{\tau_n} \notin B \}}.
\end{align}
\end{lemma}

\begin{proof}
This is a direct consequence of Lemma~\ref{lemma-Nt}.
\end{proof}

We shall now consider the SPDE
\begin{align}\label{SPDE-cutted-app}
\left\{
\begin{array}{rcl}
dr_t^{B} & = & ( A r_t^{B} + \alpha^{B}(r_t^{B}) )dt +
\sigma(r_t^{B})dW_t
\\ && + \int_{E}
\gamma^{B}(r_{t-}^{B},x) (\mu(dt,dx) - F(dx)dt)
\medskip
\\ r_0^{B} & = & h_0,
\end{array}
\right.
\end{align}
where $B \in \mathcal{E}$ is a set with $F(B^c) < \infty$, and the mappings $\alpha^{B} : H \rightarrow H$ and $\gamma^{B} : H \times E \rightarrow H$ are given by (\ref{def-alpha-B-app}), (\ref{def-gamma-B-app}).

\begin{proposition}\label{prop-r-equal-rB}
Let $h_0 : \Omega \rightarrow H$ be a $\mathcal{F}_0$-measurable random variable, let $B \in \mathcal{E}$ be a set with $F(B^c) < \infty$, and let $0 < \tau \leq \varrho^B$ be a stopping time. Then the following statements are true:
\begin{enumerate}
\item If there exists a local mild solution $r$ to (\ref{SPDE}) with lifetime $\tau$, then there also exists a local mild solution $r^B$ to (\ref{SPDE-cutted-app}) with lifetime $\tau$ such that
\begin{align}\label{r-rB-stoch-interval}
r^{\tau} \mathbbm{1}_{[\![ 0,\tau [\![ } = (r^B)^{\tau} \mathbbm{1}_{[\![ 0,\tau [\![ }.
\end{align}

\item If there exists a local mild solution $r^B$ to (\ref{SPDE-cutted-app}) with lifetime $\tau$, then there also exists a local mild solution $r$ to (\ref{SPDE}) with lifetime $\tau$ such that (\ref{r-rB-stoch-interval}) is satisfied.
\end{enumerate}
In particular, in either case we have $(r^{\tau})_- = ((r^B)^{\tau})_-$.
\end{proposition}

\begin{proof}
Let $r$ be a local mild solution to (\ref{SPDE}) with lifetime $\tau$. We define the process $r^B$ by
\begin{align*}
r^B := r - \gamma(r_{\varrho^B-},\xi_{\varrho^B}) \mathbbm{1}_{[\![ \varrho^B,\infty [\![ }.
\end{align*}
Then $r^B$ is c\`{a}dl\`{a}g and adapted, because $\gamma(r_{\varrho^B-}^B,\xi_{\varrho^B})$ is $\mathcal{F}_{\varrho^B}$-measurable, and, since $\tau \leq \varrho^B$, we have
\begin{align*}
(r^B)^{\tau} = r^{\tau} - \gamma(r_{\varrho^B-},\xi_{\varrho^B}) \mathbbm{1}_{\{ \tau = \varrho^B \}} \mathbbm{1}_{[\![ \tau,\infty [\![ } = r^{\tau} - \gamma(r_{\tau-},\xi_{\tau}) \mathbbm{1}_{\{ \tau = \varrho^B \}} \mathbbm{1}_{[\![ \tau,\infty [\![ }.
\end{align*}
Therefore, we have (\ref{r-rB-stoch-interval}), and hence $(r^{\tau})_- = ((r^B)^{\tau})_-$.
Since $r$ is a local mild solution to (\ref{SPDE}) with lifetime $\tau$, we have $\mathbb{P}$-almost surely
\begin{align*}
r_{t \wedge \tau}^B &= S_{t \wedge \tau} h_0 + \int_0^{t \wedge \tau} S_{(t \wedge \tau) - s} \alpha(r_s) ds + \int_0^{t \wedge \tau} S_{(t \wedge \tau) - s} \sigma(r_s) dW_s
\\ &\quad + \int_0^{t \wedge \tau} \int_{E} S_{(t \wedge \tau) - s} \gamma(r_{s-},x) (\mu(ds,dx) - F(dx)ds)
\\ &\quad - \gamma(r_{\tau-},\xi_{\tau}) \mathbbm{1}_{\{ \tau = \varrho^B \}} \mathbbm{1}_{\{ \tau \leq t \}}
\\ &= S_{t \wedge \tau} h_0 + \int_0^{t \wedge \tau} S_{(t \wedge \tau) - s} \alpha(r_s^B) ds + \int_0^{t \wedge \tau} S_{(t \wedge \tau) - s} \sigma(r_s^B) dW_s
\\ &\quad + \int_0^{t \wedge \tau} \int_{E} S_{(t \wedge \tau) - s} \gamma(r_{s-}^B,x) (\mu(ds,dx) - F(dx)ds)
\\ &\quad - \gamma(r_{\tau-}^B,\xi_{\tau}) \mathbbm{1}_{\{ \tau = \varrho^B \}} \mathbbm{1}_{\{ \tau \leq t \}}, \quad t \geq 0.
\end{align*}
Hence, by the Definitions (\ref{def-alpha-B-app}), (\ref{def-gamma-B-app}) of $\alpha^B, \gamma^B$ we get $\mathbb{P}$-almost surely
\begin{align*}
r_{t \wedge \tau}^B &= S_{t \wedge \tau} h_0 + \int_0^{t \wedge \tau} S_{(t \wedge \tau) - s} \alpha^B(r_s^B) ds + \int_0^{t \wedge \tau} S_{(t \wedge \tau) - s} \sigma(r_s^B) dW_s
\\ &\quad + \int_0^{t \wedge \tau} \int_{E} S_{(t \wedge \tau) - s} \gamma^B(r_{s-}^B,x) (\mu(ds,dx) - F(dx)ds)
\\ &\quad - \gamma(r_{\tau-}^B,\xi_{\tau}) \mathbbm{1}_{\{ \tau = \varrho^B \}} \mathbbm{1}_{\{ \tau \leq t \}} + \int_0^{t \wedge \tau} \int_{B^c} S_{(t \wedge \tau) - s} \gamma(r_{s-}^B,x) F(dx)ds
\\ &\quad + \int_0^{t \wedge \tau} \int_{B^c} S_{(t \wedge \tau) - s} \gamma(r_{s-}^B,x) (\mu(ds,dx) - F(dx)ds), \quad t \geq 0.
\end{align*}
By (\ref{integrate-mu}) and the representation (\ref{repr-tau-B}) from Lemma~\ref{lemma-integrate-tau-B} we have $\mathbb{P}$-almost surely
\begin{equation}\label{repr-mu-Bc}
\begin{aligned}
&\int_0^{t \wedge \tau} \int_{B^c} S_{(t \wedge \tau) - s} \gamma(r_{s-}^B,x) \mu(ds,dx)
\\ &= \sum_{n \in \mathbb{N}} S_{(t \wedge \tau) - \tau_n} \gamma(r_{\tau_n -}^B, \xi_{\tau_n}) \mathbbm{1}_{\{ \xi_{\tau_n} \notin B \}} \mathbbm{1}_{\{ \tau_n \leq t \wedge \tau \}}
\\ &= S_{(t \wedge \tau) - \varrho^B} \gamma(r_{\varrho^B -}^B, \xi_{\varrho^B}) \mathbbm{1}_{\{ \varrho^B \leq t \wedge \tau \}} = S_{(t \wedge \tau) - \varrho^B} \gamma(r_{\varrho^B -}^B, \xi_{\varrho^B}) \mathbbm{1}_{\{ \tau = \varrho^B \}} \mathbbm{1}_{\{ \tau \leq t \}}
\\ &= \gamma(r_{\tau-}^B,\xi_{\tau}) \mathbbm{1}_{\{ \tau = \varrho^B \}} \mathbbm{1}_{\{ \tau \leq t \}}, \quad t \geq 0.
\end{aligned}
\end{equation}
Therefore, we obtain $\mathbb{P}$-almost surely
\begin{align*}
r_{t \wedge \tau}^B &= S_{t \wedge \tau} h_0 + \int_0^{t \wedge \tau} S_{(t \wedge \tau) - s} \alpha^B(r_s^B) ds + \int_0^{t \wedge \tau} S_{(t \wedge \tau) - s} \sigma(r_s^B) dW_s
\\ &\quad + \int_0^{t \wedge \tau} \int_{E} S_{(t \wedge \tau) - s} \gamma^B(r_{s-}^B,x) (\mu(ds,dx) - F(dx)ds), \quad t \geq 0
\end{align*}
showing that $r^B$ is a local mild solution to (\ref{SPDE-cutted-app}) with lifetime $\tau$. This proves the first statement.

Now, let $r^B$ be a local mild solution to (\ref{SPDE-cutted-app}) with lifetime $\tau$. We define the process $r$ by
\begin{align*}
r := r^B + \gamma(r_{\varrho^B-}^B,\xi_{\varrho^B}) \mathbbm{1}_{[\![ \varrho^B,\infty [\![ }.
\end{align*}
Then $r$ is c\`{a}dl\`{a}g and adapted, because $\gamma(r_{\varrho^B-}^B,\xi_{\varrho^B})$ is $\mathcal{F}_{\varrho^B}$-measurable, and, since $\tau \leq \varrho^B$, we have
\begin{align*}
r^{\tau} = (r^B)^{\tau} + \gamma(r_{\varrho^B-}^B,\xi_{\varrho^B}) \mathbbm{1}_{\{ \tau = \varrho^B \}} \mathbbm{1}_{[\![ \tau,\infty [\![ } = (r^B)^{\tau} + \gamma(r_{\tau-}^B,\xi_{\tau}) \mathbbm{1}_{\{ \tau = \varrho^B \}} \mathbbm{1}_{[\![ \tau,\infty [\![ }.
\end{align*}
Therefore, we have (\ref{r-rB-stoch-interval}), and hence $(r^{\tau})_- = ((r^B)^{\tau})_-$.
Since $r^B$ is a local mild solution to (\ref{SPDE-cutted-app}) with lifetime $\tau$, we have $\mathbb{P}$-almost surely
\begin{align*}
r_{t \wedge \tau} &= S_{t \wedge \tau} h_0 + \int_0^{t \wedge \tau} S_{(t \wedge \tau) - s} \alpha^B(r_s^B) ds + \int_0^{t \wedge \tau} S_{(t \wedge \tau) - s} \sigma(r_s^B) dW_s
\\ &\quad + \int_0^{t \wedge \tau} \int_{E} S_{(t \wedge \tau) - s} \gamma^B(r_{s-}^B,x) (\mu(ds,dx) - F(dx)ds)
\\ &\quad + \gamma(r_{\tau-}^B,\xi_{\tau}) \mathbbm{1}_{\{ \tau = \varrho^B \}} \mathbbm{1}_{\{ \tau \leq t \}}
\\ &= S_{t \wedge \tau} h_0 + \int_0^{t \wedge \tau} S_{(t \wedge \tau) - s} \alpha^B(r_s) ds + \int_0^{t \wedge \tau} S_{(t \wedge \tau) - s} \sigma(r_s) dW_s
\\ &\quad + \int_0^{t \wedge \tau} \int_{E} S_{(t \wedge \tau) - s} \gamma^B(r_{s-},x) (\mu(ds,dx) - F(dx)ds)
\\ &\quad + \gamma(r_{\tau-},\xi_{\tau}) \mathbbm{1}_{\{ \tau = \varrho^B \}} \mathbbm{1}_{\{ \tau \leq t \}}, \quad t \geq 0.
\end{align*}
Hence, by the Definitions (\ref{def-alpha-B-app}), (\ref{def-gamma-B-app}) of $\alpha^B, \gamma^B$ we get $\mathbb{P}$-almost surely
\begin{align*}
r_{t \wedge \tau} &= S_{t \wedge \tau} h_0 + \int_0^{t \wedge \tau} S_{(t \wedge \tau) - s} \alpha(r_s) ds + \int_0^{t \wedge \tau} S_{(t \wedge \tau) - s} \sigma(r_s) dW_s
\\ &\quad + \int_0^{t \wedge \tau} \int_{B} S_{(t \wedge \tau) - s} \gamma(r_{s-},x) (\mu(ds,dx) - F(dx)ds)
\\ &\quad + \gamma(r_{\tau-},\xi_{\tau}) \mathbbm{1}_{\{ \tau = \varrho^B \}} \mathbbm{1}_{\{ \tau \leq t \}} - \int_0^{t \wedge \tau} \int_{B^c} S_{(t \wedge \tau) - s} \gamma(r_{s-},x) F(dx)ds, \quad t \geq 0.
\end{align*}
Arguing as in (\ref{repr-mu-Bc}), we have $\mathbb{P}$-almost surely
\begin{align*}
&\int_0^{t \wedge \tau} \int_{B^c} S_{(t \wedge \tau) - s} \gamma(r_{s-},x) \mu(ds,dx) = \gamma(r_{\tau-},\xi_{\tau}) \mathbbm{1}_{\{ \tau = \varrho^B \}} \mathbbm{1}_{\{ \tau \leq t \}}, \quad t \geq 0.
\end{align*}
Therefore, we obtain $\mathbb{P}$-almost surely
\begin{align*}
r_{t \wedge \tau} &= S_{t \wedge \tau} h_0 + \int_0^{t \wedge \tau} S_{(t \wedge \tau) - s} \alpha(r_s) ds + \int_0^{t \wedge \tau} S_{(t \wedge \tau) - s} \sigma(r_s) dW_s
\\ &\quad + \int_0^{t \wedge \tau} \int_{E} S_{(t \wedge \tau) - s} \gamma(r_{s-},x) (\mu(ds,dx) - F(dx)ds), \quad t \geq 0
\end{align*}
showing that $r$ is a local mild solution to (\ref{SPDE}) with lifetime $\tau$. This proves the second statement.
\end{proof}

Now, let $G$ be another separable Hilbert space. For any $k \in \mathbb{N}$ we denote by $C_b^k(G;H)$ the linear space consisting of all $f \in C^k(G;H)$ such that $D^i f$ is bounded for all $i=1,\ldots,k$. In particular, for each $f \in C_b^k(G;H)$ the mappings $D^i f$, $i=0,\ldots,k-1$ are Lipschitz continuous. We do not demand that $f$ itself is bounded, as this would exclude continuous linear operators $f \in L(G;H)$.

\begin{definition}\label{def-pull-back}
Let $\alpha : H \rightarrow H$, $\sigma^j : H \rightarrow H$, $j \in \mathbb{N}$ and $\gamma : H \times E \rightarrow H$ be mappings satisfying
\begin{align}\label{sum-pull}
&\sum_{j \in \mathbb{N}} \| \sigma^j(h) \|^2 < \infty, \quad h \in H,
\\ \label{int-pull} &\int_E \| \gamma(h,x) \|^2 F(dx) < \infty, \quad h \in H,
\end{align}
and let $f : G \rightarrow H$ and $g \in C_b^2(H;G)$ be mappings.
We define the mappings $(f,g)_{\lambda}^{\star} \alpha : G \rightarrow G$, $(f,g)_W^{\star}\sigma^j : G \rightarrow G$, $j \in \mathbb{N}$ and $(f,g)_{\mu}^{\star}\gamma : G \times E \rightarrow G$ as
\begin{align}\label{star-1}
((f,g)_{\lambda}^{\star} \alpha)(z) &:= Dg(h) \alpha(h) + \frac{1}{2} \sum_{j \in \mathbb{N}} D^2 g(h)(\sigma^j(h),\sigma^j(h))
\\ \notag &\quad + \int_E \big( g(h + \gamma(h,x)) - g(h)
- Dg(h) \gamma(h,x) \big) F(dx),
\\ \label{star-2} ((f,g)_{W}^{\star} \sigma^j)(z) &:= Dg(h) \sigma^j(h),
\\ \label{star-3} ((f,g)_{\mu}^{\star} \gamma)(z,x) &:= g(h + \gamma(h,x)) - g(h),
\end{align}
where $h = f(z)$.
\end{definition}

\begin{remark}
Note that the mapping $(f,g)_{\lambda}^{\star} \alpha$ is well-defined. Indeed, for any $h \in H$, by (\ref{sum-pull}) we have
\begin{align*}
\sum_{j \in \mathbb{N}} \| D^2 g(h)(\sigma^j(h),\sigma^j(h)) \| \leq \| D^2 g(h) \| \sum_{j \in \mathbb{N}} \| \sigma^j(h) \|^2 < \infty,
\end{align*}
and by (\ref{int-pull}) and Taylor's theorem we have
\begin{align*}
&\int_E \| g(h + \gamma(h,x)) - g(h)
- Dg(h) \gamma(h,x) \| F(dx)
\\ &\leq \frac{1}{2} \| D^2 g \|_{\infty} \int_E \| \gamma(h,x) \|^2 F(dx) < \infty.
\end{align*}
\end{remark}

\begin{lemma}\label{lemma-Lipschitz-pb}
Let $\alpha : H \rightarrow H$, $\sigma^j : H \rightarrow H$, $j \in \mathbb{N}$ and $\gamma : H \times E \rightarrow H$ be mappings satisfying the regularity conditions (\ref{Lipschitz-alpha-st})--(\ref{linear-growth-sigma}) and (\ref{Lipschitz-gamma-rho})--(\ref{sigma-C1}) such that the mappings $\rho_n : E \rightarrow \mathbb{R}_+$, $n \in \mathbb{N}$ appearing in (\ref{Lipschitz-gamma-rho}), (\ref{linear-growth-rho}) even satisfy
\begin{align}\label{rho-integrable}
\int_E \big( \rho_n(x)^2 \vee \rho_n(x)^4 \big) F(dx) < \infty.
\end{align}
Furthermore, let $f \in C_b^1(G;H)$ and $g \in C_b^3(H;G)$ be arbitrary.
Then the following statements are true:
\begin{enumerate}
\item The mappings $(f,g)_{\lambda}^{\star} \alpha$, $((f,g)_W^{\star} \sigma^j)_{j \in \mathbb{N}}$ and $(f,g)_{\mu}^{\star} \gamma$ also fulfill the regularity conditions (\ref{Lipschitz-alpha-st})--(\ref{linear-growth-sigma}) and (\ref{Lipschitz-gamma-rho})--(\ref{sigma-C1}) with the mappings $\rho_n : E \rightarrow \mathbb{R}_+$, $n \in \mathbb{N}$ appearing in (\ref{Lipschitz-gamma-rho}), (\ref{linear-growth-rho}) satisfying (\ref{rho-square-int}).

\item If $g \in L(H;G)$, then the mappings $\rho_n : E \rightarrow \mathbb{R}_+$, $n \in \mathbb{N}$ appearing in (\ref{Lipschitz-gamma-rho}), (\ref{linear-growth-rho}) even satisfy (\ref{rho-integrable}).
\end{enumerate}
\end{lemma}

\begin{proof}
We define the mappings $\hat{a} : H \rightarrow G$, $\hat{b}^j : H \rightarrow G$, $j \in \mathbb{N}$ and $\hat{c} : H \times E \rightarrow G$ as
\begin{align*}
\hat{a}(h) &:= \hat{a}_1(h) + \hat{a}_2(h) + \hat{a}_3(h),
\\ \hat{b}^j(h) &:= Dg(h) \sigma^j(h),
\\ \hat{c}(h,x) &:= g(h + \gamma(h,x)) - g(h),
\end{align*}
where $\hat{a}_1,\hat{a}_2,\hat{a}_3 : H \rightarrow G$ are given by
\begin{align*}
\hat{a}_1(h) &:= Dg(h) \alpha(h),
\\ \hat{a}_2(h) &:= \frac{1}{2} \sum_{j \in \mathbb{N}} D^2 g(h)(\sigma^j(h),\sigma^j(h)),
\\ \hat{a}_3(h) &:= \int_E \big( g(h + \gamma(h,x)) - g(h)
- Dg(h) \gamma(h,x) \big) F(dx).
\end{align*}
Then we have $\hat{b}^j \in C^1(H;G)$ for all $j \in \mathbb{N}$. By Taylor's theorem, we have the representations
\begin{align}\label{repr-hat-a}
\hat{a}_3(h) &= \int_E \int_0^1 (1-t) D^2 g(h + \gamma(h,x))(\gamma(h,x),\gamma(h,x)) dt F(dx), \quad h \in H
\\ \label{repr-hat-c} \hat{c}(h,x) &= \int_0^1 Dg(h + t \gamma(h,x)) \gamma(h,x) dt, \quad (h,x) \in H \times E.
\end{align}
Let $n \in \mathbb{N}$ be arbitrary. Furthermore, let $h \in H$ with $\| h \| \leq n$ be arbitrary.
By (\ref{linear-growth-sigma}), for all $j \in \mathbb{N}$ we have
\begin{align*}
\| \hat{b}^j(h) \| \leq \| Dg \|_{\infty} \| \sigma^j(h) \| \leq \| Dg \|_{\infty} \kappa_n^j,
\end{align*}
and by (\ref{linear-growth-rho}) and the representation (\ref{repr-hat-c}), for all $x \in E$ we have
\begin{align*}
\| \hat{c}(h,x) \| \leq \int_0^1 \| Dg \|_{\infty} \| \gamma(h,x) \| dt \leq \| Dg \|_{\infty} \rho_n(x).
\end{align*}
Now, let $h_1,h_2 \in H$ with $\| h_1 \|,\| h_2 \| \leq n$ be arbitrary. Using estimate (\ref{Lipschitz-alpha-st}), we obtain
\begin{align*}
&\| \hat{a}_1(h_1) - \hat{a}_1(h_2) \| = \| D g(h_1) \alpha(h_1) - D g(h_2) \alpha(h_2) \|
\\ &\leq \| D g(h_1) \alpha(h_1) - D g(h_2) \alpha(h_1) \| + \| D g(h_2) \alpha(h_1) - D g(h_2) \alpha(h_2) \|
\\ &\leq \big( \| D^2 g \|_{\infty} (L_n n + \| \alpha(0) \|) + \| D g \|_{\infty} L_n \big) \| h_1 - h_2 \|.
\end{align*}
Moreover, we have
\begin{align*}
&\| \hat{a}_2(h_1) - \hat{a}_2(h_2) \| \leq \frac{1}{2} \sum_{j \in \mathbb{N}} \| D^2 g(h_1) (\sigma^j(h_1),\sigma^j(h_1)) - D^2 g(h_2) (\sigma^j(h_2),\sigma^j(h_2)) \|
\\ &\leq \frac{1}{2} \| D^2 g(h_1) \| \sum_{j \in \mathbb{N}} \| \sigma^j(h_1) \| \, \| \sigma^j(h_1) - \sigma^j(h_2) \|
\\ &\quad + \frac{1}{2} \| D^2 g(h_1) - D^2 g(h_2) \| \sum_{j \in \mathbb{N}} \| \sigma^j(h_1) \| \, \| \sigma^j(h_2) \|
\\ &\quad + \frac{1}{2} \| D^2 g(h_2) \| \sum_{j \in \mathbb{N}} \| \sigma^j(h_1) - \sigma^j(h_2) \| \, \| \sigma^j(h_2) \|.
\end{align*}
By estimates (\ref{Lipschitz-sigma-st}), (\ref{linear-growth-sigma}) we obtain
\begin{align*}
\| \hat{a}_2(h_1) - \hat{a}_2(h_2) \| \leq \bigg( \| D^2 g \|_{\infty} + \frac{1}{2} \| D^3 g \|_{\infty} \bigg) \bigg( \sum_{j \in \mathbb{N}} (\kappa_n^j)^2 \bigg) \| h_1 - h_2 \|.
\end{align*}
Furthermore, we have
\begin{align*}
&\| \hat{a}_3(h_1) - \hat{a}_3(h_2) \|
\leq \int_E \int_0^1 \| D^2 g(h_1 + t \gamma(h_1,x)) (\gamma(h_1,x),\gamma(h_1,x))
\\ &\quad - D^2 g(h_2 + t \gamma(h_2,x)) (\gamma(h_2,x),\gamma(h_2,x)) \| dt F(dx)
\\ &\leq \int_E \int_0^1 \| D^2 g(h_1 + t \gamma(h_1,x)) \| \, \| \gamma(h_1,x) \| \, \| \gamma(h_1,x) - \gamma(h_2,x) \| dt F(dx)
\\ &\quad + \int_E \int_0^1 \| D^2 g(h_1 + t \gamma(h_1,x)) - D^2 g(h_2 + t \gamma(h_2,x)) \|
\\ &\quad\quad\quad\quad\quad\quad \times \| \gamma(h_1,x) \| \, \| \gamma(h_2,x) \| dt F(dx)
\\ &\quad + \int_E \int_0^1 \| D^2 g(h_2 + t \gamma(h_2,x)) \| \, \| \gamma(h_1,x) - \gamma(h_2,x) \| \, \| \gamma(h_2,x) \| dt F(dx).
\end{align*}
Noting that, by (\ref{Lipschitz-gamma-rho}), for all $(x,t) \in E \times [0,1]$ we have
\begin{equation}\label{arg-x-t}
\begin{aligned}
&\| D^2 g(h_1 + t \gamma(h_1,x)) - D^2 g(h_2 + t \gamma(h_2,x)) \|
\\ &\leq \| D^3 g \|_{\infty} \| h_1 + t \gamma(h_1,x) - h_2 - t \gamma(h_2,x) \|
\\ &\leq \| D^3 g \|_{\infty} ( \| h_1 - h_2 \| + \| \gamma(h_1,x) - \gamma(h_2,x) \| )
\leq \| D^3 g \|_{\infty} (1 + \rho_n(x)) \| h_1 - h_2 \|,
\end{aligned}
\end{equation}
using estimates (\ref{Lipschitz-gamma-rho}), (\ref{linear-growth-rho}) we get
\begin{align*}
&\| \hat{a}_3(h_1) - \hat{a}_3(h_2) \|
\\ &\leq \bigg( 2 \| D^2 g \|_{\infty} \int_E \rho_n(x)^2 F(dx) + \| D^3 g \|_{\infty} \int_E \big( \rho_n(x)^2 + \rho_n(x)^3 \big) F(dx) \bigg) \| h_1 - h_2 \|.
\end{align*}
By estimates (\ref{Lipschitz-sigma-st}), (\ref{linear-growth-sigma}), for all $j \in \mathbb{N}$ we obtain
\begin{align*}
&\| \hat{b}^j(h_1) - \hat{b}^j(h_2) \| = \| D g(h_1) \sigma^j(h_1) - Dg(h_2) \sigma^j(h_2) \|
\\ &\leq \| D g(h_1) - D g(h_2) \| \, \| \sigma^j(h_1) \| + \| D g(h_2) \| \, \| \sigma^j(h_1) - \sigma^j(h_2) \|
\\ &\leq \big( \| D^2 g \|_{\infty} + \| D g \|_{\infty} \big) \kappa_n^j \| h_1 - h_2 \|.
\end{align*}
For all $x \in E$ we obtain
\begin{align*}
&\| \hat{c}(h_1,x) - \hat{c}(h_2,x) \|
\\ &\leq \int_0^1 \| D g(h_1 + t \gamma(h_1,x)) \gamma(h_1,x) - D g(h_2 + t \gamma(h_2,x)) \gamma(h_2,x) \| dt
\\ &\leq \int_0^1 \| Dg(h_1 + t \gamma(h_1,x)) - Dg(h_2 + t \gamma(h_2,x)) \| \, \| \gamma(h_1,x) \| dt
\\ &\quad + \int_0^1 \| Dg(h_2 + t \gamma(h_2,x)) \| \, \| \gamma(h_1,x) - \gamma(h_2,x) \| dt.
\end{align*}
Arguing as in (\ref{arg-x-t}), for all $(x,t) \in E \times [0,1]$ we have
\begin{align*}
&\| D g(h_1 + t \gamma(h_1,x)) - D g(h_2 + t \gamma(h_2,x)) \|
\leq \| D^2 g \|_{\infty} (1 + \rho_n(x)) \| h_1 - h_2 \|.
\end{align*}
Using estimates (\ref{Lipschitz-gamma-rho}), (\ref{linear-growth-rho}), we obtain
\begin{equation}\label{est-D-2}
\begin{aligned}
\| \hat{c}(h_1,x) - \hat{c}(h_2,x) \| \leq \big( \| D^2 g \|_{\infty} ( \rho_n(x) + \rho_n(x)^2 ) + \| Dg \|_{\infty} \rho_n(x) \big) \| h_1 - h_2 \|.
\end{aligned}
\end{equation}
Since $(f,g)_{\lambda}^{\star} \alpha = \hat{a} \circ f$, $(f,g)_{W}^{\star} \sigma^j = \hat{b}^j \circ f$, $j \in \mathbb{N}$ and $((f,g)_{\mu}^{\star} \gamma)(\bullet,x) = \hat{c}(\bullet,x) \circ f$, $x \in E$ as well as $f \in C_b^1(G;H)$, we deduce that conditions (\ref{Lipschitz-alpha-st})--(\ref{linear-growth-sigma}) and (\ref{Lipschitz-gamma-rho})--(\ref{sigma-C1}) are satisfied with the mappings $\rho_n : E \rightarrow \mathbb{R}_+$, $n \in \mathbb{N}$ appearing in (\ref{Lipschitz-gamma-rho}), (\ref{linear-growth-rho}) satisfying (\ref{rho-square-int}), which proves the first statement. 

If $g \in L(H;G)$, then we have $D^2 g \equiv 0$, and hence, estimate (\ref{est-D-2}) shows that the mappings $\rho_n : E \rightarrow \mathbb{R}_+$, $n \in \mathbb{N}$ appearing in (\ref{Lipschitz-gamma-rho}), (\ref{linear-growth-rho}) even satisfy (\ref{rho-integrable}), establishing the second statement.
\end{proof}

The following result is a version of It\^o's formula for jump-diffusions in infinite dimension.

\begin{proposition}\label{prop-ito}
Let $\alpha : \Omega \times \mathbb{R}_+ \rightarrow G$, $\sigma : \Omega \times \mathbb{R}_+
\rightarrow L_2^0(G)$ and $\gamma : \Omega \times
\mathbb{R}_+ \times E \rightarrow G$ be predictable
processes such that for all $t \geq 0$ we have
\begin{align*}
\mathbb{P} \bigg( \int_0^t \bigg( \| \alpha_s \| + \|
\sigma_s \|_{L_2^0(G)}^2 + \int_E \|
\gamma(s,x) \|^2 F(dx) \bigg) ds < \infty \bigg) = 1.
\end{align*}
Furthermore, let $Y_0 : \Omega \rightarrow G$ be a $\mathcal{F}_0$-measurable random variable, let $Y$ be the $G$-valued It\^o process
\begin{align*}
Y_t = Y_0 + \int_0^t \alpha_s ds + \sum_{j \in \mathbb{N}} \int_0^t
\sigma_s^j d\beta_s^j + \int_0^t \int_E \gamma(s,x)
(\mu(ds,dx) - F(dx)ds), \quad t \geq 0
\end{align*}
and let $\phi \in C_b^2(G;H)$ be arbitrary. Then we have $\mathbb{P}$-almost surely
\begin{align*}
&\phi(Y_t) = \phi(Y_0) + \int_0^t \bigg( D \phi(Y_s) \alpha_s + \frac{1}{2} \sum_{j \in \mathbb{N}} D^2
\phi(Y_s)(\sigma_s^j,\sigma_s^j)
\\ &\quad + \int_E
\big( \phi(Y_s + \gamma(s,x)) - \phi(Y_s) - D \phi(Y_s)
\gamma(s,x) \big) F(dx) \bigg) ds
\\ &\quad + \sum_{j \in \mathbb{N}} \int_0^t
D \phi(Y_s) \sigma_s^j d\beta_s^j
\\ &\quad + \int_0^t \int_E \big( \phi(Y_{s-} + \gamma(s,x)) - \phi(Y_{s-}) \big) (\mu(ds,dx) -
F(dx)ds), \quad t \geq 0
\end{align*}
where $\sigma^j := \sqrt{\lambda_j} \sigma e_j$ for each $j \in \mathbb{N}$.
\end{proposition}

\begin{proof}
For the following particular cases, this version of It\^{o}'s formula is known:
\begin{itemize}
\item For $\gamma \equiv 0$ it follows by applying \cite[Theorem~2.9]{Atma-book} to the function $\langle h, \phi(Y) \rangle$ for each $h \in H$.

\item For $\sigma \equiv 0$ it follows from \cite[Theorem~3.6]{MRT}.
\end{itemize}
The general result follows by executing the proofs of the above-mentioned results simultaneously.
\end{proof}

\section{Finite dimensional submanifolds with boundary in Hilbert spaces}\label{app-submanifolds}

In this section, we provide results about
finite dimensional submanifolds with boundary in Hilbert spaces. For more details, we refer to any textbook about manifolds, e.g., \cite{Abraham}, \cite{Lang} or \cite{Warner}.

Let $H$ be a Hilbert space and let $m \in
\mathbb{N}$ be a positive integer. We denote by $ \mathbb{R}^m_+ $
the set of $m$-tuples $ y \in \mathbb{R}^m $ with non-negative first
coordinate $ y_1 \geq 0 $, that is
\begin{align*}
\mathbb{R}_+^m = \mathbb{R}_+ \times \mathbb{R}^{m-1} = \{ y \in
\mathbb{R}^m : y_1 \geq 0 \}.
\end{align*}
We consider the relative topology on $ \mathbb{R}^m_+ $. Let $V$ be an open subset in $\mathbb{R}_+^m$, i.e., there exists an open set $ \tilde{V} \subset
\mathbb{R}^m $ such that $ \tilde{V} \cap \mathbb{R}_+^m = V $. A boundary point of $V$ is by definition
any point $ y \in V $ with vanishing first coordinate $ y_1 = 0 $.
The set of all boundary points of $ V $ is denoted by $ \partial V
$, i.e.
\begin{align*}
\partial V = \{ y \in V : y_1 = 0 \}.
\end{align*}
Let $k \in \mathbb{N}$ be arbitrary.

\begin{definition}\label{def-Ck-map}
A map $ \phi: V \subset \mathbb{R}^m_+ \to H $ is
called a \emph{$ C^k $-map}, if there is an open set $ \tilde{V} \subset
\mathbb{R}^m $ together with a $ C^k $-map $ \tilde{\phi}: \tilde{V}
\to H $ such that $ \tilde{V} \cap \mathbb{R}_+^m = V $ and $
\tilde{\phi}|_{V} = \phi $.
\end{definition}

For a $C^k$-map $ \phi: V \subset \mathbb{R}^m_+ \to H $ and $y \in V$ we define the derivative $D \phi(y) := D \tilde{\phi}(y)$.
Note that this definition does not depend on the choice of $\tilde{\phi}$.

\begin{definition}\label{def-diffeo-general}
A map $ \phi: V \subset \mathbb{R}^m_+ \to W \subset \mathbb{R}^m_+ $ is
called a \emph{$ C^k $-diffeomorphism}, if $\phi$ is bijective and both, $\phi$ and $\phi^{-1}$, are $C^k$-maps.
\end{definition}

The following lemma is a standard result, whence we omit the proof.

\begin{lemma}\label{lemma-diffeo-boundary}
Let $ \phi: V \subset \mathbb{R}^m_+ \to W \subset \mathbb{R}^m_+ $ be a $C^k$-diffeomorphism for some $k \in \mathbb{N}$. Then the following statements are true:
\begin{enumerate}
\item We have $\phi(\partial V) = \partial W$.

\item For each $y \in \partial V$ we have $D\phi(y) \mathbb{R}_+^m \subset \mathbb{R}_+^m$.
\end{enumerate}
\end{lemma}

Hence, boundary points of $ V $ are mapped to boundary points of $ W $ under a $C^k$-diffeomorphism.

\begin{definition}\label{def-submanifold}
Let $\mathcal{M} \subset H$ be a nonempty subset.
\begin{enumerate}
\item $\mathcal{M}$ is a \emph{$m$-dimensional
$C^k$-submanifold with boundary of $H$}, if for all $h \in
\mathcal{M}$ there exist an open neighborhood $U \subset H$ of $h$, an open
set $V \subset \mathbb{R}^m_+$ and a map $\phi \in C^k(V;H)$ such
that
\begin{enumerate}
\item $\phi : V \rightarrow U \cap \mathcal{M}$ is a homeomorphism;
\item $D \phi(y)$ is one to one for all $y \in V$.
\end{enumerate}
The map $\phi$ is called a \emph{parametrization} of $\mathcal{M}$
around $h$.

\item The \emph{boundary} of $ \mathcal{M} $ is defined as the set of all points $ h \in
\mathcal{M} $ such that $ \phi^{-1}(h) \in \partial V$ for some parametrization $ \phi : V \to H $ around $ h $. The set
of all boundary points is denoted by $ \partial \mathcal{M} $ and is
a submanifold without boundary of dimension $ m-1$ of $H$.
Parametrizations of $ \partial \mathcal{M} $ are provided by
restricting parametrizations $ \phi: V \to H $ of $\mathcal{M}$ to
the boundary $ \partial V $.
\end{enumerate}
\end{definition}

Notice that any submanifold is a submanifold with empty boundary. In
what follows, let $\mathcal{M}$ be a $m$-dimensional $C^k$-submanifold with boundary of $H$.

\begin{definition}
Let $h \in \mathcal{M}$ be arbitrary and let $\phi : V \subset \mathbb{R}_+^m \rightarrow U \cap \mathcal{M}$ be a parametrization around $h$.
\begin{enumerate}
\item The \emph{tangent space} to $\mathcal{M}$ at $h$ is the subspace
\begin{align}\label{def-t-space}
T_h \mathcal{M} := D \phi(y) \mathbb{R}^m, \quad y = \phi^{-1}(h) \in V.
\end{align}
\item For $ h \in \partial \mathcal{M} $ we can distinguish a half
space in $ T_h \mathcal{M} $, namely the set of all inward pointing
directions in $\mathcal{M}$, given by
\begin{align}\label{def-t-space-plus}
{(T_h \mathcal{M})}_+ := D \phi(y) \mathbb{R}^m_+, \quad y =
\phi^{-1}(h) \in \partial V.
\end{align}
\end{enumerate}
\end{definition}

\begin{remark}
By \cite[Lemma~6.1.1]{fillnm} and Lemma~\ref{lemma-diffeo-boundary},
the Definitions (\ref{def-t-space}), (\ref{def-t-space-plus}) of the tangent spaces $T_h \mathcal{M}$ and ${(T_h \mathcal{M})}_+$ are independent of the choice of the parametrization.
\end{remark}

Since parametrizations of $ \partial \mathcal{M} $ are provided by
restricting parametrizations $ \phi: V \subset \mathbb{R}_+^m \to U \cap \mathcal{M} $ of $\mathcal{M}$ to
the boundary $ \partial V $, for any $h \in \partial \mathcal{M}$ we have
\begin{align}\label{tangent-partial}
T_h \partial \mathcal{M} = D \phi(y) \partial \mathbb{R}_+^m, \quad y = \phi^{-1}(h) \in \partial V.
\end{align}
In particular, we see that
\begin{equation}\label{tangent-inclusions}
\begin{aligned}
T_h \partial \mathcal{M} &= (T_h \mathcal{M})_+ \cap - (T_h \mathcal{M})_+ \subset (T_h \mathcal{M})_+
\\ &\subset (T_h \mathcal{M})_+ \cup - (T_h \mathcal{M})_+ =  T_h \mathcal{M}, \quad h \in \partial \mathcal{M}.
\end{aligned}
\end{equation}
For a subset $A \subset H$ we define
\begin{align*}
A^{\perp} &:= \{ h \in H : \langle h,g \rangle = 0 \text{ for all $g \in A$} \},
\\ A^+ &:= \{ h \in H : \langle h,g \rangle \geq 0 \text{ for all $g \in A$} \}.
\end{align*}
In order to introduce the inward pointing normal vectors at boundary points of the submanifold $\mathcal{M}$, we require the following auxiliary result. The proof is elementary and therefore omitted.

\begin{lemma}\label{lemma-normal-unit}
For each $h \in \partial \mathcal{M}$ there exists a unique vector
$\eta_h \in (T_h \mathcal{M})_+ \cap (T_h \partial \mathcal{M})^{\perp}$ with $\| \eta_h \| = 1$ such that
\begin{align}\label{tang-decomp}
T_h \mathcal{M} = T_h \partial \mathcal{M} \oplus {\rm span} \{ \eta_h \}.
\end{align}
Moreover, for each $h \in \partial \mathcal{M}$ we have
\begin{align}\label{tang-boundary-normal}
T_h \partial \mathcal{M} &=  T_h \mathcal{M} \cap \{ \eta_h \}^{\perp},
\\ \label{tang-boundary-half} (T_h \mathcal{M})_+ &=  T_h \mathcal{M} \cap \{ \eta_h \}^+.
\end{align}
\end{lemma}

\begin{definition}\label{def-normal-unit}
For each $h \in \partial \mathcal{M}$ we call $\eta_h$ the \emph{inward pointing normal vector} to $\partial \mathcal{M}$ at $h$.
\end{definition}

In the sequel, the vector $e_1 \in \mathbb{R}^m$ denotes the first unit vector $e_1 = (1,0,\ldots,0)$.

\begin{lemma}\label{lemma-repr}
Let $\phi : V \subset \mathbb{R}_+^m \rightarrow U \cap \mathcal{M}$
be a parametrization. Then, for every $h \in U \cap \partial \mathcal{M}$ there exists a unique number $\lambda > 0$ such that
\begin{align}\label{repr-eta}
\langle \eta_h,D\phi(y)v \rangle = \lambda \langle e_1,v
\rangle \quad \text{for all $v \in \mathbb{R}^m$,}
\end{align}
where $y = \phi^{-1}(h)$.
\end{lemma}

\begin{proof}
Let $h \in U \cap \partial \mathcal{M}$ be arbitrary.
We define the continuous linear functional
\begin{align*}
\ell : \mathbb{R}^m \rightarrow \mathbb{R}, \quad \ell(v) := \langle
\eta_h, D\phi(y) v \rangle.
\end{align*}
There is a unique $z \in \mathbb{R}^m$ such that
\begin{align}\label{Frechet-Riesz-Rm}
\ell(v) = \langle z,v \rangle \quad \text{for all $v \in \mathbb{R}^m$.}
\end{align}
In order to complete the proof, we shall show that $z = \lambda e_1$ for some $\lambda > 0$. By identity (\ref{tang-boundary-normal}) from Lemma~\ref{lemma-normal-unit}, for any $v \in \mathbb{R}^m$
we have $\ell(v) = 0$ if and only if $D \phi(y)v \in T_h \partial \mathcal{M}$, which, in view of (\ref{tangent-partial}), means that
$v \in \partial \mathbb{R}_+^m$. This shows $\ker(\ell) =
\partial \mathbb{R}_+^m$, and hence, there exists a unique $\lambda \in \mathbb{R}$ such that $z = \lambda e_1$. Consequently, identity (\ref{repr-eta}) is valid. By (\ref{def-t-space-plus}), (\ref{tangent-partial}) we have $D \phi(y) e_1 \in (T_h \mathcal{M})_+ \setminus T_h \partial \mathcal{M}$, and hence, inserting $v = e_1$ into (\ref{repr-eta}), by (\ref{tang-boundary-normal}), (\ref{tang-boundary-half}) we obtain
\begin{align*}
\lambda = \lambda \langle e_1,e_1 \rangle = \langle \eta_h, D \phi(y) e_1 \rangle > 0,
\end{align*}
finishing the proof.
\end{proof}

In the sequel, for $h_0 \in H$ and $\epsilon > 0$ we denote by $B_{\epsilon}(h_0)$ the open ball
\begin{align*}
B_{\epsilon}(h_0) = \{ h \in H : \| h - h_0 \| < \epsilon \}.
\end{align*}

\begin{lemma}\label{lemma-mf-closed-set}
For each $h_0 \in \mathcal{M}$ there exists $\epsilon_0 > 0$ such that for all $0 < \epsilon \leq \epsilon_0$ the following statements are true:
\begin{enumerate}
\item The set $\overline{B_{\epsilon}(h_0)} \cap \mathcal{M}$ is compact.

\item We have $B_{\epsilon}(h_0) \cap \overline{\mathcal{M}} \subset \overline{B_{\epsilon}(h_0)} \cap \mathcal{M}$.
\end{enumerate}
\end{lemma}

\begin{proof}
Let $h_0 \in \mathcal{M}$ be arbitrary, let $\phi : V \subset \mathbb{R}_+^m \rightarrow U \cap \mathcal{M}$ be a parametrization around $h_0$ and set $y_0 := \phi^{-1}(h_0) \in V$. Since $V$ is open in $\mathbb{R}_+^m$, there exist $X \subset K \subset V$ such that $X$ is open in $\mathbb{R}_+^m$ and $K$ is compact. Since $\phi : V \rightarrow U \cap \mathcal{M}$ is a homeomorphism, $\phi(X)$ is open in $U \cap \mathcal{M}$ and $\phi(K)$ is compact. Therefore, and since $U$ is an open neighborhood of $h_0$, there exists $\epsilon_0 > 0$ such that
\begin{align*}
\overline{B_{\epsilon_0}(h_0)} \subset U \quad \text{and} \quad \overline{B_{\epsilon_0}(h_0)} \cap (U \cap \mathcal{M}) \subset \phi(X).
\end{align*}
Let $0 < \epsilon \leq \epsilon_0$ be arbitrary. Since $\phi(X) \subset \phi(K) \subset U \cap \mathcal{M}$, we have the identity
\begin{align*}
\overline{B_{\epsilon}(h_0)} \cap \mathcal{M} = \overline{B_{\epsilon}(h_0)} \cap \phi(K),
\end{align*}
showing that $\overline{B_{\epsilon}(h_0)} \cap \mathcal{M}$ is closed in $\phi(K)$. Since $\phi(K)$ is compact, we deduce that $\overline{B_{\epsilon}(h_0)} \cap \mathcal{M}$ is compact, establishing the first statement.

For the proof of the second statement, let $h \in B_{\epsilon}(h_0) \cap \overline{\mathcal{M}}$ be arbitrary. Since $h \in \overline{\mathcal{M}}$, there exists a sequence $(h_n)_{n \in \mathbb{N}} \subset \mathcal{M}$ with $h_n \rightarrow h$. Therefore, and since $h \in B_{\epsilon}(h_0)$, there exists an index $n_0 \in \mathbb{N}$ such that $h_n \in B_{\epsilon}(h_0)$ for all $n \geq n_0$. Consequently, we have $h_n \in \overline{B_{\epsilon}(h_0)} \cap \mathcal{M}$ for all $n \geq n_0$. By the closedness of $\overline{B_{\epsilon}(h_0)} \cap \mathcal{M}$ we deduce that $h \in \overline{B_{\epsilon}(h_0)} \cap \mathcal{M}$, completing the proof.
\end{proof}

\begin{proposition}\label{prop-MN}
Let $\mathcal{M}_0 \subset H$ be a $m$-dimensional $C^k$-submanifold with boundary of $H$, let $h_0 \in \mathcal{M}$ be arbitrary and let $D \subset H$ be a dense subset. Then there exist
\begin{itemize}
\item a constant $\epsilon > 0$ such that $\mathcal{M} := B_{\epsilon}(h_0) \cap \mathcal{M}_0$ is a $m$-dimensional $C^k$-submanifold with boundary of $H$,

\item a $m$-dimensional $C^k$-submanifold $\mathcal{N}$ with boundary of $\mathbb{R}^m$,

\item parametrizations $\phi : V \rightarrow \mathcal{M}$ and $\psi : V \rightarrow \mathcal{N}$,

\item and elements $\zeta_1,\ldots,\zeta_m \in D$
such that the mapping $f := \phi \circ \psi^{-1} : \mathcal{N} \rightarrow \mathcal{M}$ has the inverse
\begin{align}\label{conv-para}
f^{-1} : \mathcal{M} \rightarrow \mathcal{N}, \quad f^{-1}(h) = \langle \zeta,h \rangle := (\langle \zeta_1,h \rangle,\ldots,\langle \zeta_m,h \rangle).
\end{align}
\end{itemize}
In other words, the diagram
\begin{align*}
\begin{xy}
  \xymatrix{
      \mathcal{N} \subset \mathbb{R}^m \ar@<2pt>[rr]^f  &     &  \mathcal{M} \subset H \ar@<2pt>[ll]^{\langle \zeta,\bullet \rangle} \\
                             & \ar[ul]_{\psi} V \subset \mathbb{R}_+^m \ar[ru]^{\phi} &
  }
\end{xy}
\end{align*}
commutes. Furthermore, the mappings $\phi$, $\psi$, $\Phi := \phi^{-1}$, $\Psi := \psi^{-1}$ have extensions $\phi \in C_{b}^k(\mathbb{R}^m;H)$, $\psi \in C_{b}^k(\mathbb{R}^m)$, $\Phi \in C_{b}^k(H;\mathbb{R}^m)$, $\Psi \in C_{b}^k(\mathbb{R}^m)$.
\end{proposition}

\begin{proof}
Taking into account \cite[Proposition~6.1.2]{fillnm}, there exist
\begin{itemize}
\item a constant $\epsilon > 0$,

\item a $m$-dimensional $C^k$-submanifold $\tilde{\mathcal{M}}$ of $H$ without boundary,

\item a parametrization $\tilde{\phi} : \tilde{V} \subset \mathbb{R}^m \rightarrow \tilde{\mathcal{M}}$ and such that $\tilde{\phi}(V) = \mathcal{M}$, where $V := \tilde{V} \cap \mathbb{R}_+^m$ and $\mathcal{M} := B_{\epsilon}(h_0) \cap \mathcal{M}_0$,

\item elements $\zeta_1,\ldots,\zeta_m \in D$ and a parametrization $\tilde{f} : \tilde{\mathcal{N}} \subset \mathbb{R}^m \rightarrow \tilde{\mathcal{M}}$ with inverse
\begin{align*}
\tilde{f}^{-1} : \tilde{\mathcal{M}} \rightarrow \tilde{\mathcal{N}}, \quad \tilde{f}^{-1}(h) = \langle \zeta,h \rangle := (\langle \zeta_1,h \rangle,\ldots,\langle \zeta_m,h \rangle).
\end{align*}
\end{itemize}
We set $\phi := \tilde{\phi}|_V$, $\mathcal{N} := \tilde{f}^{-1}(\mathcal{M})$, $f := \tilde{f}|_{\mathcal{N}}$ and $\psi := f^{-1} \circ \phi$. Then $\phi : V \subset \mathbb{R}_+^m \rightarrow \mathcal{M}$ is a parametrization, $\mathcal{N}$ is a $m$-dimensional $C^k$-submanifold with boundary of $\mathbb{R}^m$ and $\psi : V \subset \mathbb{R}_+^m \rightarrow \mathcal{N}$ is a parametrization.

By the inverse mapping theorem, see \cite[Theorem~2.5.2]{Abraham}, the parametrization $\psi$ is a local diffeomorphism. Hence, arguing as in \cite[Remark 6.1.1]{fillnm}, we may assume that the mappings $\phi$, $\psi$, $\Phi := \phi^{-1}$, $\Psi := \psi^{-1}$ (after restricting to smaller neighborhoods, if necessary) have the desired extensions.
\end{proof}

\begin{lemma}\label{lemma-inv-special}
Let $h \in \mathcal{M}$ be arbitrary and let $\phi : V \subset \mathbb{R}_+^m \rightarrow U \cap \mathcal{M}$ be a parametrization around $h$ such that $\Phi := \phi^{-1}$ has an extension $\Phi \in C^k(H;\mathbb{R}^m)$. Then we have
\begin{align*}
D \phi(y)^{-1} w = D \Phi(h) w \quad \text{for all $w \in T_h \mathcal{M}$,}
\end{align*}
where $y = \phi^{-1}(h)$.
\end{lemma}

\begin{proof}
The identity $D \Phi(h) D\phi(y) = D(\Phi \circ \phi)(y) = {\rm Id}|_{\mathbb{R}^m}$ yields the assertion.
\end{proof}

In what follows, let $\mathcal{M}$ be a $m$-dimensional $C^3$-submanifold with boundary of~$H$.

\begin{lemma}\label{lemma-decomp-pre}
Let $\phi : V \subset \mathbb{R}_+^m \rightarrow U \cap \mathcal{M}$ be a parametrization and let $\sigma \in C^1(H)$ be a mapping such that
\begin{align}\label{sigma-tang-decomp}
\sigma(h) \in T_h \mathcal{M}, \quad h \in U \cap \mathcal{M}.
\end{align}
We define the mapping
\begin{align}\label{def-Sigma-decomp}
\theta : V \rightarrow \mathbb{R}^m, \quad \theta(y) := D \phi(y)^{-1} \sigma(h), \quad \text{where $h := \phi(y) \in U \cap \mathcal{M}$.}
\end{align}
\begin{enumerate}
\item For each $h \in U \cap \mathcal{M}$ we have the decomposition
\begin{align}\label{dec-general-1}
D \sigma(h) \sigma(h) = D\phi(y)(D \theta(y) \theta(y)) + D^2 \phi(y)(\theta(y),\theta(y)),
\end{align}
where $y = \phi^{-1}(h) \in V$.

\item If, moreover, we have
\begin{align}\label{sigma-tang-decomp-boundary}
\sigma(h) \in T_h \partial \mathcal{M}, \quad h \in U \cap \partial \mathcal{M},
\end{align}
then for each $h \in U \cap \partial \mathcal{M}$ we have
\begin{align}\label{dec-general-2}
\langle \eta_h,D \sigma(h) \sigma(h) \rangle = \langle \eta_h, D^2 \phi(y)(\theta(y),\theta(y)) \rangle,
\end{align}
where $y = \phi^{-1}(h) \in \partial V$.
\end{enumerate}
\end{lemma}

\begin{proof}
Let $h \in U \cap \mathcal{M}$ be arbitrary and set $y := \phi^{-1}(h) \in V$. There exist $\epsilon > 0$ and $\Lambda \in \{ -1,1 \}$ such that
\begin{align}\label{curve-well}
y + \Lambda t \theta(y) \in V \quad \text{for all $t \in [0,\epsilon)$.}
\end{align}
Consequently, the curve
\begin{align*}
c : [0,\epsilon) \rightarrow U \cap \mathcal{M}, \quad c(t) := \phi(y + \Lambda t \theta(y))
\end{align*}
is well-defined and we have $c \in C^1([0,\epsilon);H)$. Note that
\begin{align*}
c(0) = h \quad \text{and} \quad \frac{d}{dt}c(t) \Big|_{t=0} = \Lambda D \phi(y) \theta(y) = \Lambda \sigma(h)
\end{align*}
by the Definition (\ref{def-Sigma-decomp}) of $\theta$. Therefore, we have
\begin{align*}
\frac{d}{dt} \sigma(c(t)) \Big|_{t=0} = \Lambda D\sigma(h)\sigma(h).
\end{align*}
On the other hand, by (\ref{def-Sigma-decomp}),
\begin{align*}
\frac{d}{dt} \sigma(c(t))\Big|_{t=0} &= \frac{d}{dt} D\phi(y + \Lambda t \theta(y)) \theta(y + \Lambda t \theta(y)) \Big|_{t=0}
\\ &= \Lambda \big( D\phi(y)(D \theta(y) \theta(y)) + D^2 \phi(y)(\theta(y),\theta(y)) \big).
\end{align*}
Combining the latter two identities yields (\ref{dec-general-1}), proving the first statement.

Now, suppose that (\ref{sigma-tang-decomp-boundary}) is satisfied. Then we have
$\theta(y) \in \partial \mathbb{R}_+^m$ for all $y \in \partial V$, and therefore
\begin{align}\label{cond-Sigma-Y-concr}
\langle e_1,\theta(y) \rangle = 0 \quad \text{for all $y \in \partial V$.}
\end{align}
Let $h \in U \cap \partial \mathcal{M}$ be arbitrary and set $y := \phi^{-1}(h) \in \partial V$. There exist $\epsilon > 0$ and $\Lambda \in \{ -1,1 \}$ such that (\ref{curve-well}) is satisfied. Moreover, we have
\begin{align*}
\langle e_1,y + \Lambda t \theta(y) \rangle = \langle e_1,y \rangle + \Lambda t \langle e_1,\theta(y) \rangle = 0 \quad \text{for all $t \in [0,\epsilon)$,}
\end{align*}
which gives us
\begin{align*}
y + \Lambda t \theta(y) \in \partial V \quad \text{for all $t \in [0,\epsilon)$.}
\end{align*}
Consequently, using Lemma~\ref{lemma-repr} and (\ref{cond-Sigma-Y-concr}), for some $\lambda > 0$ we obtain
\begin{align*}
&\langle \eta_h,D\phi(y)(D\theta(y)\theta(y)) \rangle = \lambda \langle e_1,D\theta(y)\theta(y) \rangle
\\ &= \lambda \lim_{t \downarrow 0} \frac{\langle e_1,
\theta(y + \Lambda t \theta(y)) \rangle - \langle e_1,\theta(y) \rangle}{t}
= 0.
\end{align*}
In view of (\ref{dec-general-1}), identity (\ref{dec-general-2}) follows, establishing the second statement.
\end{proof}

Let $\gamma : H \times E \rightarrow H$ be a mapping fulfilling conditions (\ref{Lipschitz-gamma-rho}), (\ref{linear-growth-rho}) with the mappings $\rho_n : E \rightarrow \mathbb{R}_+$, $n \in \mathbb{N}$ satisfying (\ref{rho-square-int}).

\begin{definition}
We introduce the following notions:
\begin{enumerate}
\item Let $h_0 \in \mathcal{M}$ be arbitrary. We say that $\gamma$ satisfies the \emph{$\epsilon$-$\delta$-jump condition in $h_0$}, if there exists $\epsilon_0 > 0$ such that for every $0 < \epsilon \leq \epsilon_0$ the set $\overline{B_{\epsilon}(h_0)} \cap \mathcal{M}$ is compact, and there are $0 < \delta < \epsilon$ and a set $B \in \mathcal{E}$ with $F(B^c) < \infty$ such that
\begin{align}\label{jumps-eps-delta}
h + \gamma(h,x) \in \overline{B_{\epsilon}(h_0)} \cap \mathcal{M} \quad \text{for $F$-almost all $x \in B$,} \quad \text{for all $h \in B_{\delta}(h_0) \cap \mathcal{M}$.}
\end{align}
\item We say that $\gamma$ satisfies the \emph{$\epsilon$-$\delta$-jump condition on $\mathcal{M}$}, if $\gamma$ satisfies the $\epsilon$-$\delta$-jump condition in $h_0$ for each $h_0 \in \mathcal{M}$.
\end{enumerate}
\end{definition}

\begin{lemma}\label{lemma-jumps-eps-delta}
Let $h_0 \in \mathcal{M}$ be such that for some neighborhood $U$ of $h_0$ we have
\begin{align}\label{jumps-on-O}
h + \gamma(h,x) \in \overline{\mathcal{M}} \quad \text{for $F$-almost all $x \in E$,} \quad \text{for all $h \in U \cap \mathcal{M}$.}
\end{align}
Then $\gamma$ satisfies the $\epsilon$-$\delta$-jump condition in $h_0$.
\end{lemma}

\begin{proof}
By Lemma~\ref{lemma-mf-closed-set} there exists $\epsilon_0 > 0$ such that for every $0 < \epsilon \leq \epsilon_0$ the set $\overline{B_{\epsilon}(h_0)} \cap \mathcal{M}$ is compact and we have $B_{\epsilon}(h_0) \cap \overline{\mathcal{M}} \subset \overline{B_{\epsilon}(h_0)} \cap \mathcal{M}$. Let $0 < \epsilon \leq \epsilon_0$ be arbitrary. There exists $0 < \delta < \epsilon / 2$ such that $B_{\delta}(h_0) \subset U$. Moreover, there is $n \in \mathbb{N}$ such that $\| h \| \leq n$ for all $h \in B_{\delta}(h_0) \cap \mathcal{M}$. Setting $B := \{ \rho_n < \delta \} \in \mathcal{E}$, by (\ref{rho-square-int}) and Chebyshev's inequality we obtain
\begin{align*}
F(B^c) \leq \frac{1}{\delta^2} \int_E \rho_n(x)^2 F(dx) < \infty.
\end{align*}
Let $h \in B_{\delta}(h_0) \cap \mathcal{M}$ be arbitrary. By (\ref{linear-growth-rho}) we have
\begin{align*}
\| \gamma(h,x) \| \leq \rho_n(x) < \delta \quad \text{for all $x \in B$.}
\end{align*}
Taking into account (\ref{jumps-on-O}), we deduce
\begin{align*}
h + \gamma(h,x) \in B_{\epsilon}(h_0) \cap \overline{\mathcal{M}} \subset \overline{B_{\epsilon}(h_0)} \cap \mathcal{M} \quad \text{for $F$-almost all $x \in E$,}
\end{align*}
showing that $\gamma$ satisfies the $\epsilon$-$\delta$-jump condition in $h_0$.
\end{proof}

\begin{lemma}\label{lemma-Taylor-para}
Let $h_0 \in \mathcal{M}$ be such that $\gamma$ satisfies the $\epsilon$-$\delta$-jump condition in $h_0$. Let $\phi : V \subset \mathbb{R}_+^m \rightarrow U \cap \mathcal{M}$ be a parametrization around $h_0$ such that $\phi$ and $\Phi := \phi^{-1}$ have extensions $\phi \in C_b^2(\mathbb{R}^m;H)$ and $\Phi \in C_b^1(H; \mathbb{R}^m)$.
Then there exist $\delta > 0$, a set $B \in \mathcal{E}$ with $F(B^c) < \infty$ and a measurable mapping $\rho : E \rightarrow \mathbb{R}_+$ satisfying $\int_E \rho(x)^2 F(dx) < \infty$ such that
\begin{equation}\label{gamma-Taylor}
\begin{aligned}
\| \gamma(h,x) - D\phi(y)(\Phi(h + \gamma(h,x)) - \Phi(h)) \| \leq \rho(x)^2 \quad &\text{for $F$-almost all $x \in B$,}
\\ &\text{for all $h \in B_{\delta}(h_0) \cap \mathcal{M}$,}
\end{aligned}
\end{equation}
where $y = \phi^{-1}(h)$.
\end{lemma}

\begin{proof}
Since $\gamma$ satisfies the $\epsilon$-$\delta$-jump condition in $h_0$, there exist $\delta > 0$ and a set $B \in \mathcal{E}$ with $F(B^c) < \infty$ such that
\begin{align*}
h + \gamma(h,x) \in U \cap \mathcal{M} \quad \text{for $F$-almost all $x \in B$,} \quad \text{for all $h \in B_{\delta}(h_0) \cap \mathcal{M}$.}
\end{align*}
Furthermore, there exists $n \in \mathbb{N}$ such that $\| h \| \leq n$ for all $h \in B_{\delta}(h_0) \cap \mathcal{M}$. Let $h \in B_{\delta}(h_0) \cap \mathcal{M}$ be arbitrary and set $y := \phi^{-1}(h)$. With $M := \| D^2 \phi \|_{\infty}$ and $N := \| D \Phi \|_{\infty}$, by Taylor's theorem and (\ref{linear-growth-rho}), for $F$-almost all $B \in \mathcal{E}$ we obtain
\begin{align*}
&\| \gamma(h,x) - D\phi(y)(\Phi(h + \gamma(h,x)) - \Phi(h)) \|
\\ &\leq \| \phi(\Phi(h + \gamma(h,x))) - \phi(\Phi(h)) - D\phi(y)(\Phi(h+\gamma(h,x)) - \Phi(h)) \|
\\ &\leq \frac{1}{2} M \| \Phi(h + \gamma(h,x)) - \Phi(h) \|^2
\leq \frac{1}{2} MN \| \gamma(h,x) \|^2 \leq \frac{1}{2} MN \rho_n(x)^2,
\end{align*}
proving (\ref{gamma-Taylor}).
\end{proof}

For a closed subspace $K \subset H$ we denote by $\Pi_K : H \rightarrow K$ the orthogonal projection on $K$, that is, for each $h \in H$ the vector $\Pi_K h$ is the unique element from $K$ such that
\begin{align*}
\| \Pi_K h - h \| = \inf_{g \in K} \| g - h \|.
\end{align*}

\begin{lemma}\label{lemma-proj-tang-cont}
Suppose that $\gamma$ satisfies the $\epsilon$-$\delta$-jump condition on $\mathcal{M}$. Then the following statements are true:
\begin{enumerate}
\item For each $h \in \mathcal{M}$ we have
\begin{align}\label{int-pi-fin}
\int_E \| \Pi_{(T_h \mathcal{M})^{\perp}} \gamma(h,x) \| F(dx) < \infty.
\end{align}

\item The mapping
\begin{align}\label{int-pi-cont}
\mathcal{M} \rightarrow H, \quad h \mapsto \int_E \Pi_{(T_h \mathcal{M})^{\perp}} \gamma(h,x) F(dx)
\end{align}
is continuous.
\end{enumerate}
\end{lemma}

\begin{proof}
Let $h_0 \in \mathcal{M}$ be arbitrary. By Proposition~\ref{prop-MN} there exists a parametrization $\phi : V \subset \mathbb{R}_+^m \rightarrow U \cap \mathcal{M}$ around $h_0$ such that $\phi$ and $\Phi := \phi^{-1}$ have extensions $\phi \in C_b^2(\mathbb{R}^m;H)$ and $\Phi \in C_b^1(H; \mathbb{R}^m)$. According to Lemma~\ref{lemma-Taylor-para} there exist $\delta > 0$, a set $B \in \mathcal{E}$ with $F(B^c) < \infty$ and a measurable mapping $\rho : E \rightarrow \mathbb{R}_+$ satisfying $\int_E \rho(x)^2 F(dx) < \infty$ such that (\ref{gamma-Taylor}) is satisfied. Let $h \in B_{\delta}(h_0) \cap \mathcal{M}$ be arbitrary. Then, for $F$-almost all $x \in B$ we obtain
\begin{align*}
&\| \Pi_{(T_h \mathcal{M})^{\perp}} \gamma(h,x) \| = \| \gamma(h,x) - \Pi_{T_h \mathcal{M}}\gamma(h,x) \|
\\ &\leq \| \gamma(h,x) - D\phi(y)(\Phi(h + \gamma(h,x)) - \Phi(h)) \| \leq \rho(x)^2.
\end{align*}
Moreover, by (\ref{Lipschitz-gamma-rho}), for each $x \in E$ the mapping
\begin{align*}
H \rightarrow H, \quad h \mapsto \Pi_{(T_h \mathcal{M})^{\perp}} \gamma(h,x)
\end{align*}
is continuous. Thus, by Lebesgue's dominated convergence theorem and Lemma~\ref{lemma-gamma-Bc-cont} we deduce (\ref{int-pi-fin}) and the continuity of the mapping (\ref{int-pi-cont}).
\end{proof}

\begin{lemma}\label{lemma-eta-int-fin}
Suppose that $\gamma$ satisfies the $\epsilon$-$\delta$-jump condition on $\mathcal{M}$ and
let $\phi : V \subset \mathbb{R}_+^m \rightarrow U \cap \mathcal{M}$ be a parametrization such that $\phi$ and $\Phi := \phi^{-1}$ have extensions $\phi \in C_b^2(\mathbb{R}^m; H)$ and $\Phi \in C_b^1(H; \mathbb{R}^m)$. Then the following statements are equivalent:
\begin{enumerate}
\item We have
\begin{align*}
\int_E |\langle \eta_h,\gamma(h,x) \rangle| F(dx) < \infty, \quad h \in U \cap \partial \mathcal{M}.
\end{align*}

\item We have
\begin{align*}
\int_E |\langle \eta_h,D\phi(y)(\Phi(h + \gamma(h,x)) - \Phi(h)) \rangle| F(dx) < \infty, \quad h \in U \cap \partial \mathcal{M}
\end{align*}
where $y = \phi^{-1}(h)$.
\end{enumerate}
\end{lemma}

\begin{proof}
Let $h \in U \cap \partial \mathcal{M}$ be arbitrary and set $y := \phi^{-1}(h)$. By Lemma~\ref{lemma-Taylor-para} there exists set $B \in \mathcal{E}$ with $F(B^c) < \infty$ such that
\begin{align*}
\int_B | \langle \eta_h, \gamma(h,x) - D\phi(y)(\Phi(h + \gamma(h,x)) - \Phi(h)) \rangle | F(dx) < \infty.
\end{align*}
Setting $M := \| D \phi \|_{\infty}$ and $N := \| D \Phi \|_{\infty}$, by using Lemma~\ref{lemma-gamma-Bc-cont} we have
\begin{align*}
\int_{B^c} |\langle \eta_h,\gamma(h,x) \rangle| F(dx) \leq \| \eta_h \| \int_{B^c} \| \gamma(h,x) \| F(dx) < \infty
\end{align*}
as well as
\begin{align*}
&\int_{B^c} |\langle \eta_h,D\phi(y)(\Phi(h + \gamma(h,x)) - \Phi(h)) \rangle| F(dx)
\\ &\leq \| \eta_h \| MN \int_{B^c} \| \gamma(h,x) \| F(dx) < \infty.
\end{align*}
Therefore, the claimed equivalence follows.
\end{proof}

Let $\beta : H \rightarrow H$ and $\gamma : H \times E \rightarrow H$ be mappings such that conditions (\ref{Lipschitz-gamma-rho}), (\ref{linear-growth-rho}) are fulfilled with the mappings $\rho_n : E \rightarrow \mathbb{R}_+$, $n \in \mathbb{N}$ satisfying (\ref{rho-square-int}). Let $B \in \mathcal{E}$ be a set with $F(B^c) < \infty$ and define the mappings $\beta^B : H \rightarrow H$ and $\gamma^B : H \times E \rightarrow H$ as
\begin{align*}
\beta^{B}(h) &:= \beta(h) - \int_{B^c} \gamma(h,x) F(dx),
\\ \gamma^{B}(h,x) &:= \gamma(h,x) \mathbbm{1}_{B}(x).
\end{align*}
Note that $\beta^B$ is well-defined according to Lemma~\ref{lemma-gamma-Bc-cont}.

\begin{proposition}\label{prop-vf-vf-B}
Suppose that $\gamma$ satisfies the $\epsilon$-$\delta$-jump condition on $\mathcal{M}$. Then the following statements are true:
\begin{enumerate}
\item We have
\begin{align}\label{manifold-bound-B}
&\int_E |\langle \eta_{h}, \gamma(h,x) \rangle| F(dx), \quad h \in \partial \mathcal{M}
\\ \label{beta-cont-B} &\beta(h) - \int_E \Pi_{(T_h \mathcal{M})^{\perp}} \gamma(h,x) F(dx) \in T_h
\mathcal{M}, \quad h \in \mathcal{M}
\\ \label{beta-geq-0-B} &\langle \eta_h,\beta(h) \rangle - \int_{E} \langle \eta_h,\gamma(h,x) \rangle F(dx) \geq 0, \quad h \in \partial \mathcal{M}
\end{align}
if and only if
\begin{align}\label{manifold-bound-C}
&\int_E |\langle \eta_{h}, \gamma^B(h,x) \rangle| F(dx), \quad h \in \partial \mathcal{M}
\\ \label{beta-cont-C}
&\beta^B(h) - \int_E \Pi_{(T_h \mathcal{M})^{\perp}} \gamma^B(h,x) F(dx) \in T_h
\mathcal{M}, \quad h \in \mathcal{M}
\\ \label{beta-geq-0-C} &\langle \eta_h,\beta^B(h) \rangle - \int_{E} \langle \eta_h,\gamma^B(h,x) \rangle F(dx) \geq 0, \quad h \in \partial \mathcal{M}.
\end{align}
\item The mapping in (\ref{beta-cont-B}) is continuous on $\mathcal{M}$ if and only if the mapping in (\ref{beta-cont-C}) is continuous on $\mathcal{M}$.
\end{enumerate}
\end{proposition}

\begin{proof}
This is a consequence of Lemmas~\ref{lemma-vf-vf-1}--\ref{lemma-vf-vf-3} below.
\end{proof}

\begin{lemma}\label{lemma-vf-vf-1}
Conditions (\ref{manifold-bound-B}) and (\ref{manifold-bound-C}) are equivalent.
\end{lemma}

\begin{proof}
Let $h \in \partial \mathcal{M}$ be arbitrary. Then we have
\begin{align*}
\int_E |\langle \eta_{h}, \gamma(h,x) \rangle| F(dx) &= \int_{B^c} |\langle \eta_{h}, \gamma(h,x) \rangle| F(dx) + \int_{B} |\langle \eta_{h}, \gamma(h,x) \rangle| F(dx)
\\ &= \int_{B^c} |\langle \eta_{h}, \gamma(h,x) \rangle| F(dx) + \int_{E} |\langle \eta_{h}, \gamma^B(h,x) \rangle| F(dx).
\end{align*}
Taking into account Lemma~\ref{lemma-gamma-Bc-cont}, the claimed equivalence (\ref{manifold-bound-B}) $\Leftrightarrow$ (\ref{manifold-bound-C}) follows.
\end{proof}

\begin{lemma}\label{lemma-vf-vf-2}
Suppose that $\gamma$ satisfies the $\epsilon$-$\delta$-jump condition on $\mathcal{M}$. Then the following statements are true:
\begin{enumerate}
\item Conditions (\ref{beta-cont-B}) and (\ref{beta-cont-C}) are equivalent.

\item The mapping in (\ref{beta-cont-B}) is continuous on $\mathcal{M}$ if and only if the mapping in (\ref{beta-cont-C}) is continuous on $\mathcal{M}$.
\end{enumerate}
\end{lemma}

\begin{proof}
Let $h \in \mathcal{M}$ be arbitrary. The calculation
\begin{align*}
&\beta^B(h) - \int_{E} \Pi_{(T_h \mathcal{M})^{\perp}} \gamma^B(h,x) F(dx)
\\ &= \beta(h) - \int_{B^c} \gamma(h,x)F(dx) - \int_{B} \Pi_{(T_h \mathcal{M})^{\perp}} \gamma(h,x) F(dx)
\\ &= \beta(h) - \int_{E} \Pi_{(T_h \mathcal{M})^{\perp}} \gamma(h,x) F(dx) - \int_{B^c} \gamma(h,x)F(dx)
\\ &\quad - \int_{B} \Pi_{(T_h \mathcal{M})^{\perp}} \gamma(h,x) F(dx) + \int_{E} \Pi_{(T_h \mathcal{M})^{\perp}} \gamma(h,x) F(dx)
\\ &= \beta(h) - \int_{E} \Pi_{(T_h \mathcal{M})^{\perp}} \gamma(h,x) F(dx) - \Pi_{T_h \mathcal{M}} \int_{B^c} \gamma(h,x) F(dx),
\end{align*}
together with Lemma~\ref{lemma-gamma-Bc-cont}, proves the claimed equivalences.
\end{proof}

\begin{lemma}\label{lemma-vf-vf-3}
Suppose that (\ref{manifold-bound-B}) is satisfied. Then conditions (\ref{beta-geq-0-B}) and (\ref{beta-geq-0-C}) are equivalent.
\end{lemma}

\begin{proof}
According to Lemma~\ref{lemma-vf-vf-1}, condition (\ref{manifold-bound-C}) is satisfied, too. Let $h \in \partial \mathcal{M}$ be arbitrary. Then we have
\begin{align*}
&\langle \eta_h,\beta^B(h) \rangle - \int_{E} \langle \eta_h,\gamma^B(h,x) \rangle F(dx)
\\ &= \Big\langle \eta_h,\beta(h) - \int_{B^c} \gamma(h,x) F(dx) \Big\rangle - \int_{B} \langle \eta_h,\gamma(h,x) \rangle F(dx)
\\ &= \langle \eta_h,\beta(h) \rangle - \int_{E} \langle \eta_h,\gamma(h,x) \rangle F(dx),
\end{align*}
proving the claimed equivalence (\ref{beta-geq-0-B}) $\Leftrightarrow$ (\ref{beta-geq-0-C}).
\end{proof}

Let $G$ be another separable Hilbert space and let $\mathcal{N}$ a $m$-dimensional $C^3$-submanifold with boundary of $G$. We assume there exist parametrizations $\phi : V \subset \mathbb{R}_+^m \rightarrow \mathcal{M}$ and $\psi : V \subset \mathbb{R}_+^m \rightarrow \mathcal{N}$. Defining $f := \phi \circ \psi^{-1} : \mathcal{N} \rightarrow \mathcal{M}$ and $g := \psi \circ \phi^{-1} : \mathcal{M} \rightarrow \mathcal{N}$, the diagram
\begin{align*}
\begin{xy}
  \xymatrix{
      \mathcal{N} \subset G \ar@<2pt>[rr]^f  &     &  \mathcal{M} \subset H \ar@<2pt>[ll]^{g} \\
                             & \ar[ul]_{\psi} V \subset \mathbb{R}_+^m \ar[ru]^{\phi} &
  }
\end{xy}
\end{align*}
commutes. We assume that $\phi$, $\psi$, $\Phi := \phi^{-1}$, $\Psi := \psi^{-1}$ have extensions $\phi \in C_{b}^3(\mathbb{R}^m;H)$, $\psi \in C_{b}^3(\mathbb{R}^m;G)$, $\Phi \in C_{b}^3(H;\mathbb{R}^m)$, $\Psi \in C_{b}^3(G;\mathbb{R}^m)$. Consequently, the mappings $f$, $g$ have extensions $f \in C_{b}^3(G;H)$, $g \in C_{b}^3(H;G)$. Let $O_{\mathcal{M}} \subset C_{\mathcal{M}} \subset \mathcal{M}$ be subsets such that $O_{\mathcal{M}}$ is open in $\mathcal{M}$. We define the subsets $O_{\mathcal{N}} \subset C_{\mathcal{N}} \subset \mathcal{N}$ by $O_{\mathcal{N}} := g(O_{\mathcal{M}})$, $C_{\mathcal{N}} := g(C_{\mathcal{M}})$ and the subsets $O_{V} \subset C_{V} \subset V$ by $O_{V} := \psi^{-1}(O_{\mathcal{N}})$, $C_{V} := \psi^{-1}(C_{\mathcal{N}})$. Since $f : \mathcal{N} \rightarrow \mathcal{M}$ and $\psi : V \rightarrow \mathcal{N}$ are homeomorphisms, $O_{\mathcal{N}}$ is open in $\mathcal{N}$ and $O_V$ is open in $V$.

Let $\beta : O_{\mathcal{M}} \rightarrow H$, $\sigma^j : H \rightarrow H$, $j \in \mathbb{N}$, $\gamma : H \times E \rightarrow H$ and $a : O_{\mathcal{N}} \rightarrow G$, $b^j : G \rightarrow G$, $j \in \mathbb{N}$, $c : G \times E \rightarrow G$ be mappings satisfying the regularity conditions (\ref{Lipschitz-sigma-st}), (\ref{linear-growth-sigma}) and (\ref{Lipschitz-gamma-rho})--(\ref{sigma-C1}). The mappings $f_{\lambda}^{\star} \beta : O_{\mathcal{N}} \rightarrow G$, $f_W^{\star}\sigma^j : O_{\mathcal{N}} \rightarrow G$, $j \in \mathbb{N}$ and $f_{\mu}^{\star} \gamma : O_{\mathcal{N}} \times E \rightarrow G$ are defined as
\begin{align*}
(f_{\lambda}^{\star} \beta)(z) &:= ((f,g)_{\lambda}^{\star} \beta)(z),
\\ (f_W^{\star} \sigma^j)(z) &:= ((f,g)_W^{\star} \sigma^j)(z),
\\ (f_{\mu}^{\star} \gamma)(z,x) &:= ((f,g)_{\mu}^{\star} \gamma)(z,x)
\end{align*}
according to (\ref{star-1})--(\ref{star-3}).
In the sequel, for $z \in \partial \mathcal{N}$ the vector $\xi_z$ denotes the inward pointing normal vector to $\partial \mathcal{N}$ at $z$.

\begin{proposition}\label{prop-cond-imply-cond}
Suppose that
\begin{align}\label{a-f-star-beta}
a(z) &= (f_{\lambda}^{\star} \beta)(z), \quad z \in O_{\mathcal{N}},
\\ \label{b-f-star-sigma} b^j(z) &= (f_W^{\star} \sigma^j)(z), \quad j \in \mathbb{N} \text{ and } z \in O_{\mathcal{N}},
\\ \label{c-f-star-gamma} c(z,x) &= (f_{\mu}^{\star} \gamma)(z,x) \quad \text{for $F$-almost all $x \in E$,} \quad \text{for all $z \in O_{\mathcal{N}}$,}
\end{align}
and that the following conditions are satisfied:
\begin{align}\label{sigma-tangent-local-C-1}
&\sigma^j(h) \in T_h \mathcal{M}, \quad h \in O_{\mathcal{M}}, \quad j \in \mathbb{N},
\\ \label{sigma-tangent-local-C-2} &\sigma^j(h) \in T_h \partial \mathcal{M}, \quad h \in O_{\mathcal{M}} \cap \partial \mathcal{M}, \quad j \in \mathbb{N},
\\ \label{jumps-local-C}
&h + \gamma(h,x) \in C_{\mathcal{M}} \quad \text{for $F$-almost all $x \in E$,} \quad \text{for all $h \in O_{\mathcal{M}}$,}
\\ \label{manifold-boundary-local-C} &\int_{E} |\langle \eta_h, \gamma(h,x) \rangle| F(dx) < \infty, \quad h \in O_{\mathcal{M}} \cap \partial \mathcal{M},
\\ \label{alpha-tangent-local-C} &\beta(h) - \frac{1}{2} \sum_{j \in \mathbb{N}} D\sigma^j(h)\sigma^j(h)
\\ & \notag \quad- \int_E \Pi_{(T_h \mathcal{M})^{\perp}} \gamma(h,x) F(dx) \in T_h
\mathcal{M}, \quad h \in O_{\mathcal{M}},
\\ \label{alpha-tangent-plus-local-C} &\langle \eta_h,\beta(h) \rangle - \frac{1}{2} \sum_{j \in \mathbb{N}} \langle \eta_h, D \sigma^j(h) \sigma^j(h) \rangle
\\ \notag &\quad - \int_{E} \langle \eta_h,\gamma(h,x) \rangle F(dx) \geq 0, \quad h \in O_{\mathcal{M}} \cap \partial
\mathcal{M}.
\end{align}
Then the following conditions also hold true:
\begin{align}\label{sigma-tangent-local-D-1}
&b^j(z) \in T_z \mathcal{N}, \quad z \in O_{\mathcal{N}}, \quad j \in \mathbb{N},
\\ \label{sigma-tangent-local-D-2} &b^j(z) \in T_z \partial \mathcal{N}, \quad z \in O_{\mathcal{N}} \cap \partial \mathcal{N}, \quad j \in \mathbb{N},
\\ \label{jumps-local-D}
&z + c(z,x) \in C_{\mathcal{N}} \quad \text{for $F$-almost all $x \in E$,} \quad \text{for all $z \in O_{\mathcal{N}}$,}
\\ \label{manifold-boundary-local-D} &\int_{E} |\langle \xi_z, c(z,x) \rangle| F(dx) < \infty, \quad z \in O_{\mathcal{N}} \cap \partial \mathcal{N},
\\ \label{alpha-tangent-local-D} &a(z) - \frac{1}{2} \sum_{j \in \mathbb{N}} Db^j(z)b^j(z)
\\ \notag &\quad- \int_E \Pi_{(T_z \mathcal{N})^{\perp}} c(z,x) F(dx) \in T_z
\mathcal{N}, \quad z \in O_{\mathcal{N}},
\\ \label{alpha-tangent-plus-local-D} &\langle \xi_z,a(z) \rangle - \frac{1}{2} \sum_{j \in \mathbb{N}} \langle \xi_z, D b^j(z) b^j(z) \rangle
\\ \notag &\quad - \int_{E} \langle \xi_z,c(z,x) \rangle F(dx) \geq 0, \quad z \in O_{\mathcal{N}} \cap \partial \mathcal{N}.
\end{align}
\end{proposition}

For the proof of Proposition~\ref{prop-cond-imply-cond} we prepare several auxiliary results. Note that, under conditions (\ref{a-f-star-beta})--(\ref{c-f-star-gamma}), for all $z \in O_{\mathcal{N}}$ we have
\begin{align}\label{repr-a}
a(z) &= Dg(h)\beta(h) + \frac{1}{2} \sum_{j \in \mathbb{N}} D^2 g(h)(\sigma^j(h),\sigma^j(h))
\\ \notag &\quad + \int_E \big( g(h+\gamma(h,x)) - g(h) - Dg(h)\gamma(h,x) \big) F(dx),
\\ \label{repr-b-j} b^j(z) &= Dg(h)\sigma^j(h) \quad \text{for all $j \in \mathbb{N}$,}
\\ \label{repr-c} c(z,x) &= g(h+\gamma(h,x)) - g(h) \quad \text{for $F$-almost all $x \in E$,}
\end{align}
where $h = f(z) \in O_{\mathcal{M}}$.

\begin{lemma}\label{lemma-tang-pull-back}
Let $h \in \mathcal{M}$ be arbitrary and set $z :
= g(h) \in \mathcal{N}$.
\begin{enumerate}
\item For each $w \in T_h \mathcal{M}$ we have $Dg(h)w \in T_z \mathcal{N}$.

\item For each $w \in (T_h \mathcal{M})_+$ we have $Dg(h)w \in (T_z \mathcal{N})_+$.

\item For each $w \in T_h \partial \mathcal{M}$ we have $Dg(h)w \in T_z \partial \mathcal{N}$.
\end{enumerate}
\end{lemma}

\begin{proof}
Let $w \in T_h \mathcal{M}$ be arbitrary and set $y := \phi^{-1}(h) \in V$. By Lemma~\ref{lemma-inv-special} we have
\begin{align*}
Dg(h) w = D(\psi \circ \Phi)(h) w = D\psi(y) D \Phi(h) w = D\psi(y) ( D \phi(y)^{-1} w ),
\end{align*}
proving the three assertions.
\end{proof}

\begin{lemma}\label{lemma-imply-0}
Suppose that (\ref{b-f-star-sigma}) is satisfied. Then the following statements are true:
\begin{enumerate}
\item Condition (\ref{sigma-tangent-local-C-1}) implies (\ref{sigma-tangent-local-D-1}).

\item Condition (\ref{sigma-tangent-local-C-2}) implies (\ref{sigma-tangent-local-D-2}).
\end{enumerate}
\end{lemma}

\begin{proof}
This follows from (\ref{repr-b-j}) and Lemma~\ref{lemma-tang-pull-back}.
\end{proof}

\begin{lemma}\label{lemma-imply-1}
Suppose that (\ref{c-f-star-gamma}) is satisfied. Then condition (\ref{jumps-local-C}) implies (\ref{jumps-local-D}).
\end{lemma}

\begin{proof}
Let $z \in O_{\mathcal{N}}$ be arbitrary and set $h := f(z) \in O_{\mathcal{M}}$. Then, by (\ref{repr-c}) and (\ref{jumps-local-C}), for $F$-almost all $x \in E$ we obtain
\begin{align*}
z + c(z,x) &= z + g(h + \gamma(h,x)) - g(h) = g(h + \gamma(h,x)) \in C_{\mathcal{N}},
\end{align*}
showing (\ref{jumps-local-D}).
\end{proof}

\begin{lemma}\label{lemma-eps-delta-apply}
Suppose that
\begin{align*}
h + \gamma(h,x) \in \overline{\mathcal{M}} \quad \text{for $F$-almost all $x \in E$,} \quad \text{for all $h \in O_{\mathcal{M}}$.}
\end{align*}
Then $\gamma$ satisfies the $\epsilon$-$\delta$-jump condition on $O_{\mathcal{M}}$.
\end{lemma}

\begin{proof}
Let $h_0 \in O_{\mathcal{M}}$ be arbitrary. By Lemma~\ref{lemma-jumps-eps-delta}, and since $O_{\mathcal{M}}$ is open in $\mathcal{M}$, there exists $\epsilon_0 > 0$ such that for every $0 < \epsilon \leq \epsilon_0$ there are $0 < \delta < \epsilon$ and a set $B \in \mathcal{E}$ with $F(B^c) < \infty$ such that $\overline{B_{\epsilon}(h_0)} \cap \mathcal{M}$ is compact, $\overline{B_{\epsilon}(h_0)} \cap \mathcal{M} \subset O_{\mathcal{M}}$ and (\ref{jumps-eps-delta}) is satisfied. Noting that
\begin{align*}
\overline{B_{\epsilon}(h_0)} \cap \mathcal{M} = \overline{B_{\epsilon}(h_0)} \cap O_{\mathcal{M}},
\end{align*}
we deduce that
\begin{align*}
h + \gamma(h,x) \in \overline{B_{\epsilon}(h_0)} \cap O_{\mathcal{M}} \quad \text{for $F$-almost all $x \in B$,} \quad \text{for all $h \in B_{\delta}(h_0) \cap O_{\mathcal{M}}$,}
\end{align*}
finishing the proof.
\end{proof}

\begin{corollary}\label{cor-proj-tang-cont}
Suppose that condition (\ref{jumps-local-C}) is satisfied. Then the following statements are true:
\begin{enumerate}
\item For each $h \in O_{\mathcal{M}}$ we have
\begin{align*}
\int_E \| \Pi_{(T_h \mathcal{M})^{\perp}} \gamma(h,x) \| F(dx) < \infty.
\end{align*}

\item The mapping
\begin{align*}
O_{\mathcal{M}} \rightarrow H, \quad h \mapsto \int_E \Pi_{(T_h \mathcal{M})^{\perp}} \gamma(h,x) F(dx)
\end{align*}
is continuous.
\end{enumerate}
\end{corollary}

\begin{proof}
This is a direct consequence of Lemmas~\ref{lemma-eps-delta-apply} and \ref{lemma-proj-tang-cont}.
\end{proof}

\begin{lemma}\label{lemma-inner-prod-two}
For every $h \in \partial \mathcal{M}$ there exists a unique number $\lambda > 0$ such that
\begin{align}\label{inner-prod-id-1}
\langle \xi_z, D\psi(y) v \rangle = \lambda \langle \eta_h, D\phi(y) v \rangle \quad \text{for all $v \in \mathbb{R}^m$,}
\end{align}
where $y = \phi^{-1}(h)$ and $z = \psi(y)$. Moreover, we have
\begin{align}\label{inner-prod-id-2}
\langle \xi_z, Dg(h)w \rangle = \lambda \langle \eta_h, w \rangle \quad \text{for all $w \in T_h \mathcal{M}$.}
\end{align}
\end{lemma}

\begin{proof}
Identity (\ref{inner-prod-id-1}) is a direct consequence of Lemma~\ref{lemma-repr}. Using Lemma~\ref{lemma-inv-special} and (\ref{inner-prod-id-1}), for all $w \in T_h \mathcal{M}$ we obtain
\begin{align*}
\langle \xi_z, Dg(h)w \rangle &= \langle \xi_z, D(\psi \circ \Phi)(h)w \rangle = \langle \xi_z, D\psi(y) D\Phi(h)w \rangle
\\ &= \langle \xi_z, D\psi(y) D\phi^{-1}(y)w \rangle = \lambda \langle \eta_h, w \rangle,
\end{align*}
which proves (\ref{inner-prod-id-2}).
\end{proof}

\begin{lemma}\label{lemma-imply-2}
Suppose that (\ref{c-f-star-gamma}), (\ref{jumps-local-C}) are satisfied. Then condition (\ref{manifold-boundary-local-C}) implies (\ref{manifold-boundary-local-D}).
\end{lemma}

\begin{proof}
According to Lemma~\ref{lemma-imply-1}, condition (\ref{jumps-local-D}) is satisfied, too. Let $z \in O_{\mathcal{N}} \cap \partial \mathcal{N}$ be arbitrary. We set $h := f(z) \in O_{\mathcal{M}} \cap \partial \mathcal{M}$ and $y := \phi^{-1}(h) \in O_V \cap \partial V$.
By Lemma~\ref{lemma-eta-int-fin} we have
\begin{align*}
\int_E |\langle \eta_h,D\phi(y)(\Phi(h + \gamma(h,x)) - \Phi(h)) \rangle| F(dx) < \infty.
\end{align*}
Using (\ref{repr-c}) and Lemma~\ref{lemma-inner-prod-two}, for some $\lambda > 0$ we obtain
\begin{align*}
&\int_E | \langle \xi_z, D\psi(y) (\Psi(z + c(z,x)) - \Psi(z)) \rangle | F(dx)
\\ &= \int_E | \langle \xi_z, D\psi(y) (\Psi(g(h + \gamma(h,x))) - \Psi(z)) \rangle | F(dx)
\\ &= \int_E | \langle \xi_z, D\psi(y) (\Phi(h + \gamma(h,x)) - \Phi(h)) \rangle | F(dx)
\\ &= \lambda \int_E | \langle \eta_h, D\phi(y) (\Phi(h + \gamma(h,x)) - \Phi(h)) \rangle | F(dx) < \infty.
\end{align*}
Applying Lemma~\ref{lemma-eta-int-fin} yields condition (\ref{manifold-boundary-local-D}).
\end{proof}

\begin{lemma}\label{lemma-decomposition}
Suppose that (\ref{b-f-star-sigma}), (\ref{sigma-tangent-local-C-1}) are satisfied and let $j \in \mathbb{N}$ be arbitrary. For each $z \in O_{\mathcal{N}}$ we have the decomposition
\begin{align*}
Db^j(z)b^j(z) = Dg(h) ( D\sigma^j(h)\sigma^j(h) ) + D^2 g(h) (\sigma^j(h),\sigma^j(h)),
\end{align*}
where $h = f(z) \in O_{\mathcal{M}}$.
\end{lemma}

\begin{proof}
According to Lemma~\ref{lemma-imply-0}, condition (\ref{sigma-tangent-local-D-1}) is satisfied, too. Note that, by Lemma~\ref{lemma-inv-special} and (\ref{repr-b-j}), for all $y \in O_V$ we have
\begin{align*}
D \phi(y)^{-1} \sigma^j(h) &= D \Phi(h) \sigma^j(h) = D (\Psi \circ g)(h) \sigma^j(h) = D \Psi(z)Dg(h) \sigma^j(h)
\\ &= D\Psi(z) b^j(z) = D \psi(y)^{-1} b^j(z),
\end{align*}
where $h := \phi(y) \in O_\mathcal{M}$ and $z := \psi(y) \in O_{\mathcal{N}}$. We define the mapping
\begin{align*}
\theta^j : O_V \rightarrow \mathbb{R}^m, \quad \theta^j(y) := D \phi(y)^{-1} \sigma^j(h), \quad \text{where $h := \phi(y)$.}
\end{align*}
Let $z \in O_{\mathcal{N}}$ be arbitrary. We set $h := f(z) \in O_{\mathcal{M}}$ and $y := \phi^{-1}(h) \in O_V$. Using Lemma~\ref{lemma-decomp-pre} we obtain the decompositions
\begin{align}\label{dec-sigma}
D \sigma^j(h) \sigma^j(h) &= D\phi(y)(D \theta^j(y) \theta^j(y)) + D^2 \phi(y)(\theta^j(y),\theta^j(y)),
\\ \label{dec-b} D b^j(z) b^j(z) &= D\psi(y)(D \theta^j(y) \theta^j(y)) + D^2 \psi(y)(\theta^j(y),\theta^j(y)).
\end{align}
Note that we have
\begin{align}\label{decomp-pr-1}
D\psi(y)(D \theta^j(y) \theta^j(y)) = D(g \circ \phi)(y)(D \theta^j(y) \theta^j(y)) = Dg(h) D\phi(y) (D \theta^j(y) \theta^j(y)).
\end{align}
By the second order chain rule we obtain
\begin{equation}\label{decomp-pr-2}
\begin{aligned}
&D^2 \psi(y)(\theta^j(y),\theta^j(y)) = D^2(g \circ \phi)(y)(\theta^j(y),\theta^j(y))
\\ &= D^2 g(h)(D\phi(y) \theta^j(y),D\phi(y) \theta^j(y)) + Dg(h) (D^2 \phi(y) (\theta^j(y),\theta^j(y)))
\\ &= D^2 g(h)(\sigma^j(h),\sigma^j(h)) + Dg(h) (D^2 \phi(y) (\theta^j(y),\theta^j(y))).
\end{aligned}
\end{equation}
Moreover, by (\ref{dec-sigma}) we have
\begin{equation}\label{decomp-pr-3}
\begin{aligned}
&Dg(h) (D^2 \phi(y) (\theta^j(y),\theta^j(y)))
\\ &= Dg(h)(D\sigma^j(h)D\sigma^j(h)) - Dg(h)D\phi(y)(D\theta^j(y)\theta^j(y)).
\end{aligned}
\end{equation}
Inserting (\ref{decomp-pr-1})--(\ref{decomp-pr-3}) into (\ref{dec-b}) we arrive at
\begin{align*}
Db^j(z)b^j(z) &= Dg(h) D\phi(y) (D \theta^j(y) \theta^j(y)) + D^2 g(h)(\sigma^j(h),\sigma^j(h)) \\ &\quad + Dg(h) (D^2 \phi(y) (\theta^j(y),\theta^j(y)))
\\ &= Dg(h) D\phi(y) (D \theta^j(y) \theta^j(y)) + D^2 g(h)(\sigma^j(h),\sigma^j(h))
\\ &\quad + Dg(h)(D\sigma^j(h)\sigma^j(h)) - Dg(h) D\phi(y) (D \theta^j(y) \theta^j(y))
\\ &= Dg(h) ( D\sigma^j(h)\sigma^j(h) ) + D^2 g(h) (\sigma^j(h),\sigma^j(h)),
\end{align*}
completing the proof.
\end{proof}

\begin{lemma}\label{lemma-imply-3}
Suppose that (\ref{a-f-star-beta})--(\ref{c-f-star-gamma}) and (\ref{sigma-tangent-local-C-1}), (\ref{jumps-local-C}) are satisfied. Then condition (\ref{alpha-tangent-local-C}) implies (\ref{alpha-tangent-local-D}).
\end{lemma}

\begin{proof}
According to Lemma~\ref{lemma-imply-1}, condition (\ref{jumps-local-D}) is satisfied, too. Let $z \in O_{\mathcal{N}}$ be arbitrary and set $h := f(z) \in O_{\mathcal{M}}$. By (\ref{repr-a}), (\ref{repr-c}) we obtain
\begin{align*}
&a(z) - \frac{1}{2} \sum_{j \in \mathbb{N}} Db^j(z)b^j(z) - \int_E \Pi_{(T_z \mathcal{N})^{\perp}} c(z,x) F(dx)
\\ &= Dg(h)\beta(h) + \frac{1}{2} \sum_{j \in \mathbb{N}} D^2 g(h)(\sigma^j(h),\sigma^j(h))
\\ &\quad + \int_E \big( g(h+\gamma(h,x)) - g(h) - Dg(h)\gamma(h,x) \big) F(dx)
\\ &\quad - \frac{1}{2} \sum_{j \in \mathbb{N}} Db^j(z)b^j(z)
- \int_E \Pi_{(T_z \mathcal{N})^{\perp}} \big( g(h+\gamma(h,x)) - g(h) \big) F(dx).
\end{align*}
Thus, by Lemma~\ref{lemma-decomposition}, relation (\ref{alpha-tangent-local-C})
and Lemma~\ref{lemma-tang-pull-back} we arrive at
\begin{align*}
&a(z) - \frac{1}{2} \sum_{j \in \mathbb{N}} Db^j(z)b^j(z) - \int_E \Pi_{(T_z \mathcal{N})^{\perp}} c(z,x) F(dx)
\\ &= Dg(h)\beta(h) - \frac{1}{2} \sum_{j \in \mathbb{N}} Dg(h) (D\sigma^j(h) \sigma^j(h))
\\ &\quad - \int_E \big( \Pi_{T_z \mathcal{N}} ( g(h + \gamma(h,x)) - g(h) ) + Dg(h) \gamma(h,x) \big) F(dx)
\\ &= Dg(h) \bigg( \beta(h) - \frac{1}{2} \sum_{j \in \mathbb{N}} D\sigma^j(h)\sigma^j(h) - \int_E \Pi_{(T_h \mathcal{M})^{\perp}} \gamma(h,x) F(dx) \bigg)
\\ &\quad - \int_E \big( \Pi_{T_z \mathcal{N}} ( g(h + \gamma(h,x)) - g(h) ) + Dg(h) \Pi_{T_h \mathcal{M}} \gamma(h,x) \big) F(dx) \in T_z \mathcal{N},
\end{align*}
proving that (\ref{alpha-tangent-local-D}) is fulfilled.
\end{proof}

\begin{proof}[Proof of Proposition~\ref{prop-cond-imply-cond}]
According to Lemmas~\ref{lemma-imply-0}, \ref{lemma-imply-1}, \ref{lemma-imply-2}, \ref{lemma-imply-3}, conditions (\ref{sigma-tangent-local-D-1})--(\ref{alpha-tangent-local-D}) are satisfied. Let $z \in O_{\mathcal{N}}$ be arbitrary and set $h := f(z) \in O_{\mathcal{M}}$. By (\ref{repr-a}), (\ref{repr-c}) and Lemma~\ref{lemma-decomposition} we obtain
\begin{align*}
&\langle \xi_z,a(z) \rangle - \frac{1}{2} \sum_{j \in \mathbb{N}} \langle \xi_z, D b^j(z) b^j(z) \rangle
- \int_{E} \langle \xi_z,c(z,x) \rangle F(dx)
\\ &= \Big\langle \xi_z, Dg(h)\beta(h) + \frac{1}{2} \sum_{j \in \mathbb{N}} D^2 g(h)(\sigma^j(h),\sigma^j(h))
\\ &\quad + \int_E \big( g(h+\gamma(h,x)) - g(h) - Dg(h)\gamma(h,x) \big) F(dx) \Big\rangle
\\ &\quad - \frac{1}{2} \sum_{j \in \mathbb{N}} \langle \xi_z, D b^j(z) b^j(z) \rangle
- \int_{E} \langle \xi_z,g(h+\gamma(h,x)) - g(h) \rangle F(dx)
\\ &= \langle \xi_z,Dg(h)\beta(h) \rangle - \frac{1}{2} \sum_{j \in \mathbb{N}} \langle \xi_z,Dg(h)(D\sigma^j(h)\sigma^j(h)) \rangle
\\ &\quad - \int_E \langle \xi_z,Dg(h)\gamma(h,x) \rangle F(dx).
\end{align*}
Taking into account Lemma~\ref{lemma-inner-prod-two}, by (\ref{alpha-tangent-local-C}), (\ref{alpha-tangent-plus-local-C}) for some $\lambda > 0$ we get
\begin{align*}
&\langle \xi_z,a(z) \rangle - \frac{1}{2} \sum_{j \in \mathbb{N}} \langle \xi_z, D b^j(z) b^j(z) \rangle
- \int_{E} \langle \xi_z,c(z,x) \rangle F(dx)
\\ &= \Big\langle \xi_z, Dg(h) \bigg( \beta(h) - \frac{1}{2} \sum_{j \in \mathbb{N}} D\sigma^j(h)\sigma^j(h) - \int_E \Pi_{(T_h \mathcal{M})^{\perp}} \gamma(h,x) F(dx) \bigg) \Big\rangle
\\ &\quad - \int_E \langle \xi_z,Dg(h) \Pi_{T_h \mathcal{M}} \gamma(h,x) \rangle F(dx)
\\ &= \lambda \Big\langle \eta_h, \beta(h) - \frac{1}{2} \sum_{j \in \mathbb{N}} D\sigma^j(h)\sigma^j(h) - \int_E \Pi_{(T_h \mathcal{M})^{\perp}} \gamma(h,x) F(dx) \Big\rangle
\\ &\quad - \lambda \int_E \langle \eta_h,\Pi_{T_h \mathcal{M}} \gamma(h,x) \rangle F(dx)
\\ &= \lambda \bigg( \langle \eta_h,\beta(h) \rangle - \frac{1}{2} \sum_{j \in \mathbb{N}} \langle \eta_h, D \sigma^j(h) \sigma^j(h) \rangle
- \int_{E} \langle \eta_h,\gamma(h,x) \rangle F(dx) \bigg) \geq 0,
\end{align*}
showing that (\ref{alpha-tangent-plus-local-D}) is satisfied.
\end{proof}

\begin{proposition}\label{prop-umr-general}
Suppose we have (\ref{a-f-star-beta})--(\ref{c-f-star-gamma}) and (\ref{sigma-tangent-local-C-1}), (\ref{jumps-local-C}), (\ref{alpha-tangent-local-C}).
Then the following conditions also hold true:
\begin{align}\label{beta-umg}
\beta(h) &= (g_{\lambda}^{\star} a)(h), \quad h \in O_{\mathcal{M}},
\\ \label{sigma-umg} \sigma^j(h) &= (g_W^{\star} b^j)(h), \quad j \in \mathbb{N} \text{ and } h \in O_{\mathcal{M}},
\\ \label{gamma-umg} \gamma(h,x) &= (g_{\mu}^{\star} c)(h,x) \quad \text{for $F$-almost all $x \in E$}, \quad \text{for all $h \in O_{\mathcal{M}}$.}
\end{align}
\end{proposition}

For the proof of Proposition~\ref{prop-umr-general} we prepare some auxiliary results. Note that for each $h \in O_{\mathcal{M}}$ we have
\begin{align}\label{repr-beta}
(g_{\lambda}^{\star} a)(h) &= Df(z)a(z) + \frac{1}{2} \sum_{j \in \mathbb{N}} D^2 f(z)(b^j(z),b^j(z))
\\ \notag &\quad + \int_E \big( f(z+c(z,x)) - f(z) - Df(z)c(z,x) \big) F(dx),
\\ \label{repr-sigma-j} (g_W^{\star} b^j)(h) &= Df(z)b^j(z) \quad \text{for all $j \in \mathbb{N}$,}
\\ \label{repr-gamma} (g_{\mu}^{\star} c)(h,x) &= f(z+c(z,x)) - f(z) \quad \text{for all $x \in E$,}
\end{align}
where $z = g(h) \in O_{\mathcal{N}}$.

\begin{lemma}\label{lemma-tang-inverse}
Let $h \in \mathcal{M}$ be arbitrary. Then we have
\begin{align*}
D f(z) Dg(h) w = w \quad \text{for all $w \in T_h \mathcal{M}$,}
\end{align*}
where $z = g(h) \in \mathcal{N}$.
\end{lemma}

\begin{proof}
For $h \in \mathcal{M}$ we set $z := g(h) \in \mathcal{N}$ and $y := \phi^{-1}(h) \in V$. By Lemma~\ref{lemma-inv-special}, for all $w \in T_h \mathcal{M}$ we have
\begin{align*}
Df(z) Dg(h) w &= D(\phi \circ \Psi)(z) D(\psi \circ \Phi)(h) w = D \phi(y) D \Psi(z) D \psi(y) D \Phi(h) w
\\ &= D \phi(y) D \psi(y)^{-1} D \psi(y) D \phi(y)^{-1} w = w,
\end{align*}
which proves the claimed identity.
\end{proof}

\begin{lemma}\label{lemma-umr-1}
Conditions (\ref{b-f-star-sigma}), (\ref{sigma-tangent-local-C-1}) imply (\ref{sigma-umg}).
\end{lemma}

\begin{proof}
Let $h \in O_{\mathcal{M}}$ be arbitrary and set $z := g(h) \in O_{\mathcal{N}}$. By (\ref{repr-sigma-j}), (\ref{repr-b-j}), Lemma~\ref{lemma-tang-inverse} and (\ref{sigma-tangent-local-C-1}), for each $j \in \mathbb{N}$ we obtain
\begin{align*}
(g_W^{\star} b^j)(h) = Df(z) b^j(z) = Df(z) Dg(h) \sigma^j(h) = \sigma^j(h),
\end{align*}
proving that (\ref{sigma-umg}) is fulfilled.
\end{proof}

\begin{lemma}\label{lemma-umr-2}
Conditions (\ref{c-f-star-gamma}), (\ref{jumps-local-C}) imply (\ref{gamma-umg}).
\end{lemma}

\begin{proof}
Let $h \in O_{\mathcal{M}}$ be arbitrary and set $z := g(h) \in O_{\mathcal{N}}$. By (\ref{repr-gamma}), (\ref{repr-c}) and (\ref{jumps-local-C}), for $F$-almost all $x \in E$ we obtain
\begin{align*}
(g_{\mu}^{\star} c)(h,x) &= f(z + c(z,x)) - f(z) = f(z + g(h + \gamma(h,x)) - g(h)) - f(z)
\\ &= f(g(h + \gamma(h,x))) - h = \gamma(h,x),
\end{align*}
showing that (\ref{gamma-umg}) is satisfied.
\end{proof}

\begin{proof}[Proof of Proposition~\ref{prop-umr-general}]
By Lemmas~\ref{lemma-umr-1}, \ref{lemma-umr-2}, conditions (\ref{sigma-umg}), (\ref{gamma-umg}) are satisfied. Let $h \in O_{\mathcal{M}}$ be arbitrary and set $z := g(h) \in O_{\mathcal{N}}$. By (\ref{repr-beta}), (\ref{repr-a}) we obtain
\begin{align*}
(g_{\lambda}^{\star} a)(h) &= Df(z) a(z) + \frac{1}{2} \sum_{j \in \mathbb{N}} D^2 f(z)(b^j(z),b^j(z))
\\ &\quad + \int_E \big( f(z + c(z,x)) - f(z) - Df(z)c(z,x) \big) F(dx)
\\ &= Df(z) \bigg( Dg(h)\beta(h) + \frac{1}{2} \sum_{j \in \mathbb{N}} D^2 g(h) (\sigma^j(h),\sigma^j(h))
\\ &\quad + \int_E \big( g(h+\gamma(h,x)) - g(h) - Dg(h)\gamma(h,x) \big) F(dx) \bigg)
\\ &\quad + \frac{1}{2} \sum_{j \in \mathbb{N}} D^2 f(z)(b^j(z),b^j(z))
\\ &\quad + \int_E \big( f(z + c(z,x)) - f(z) - Df(z)c(z,x) \big) F(dx).
\end{align*}
By (\ref{sigma-tangent-local-C-1}) and Lemma~\ref{lemma-tang-pull-back}, condition (\ref{sigma-tangent-local-D-1}) is satisfied, too. Hence, applying Lemma~\ref{lemma-decomposition} two times, by taking into account (\ref{b-f-star-sigma}), (\ref{sigma-tangent-local-C-1}) and (\ref{sigma-umg}), (\ref{sigma-tangent-local-D-1}), and using (\ref{repr-c}) as well as (\ref{gamma-umg}), (\ref{repr-gamma}) we get
\begin{align*}
(g_{\lambda}^{\star} a)(h) &= Df(z) \bigg( Dg(h)\beta(h) + \frac{1}{2} \sum_{j \in \mathbb{N}} \big( Db^j(z)b^j(z) - Dg(h)(D \sigma^j(h) \sigma^j(h)) \big)
\\ &\quad + \int_E \big( g(h+\gamma(h,x)) - g(h) - Dg(h)\gamma(h,x) \big) F(dx) \bigg)
\\ &\quad + \frac{1}{2} \sum_{j \in \mathbb{N}} \big( D\sigma^j(h)\sigma^j(h) - Df(z)(Db^j(z)b^j(z)) \big)
\\ &\quad + \int_E \big( \gamma(h,x) - Df(z)( g(h+\gamma(h,x)) - g(h) ) \big) F(dx)
\\ &= Df(z) Dg(h) \bigg( \beta(h) - \frac{1}{2} \sum_{j \in \mathbb{N}} D\sigma^j(h)\sigma^j(h) \bigg) + \frac{1}{2} \sum_{j \in \mathbb{N}} D\sigma^j(h)\sigma^j(h)
\\ &\quad + \int_E \big( \gamma(h,x) - Df(z) Dg(h) \gamma(h,x) \big) F(dx).
\end{align*}
Using (\ref{alpha-tangent-local-C}), by Lemma~\ref{lemma-tang-inverse} we obtain
\begin{align*}
(g_{\lambda}^{\star} a)(h) &= Df(z) Dg(h) \bigg( \beta(h) - \frac{1}{2} \sum_{j \in \mathbb{N}} D\sigma^j(h)\sigma^j(h) \bigg) + \frac{1}{2} \sum_{j \in \mathbb{N}} D\sigma^j(h)\sigma^j(h)
\\ &\quad + \int_E \big( \Pi_{T_h \mathcal{M}} \gamma(h,x) - Df(z) Dg(h) \Pi_{T_h \mathcal{M}} \gamma(h,x)
\\ &\quad\quad\quad\quad + \Pi_{(T_h \mathcal{M})^{\perp}} \gamma(h,x) - Df(z) Dg(h) \Pi_{(T_h \mathcal{M})^{\perp}} \gamma(h,x) \big) F(dx)
\\ &= Df(z) Dg(h) \bigg( \beta(h) - \frac{1}{2} \sum_{j \in \mathbb{N}} D\sigma^j(h)\sigma^j(h) - \int_E \Pi_{(T_h \mathcal{M})^{\perp}} \gamma(h,x) F(dx) \bigg)
\\ &\quad + \frac{1}{2} \sum_{j \in \mathbb{N}} D\sigma^j(h)\sigma^j(h) + \int_E \Pi_{(T_h \mathcal{M})^{\perp}} \gamma(h,x) F(dx) = \beta(h),
\end{align*}
showing that (\ref{beta-umg}) is fulfilled.
\end{proof}